\documentclass{amsart}
\usepackage[utf8]{inputenc}

\usepackage{geometry}
\usepackage{extarrows}
\usepackage{amsmath}
\usepackage{amsthm}
\usepackage{mathrsfs}
\usepackage{amsfonts}
\usepackage{ amssymb }
\usepackage{amsmath,amscd}
\usepackage{array}
\usepackage{amssymb}
\usepackage[all, cmtip]{xy}
\usepackage{tikz}
\usepackage{enumitem}
\usepackage{rotating}
\usepackage{makecell}
\usepackage{pdflscape}
\usepackage{tensor}
\usepackage{longtable}
\usetikzlibrary{arrows}
\usetikzlibrary{cd}
\usepackage{hyperref}
\usepackage[backend=biber,
style=alphabetic,
isbn=false,
url=false,
sorting=nyt]{biblatex}
\usepackage{bm}
\newtheorem{theorem}{Theorem}[subsection]
\newtheorem{lemma}[theorem]{Lemma}

\newtheorem{proposition}[theorem]{Proposition}
\newtheorem{definition}[theorem]{Definition}
\newtheorem{corollary}[theorem]{Corollary}
\newtheorem{situation}[theorem]{Situation}
\theoremstyle{plain}

\theoremstyle{definition}
\newtheorem{example}[theorem]{Example}
\newtheorem{claim}[theorem]{Claim}
\newtheorem{remark}[theorem]{Remark}
\newtheorem{construction}[theorem]{Construction}

\newcommand{\EE}{\mathbb{E}}

\newcommand{\LL}{\mathbb{L}}

\newcommand{\ZZ}{\mathbb{Z}}

\newcommand{\OO}{\mathscr{O}}

\renewcommand{\L}{\mathsf{L}}
\newcommand{\R}{\mathsf{R}}

\newcommand{\omegabul}{\omega^\bullet}

\newcommand{\lotimes}{\overset{\L}{\otimes}}

\newcommand{\cX}{\mathcal{X}}
\newcommand{\cV}{\mathcal{V}}
\newcommand{\cU}{\mathcal{U}}
\newcommand{\cY}{\mathcal{Y}}
\newcommand{\cP}{\mathcal{P}}
\newcommand{\cQ}{\mathcal{Q}}
\newcommand{\cK}{\mathcal{K}}
\newcommand{\cF}{\mathcal{F}}
\newcommand{\cS}{\mathscr{S}}
\newcommand{\cG}{\mathcal{G}}
\newcommand{\fa}{\mathfrak{a}}

\newcommand{\tr}{tr}
\newcommand{\gray}[1]{{\color{gray}#1}}
\newcommand{\Mod}[1]{\mathrm{Mod}(#1)}

\newcommand{\Sec}{{\mathrm{Sec}}}

\newcommand{\liset}{\operatorname{lis-et}}
\newcommand{\et}{\mathrm{et}}
\newcommand{\zar}{\mathrm{zar}}

\newcommand{\Rhom}{\R\Shom}
\newcommand{\Ghom}{\R\mathrm{Hom}}
\newcommand{\SZ}{\mathfrak{S}(\cZ)}
\newcommand{\SW}{\mathfrak{S}(\cW)}

\newcommand{\cW}{\mathcal{W}}
\newcommand{\cZ}{\mathcal{Z}}
\newcommand{\cC}{\mathcal{C}}

\newcommand{\cD}{\mathcal{D}}

\newcommand{\sC}{\mathscr{C}}
\newcommand{\sD}{\mathscr{D}}
\newcommand{\sS}{\mathscr{S}}
\newcommand{\sT}{\mathscr{T}}
\newcommand{\sO}{\mathscr{O}}

\newcommand{\cM}{\mathcal{M}}
\newcommand{\cN}{\mathcal{N}}

\newcommand{\uR}{{\underline{R}}}

\newcommand{\Def}{\underline{\mathrm{Def}}}
\newcommand{\Defset}{\mathrm{Def}}
\newcommand{\Exal}{\underline{\mathrm{Exal}}}
\newcommand{\Extcat}{\underline{\mathrm{Ext}}}
\newcommand{\Exalset}{\mathrm{Exal}}

\newcommand{\Estack}[2]{\mathrm{Ext}^{0/-1}(#1, #2)}
\newcommand{\Estackbase}[3]{\mathrm{Ext}^{0/-1}_{#3}(#1, #2)}

\newcommand{\piso}{\varphi}
\newcommand{\qc}{\mathrm{qc}}
\newcommand{\Shom}{\mathcal{H}om}

\DeclareMathOperator{\Hom}{Hom}

\DeclareMathOperator{\Spec}{Spec}
\DeclareMathOperator{\Dqc}{D_{qc}}

\DeclareMathOperator{\D}{D}
\DeclareMathOperator{\Qcoh}{QCoh}

\DeclareMathOperator{\Ext}{Ext}

\DeclareMathOperator{\cosk}{cosk}

\DeclareMathOperator{\colim}{colim}

\DeclareMathOperator{\Et}{Et}
\DeclareMathOperator{\Le}{Le}
\DeclareMathOperator{\Es}{Es}

\usepackage{xcolor}

\newcommand{\NOTEoff}{\newcommand{\Commentn}[1]{}}
\newcommand{\note}[1]{\Commentn{#1}}

\newcommand\TODOoff{\newcommand{\Comment}[1]{}}
\newcommand{\todo}[1]{\Comment{#1}}

\begin{document}
\TODOoff
\NOTEoff

\title{The moduli of sections has a canonical obstruction theory}
\author{Rachel Webb}

\address[R. Webb]{Department of Mathematics\\
University of California, Berkeley\\
Berkeley, CA 94720-3840\\
U.S.A.}
\email{rwebb@berkeley.edu}


 \begin{abstract}
We give a detailed proof that locally Noetherian moduli stacks of sections carry canonical obstruction theories. As part of the argument we construct a dualizing sheaf and trace map, in the lisse-\'etale topology, for families of tame twisted curves, when the base stack is locally Noetherian.
\end{abstract}



\maketitle

\setcounter{tocdepth}{1}
\tableofcontents

\section{Introduction}\label{sec:intro}

\subsection{Overview}
Let $\cM$ be a locally Noetherian algebraic stack and let $\cC \rightarrow \cM$ be a family of twisted curves as in  \cite[Def~2.1]{AOV11}. Let $\cZ \rightarrow \cC$ be a morphism of algebraic stacks such that $\cZ \to \cM$ is locally of finite presentation, quasi-separated, and has affine stabilizers. By \cite[Thm~1.3]{HR19} there is an algebraic stack $\Sec_{\cM}({\cZ}/{\cC})$ over $\cM$ whose fiber over a scheme $T \rightarrow \cM$ is
\[
\Sec_{\cM}({\cZ}/{\cC})(T) := \Hom_{\cC}(\cC\times_{\cM} T, \cZ)
\]
where the right hand side is the groupoid of morphisms of stacks over $\cC$. Recall that an obstruction theory for $\Sec_{\cM}({\cZ}/{\cC})$ is a morphism of complexes $\phi: E \rightarrow \LL_{\Sec/\cM}$ in the derived category of $\Sec := \Sec_{\cM}({\cZ}/{\cC})$ whose mapping cone has vanishing cohomology sheaves in degrees $[-1, \infty)$ (see Section \ref{sec:def-ot}). An implication of our main theorem is the following.
\begin{theorem}\label{thm:main-intro}
The stack $\Sec_{\cM}({\cZ}/{\cC})$ carries a canonical obstruction theory.
\end{theorem}
We define the canonical obstruction theory in Section \ref{sec:moduli-of-sections}. Theorem \ref{thm:main-intro} is generalized and stated more precisely as Theorem \ref{thm:main} below. An important feature of the canonical obstruction theory is its functoriality, as explained in \cite[Appendix~A]{CJW}.

When the obstruction theory in Theorem \ref{thm:main-intro} is perfect and $\Sec:= \Sec_{\cM}(\cZ/\cC)$ is Deligne-Mumford, quasi-separated, and locally finite type over a field, the machinery in \cite{BF97} and \cite[Sec~5.2]{Kre99} defines a virtual fundamental class on $\Sec$. This is a key part of the construction of Gromov-Witten theory and related enumerative theories: see for example \cite{Be97, AGV, CCK15, CL12}. On the other hand, Theorem \ref{thm:main-intro} is used with a non-Deligne-Mumford instance of $\Sec$ to functorially compare different obstruction theories on quasimap moduli spaces in \cite[Lem~A.2.5]{CJW}. This comparison is crucial for the application of \cite{CJW} to quasimap theory and also for the computations of quasimap $I$-functions in \cite{nonab1, nonab2}.

\subsection{Discussion of Theorem \ref{thm:main-intro}}
\label{sec:discussion}

 The usual argument supporting  Theorem \ref{thm:main-intro} when $\Sec$ is Deligne-Mumford is as follows (this is used, for example, in \cite[Prop~6.2]{BF97}). 
 First reduce to showing that for each affine $f:T \rightarrow \Sec$ and square-zero quasi-coherent ideal sheaf $I$ on $T$, the induced map
\begin{equation}\label{eq:ext}
\Ext^i(\L f^*\LL_{\Sec/\cM}, I) \rightarrow \Ext^i(\L f^*E, I)
\end{equation}
has the following properties (see \cite[Thm~4.5]{BF97}):
\begin{align}
\bullet &\text{ When $i=1$, \eqref{eq:ext} is injective on obstructions.}\label{eq:req}\\
\bullet &\text{ When $i=0$ and there exists a deformation of $f$ by $I$, \eqref{eq:ext} is an isomorphism.}\notag
\end{align} Second, use standard deformation theory to relate the groups $\Ext^i(\L f^*\LL_{\Sec/\cM}, I)$ (resp. $\Ext^i(\L f^*E, I)$) to deformations of the morphism $f:T \rightarrow \Sec$ (resp. $\cC \times_{\cM} T \rightarrow \cZ$). Since morphisms $T \rightarrow \Sec$ are equivalent to morphisms $\cC \times_{\cM} T \rightarrow \cZ$ by definition of $\Sec$, one concludes that \eqref{eq:ext} is an injection (on obstructions) when $i=1$ and that \textit{the groups in \eqref{eq:ext} are isomorphic} (if the obstruction vanishes) when $i=0$. We note that this falls just shy of the second requirement in \eqref{eq:req}, since it is not clear that the morphism in \eqref{eq:ext} is itself an isomorphism.

In this paper, we copy the first step above in Lemma \ref{lem:BF}. However, in the second step, we analyze the functoriality of the isomorphism of Picard categories
\begin{equation}\label{eq:ft-intro}
\Exal_{\cY}(\cX, I) \simeq \Estack{\LL_{\cX/\cY}}{I[1])}
\end{equation}
due to Illusie and Olsson (\cite{illusie, olsson-deformation}), for $\cX \rightarrow \cY$ a representable morphism of algebraic stacks. (See Section \ref{sec:picard-stacks} for the notation and Theorem \ref{thm:ft} for the precise statement). Our proof shows that when $i=0$, not only are the groups in \eqref{eq:ext} isomorphic (in the case of vanishing obstruction), but in fact \textit{the morphism in \eqref{eq:ext} is an isomorphism}, completing the proof of the second requirement in \eqref{eq:req}. Our proof also covers the case when $\Sec\rightarrow \cM$ is not representable or even relatively Deligne-Mumford. 


The correct approach to Theorem \ref{thm:main-intro} is likely through derived algebraic geometry, as in \cite[Sec~2.2]{STV}. The functoriality properties of the obstruction theory proved in \cite[Appendix~A]{CJW} would be natural consequences of such a construction. Unfortunately this author is not equipped to produce the argument. Though the statement of Theorem \ref{thm:main-intro} is certainly familiar, we note that there does not seem to be a reference in the literature for the generality in which we have stated it here.

\subsection{Duality for twisted curves}
A key ingredient for the construction of the obstruction theory in Theorem \ref{thm:main-intro} is the following (stated more precisely as Proposition \ref{prop:duality} below). 
\begin{theorem}\label{thm:duality-intro}
For every family $p: \cC \rightarrow \cM$ of tame twisted curves on a locally Noetherian algebraic stack $\cM$, there is a functorial pair $(\omega_{\cM}, \tr_{\cM})$ with $\omega_{\cM}$ a quasi-coherent sheaf on $\cC$ and $\tr_{\cM}: \R\pi_*\omega_{\cM} \rightarrow \OO_{\cM}[-1]$. 
When $\cM$ is a quasi-separated Noetherian algebraic space, the pair $(\omega_\cM,\tr_\cM)$ agrees with the right adjoint to $\R p_*$.
\end{theorem}
 
We restate the last sentence of the theorem more precisely: if $p: \cC \to \cM$ is a family of twisted curves, a right adjoint $p^!$ to $\R p_*$ exists by \cite[Thm~4.14(1)]{HR17} (see also Lemma \ref{lem:f!-exists} below). The last sentence of Theorem \ref{thm:duality-intro} says that if $\cM$ is a quasi-separated Noetherian algebraic space, we have that $\omega_{\cM}[1] = p^!\OO_{\cM}$ and $tr_\cM$ is the counit of the $(\R p_*, p^!)$ adjunction.  

The reason we do not have this agreement for arbitrary locally Noetherian $\cM$ is that it seems difficult to show that $p^!$ is compatible with arbitrary base change. Following the exposition of \cite{lipman} for schemes, we prove base change for the right adjoint to pushforward for certain morphisms of algebraic stacks in Lemma \ref{lem:f!-basechange} below (see also \cite{Nee17} for a complementary result). However, in applications, one would like to have base change for $p^!$ for families of curves over arbitrary morphisms of algebraic stacks.

The base change problem arises for non-twisted prestable curves as well, and \cite[Tag~0E5W]{tag} addresses the issue by ``gluing'' the pairs $(\omega_{\cM}, \tr_{\cM})$ to get a functorial construction of a dualizing complex and trace map. We adopt the same strategy to prove Theorem \ref{thm:duality-intro}. Again, while the statement of Theorem \ref{thm:duality-intro} is well-known, we do not know a reference for twisted curves, even over the complex numbers.






\subsection{Contents of the paper}
The main goal of Section \ref{sec:formal} is to derive a certain commuting diagram (Lemma \ref{lem:strange}) which will be used in our proof of Theorem \ref{thm:main-intro}. Because the notation is simpler and because we can reuse various parts of the argument in other parts of the paper, we work in the setting of abstract closed symmetric monoidal categories. 

In Section \ref{sec:duality} we prove Theorem \ref{thm:duality-intro}. The proof requires us to construct a special kind of hypercover of an algebraic stack and an associated lisse-\'etale topos. This is an application of the general results proved in \cite{tag} and we explain the details in Appendix \ref{app:descent}.

We explain \eqref{eq:ft-intro} and prove Theorem \ref{thm:main-intro} in Section \ref{sec:main-proof}. The proof itself is fairly short, granting the existence of the dualizing complex and the functoriality of \eqref{eq:ft-intro}. We reserve our proof of the functoriality of \eqref{eq:ft-intro} for Appendix \ref{sec:ft-functoriality}. Since the functoriality is critical to our argument we include the details, but said details are unsurprising.



\subsection{The locally Noetherian hypothesis}
We expect that the locally Noetherian assumption on $\cM$ can be relaxed. It is used only in the proof of Lemma \ref{lem:pushing-ps-per}, to show that pushing forward to the coarse moduli space of a twisted curve preserves pseudo-coherent objects. See Remark \ref{rmk:lazy-me}.

\subsection{Conventions and notation}\label{sec:notation}
We collect some conventions and recurring notation. Our list of notation here is not exhaustive.\\

\noindent
\textit{Algebraic stacks.}
We follow the conventions in \cite[Tag~0260]{tag}; in particular, an algebraic stack need not be quasicompact or quasi-separated.\\

\noindent
\textit{Twisted curves.}
A morphism $p: \cC \rightarrow \cM$ of algebraic stacks is a \textit{family of twisted curves} if smooth-locally on $\cM$ it is a twisted curve in the sense of \cite[Def~2.1]{AOV11}. \\


\renewcommand{\arraystretch}{2}
\setlength{\tabcolsep}{9pt}
\begin{center}
\begin{longtable}{ c|m{7.2cm}|c } 
\multicolumn{3}{c}{\textit{Notation for closed categories and internal hom}} \\ \hline
 Notation & \makecell{Category} & Internal hom \\ \hline \hline 
 $\Mod{A}$ & category of $A$-modules for a sheaf of rings $A$ on a site $\sS$& $\Shom_A$\\ 
 $\D(A)$ & unbounded derived category of $\Mod{A}$ & $\Rhom_A$ \\ [-.2cm]
 & \textit{derived global hom functor (valued in the derived category of $\Gamma(\sS, A)$-modules)}. &\makecell{\textit{notated}\\ $\Ghom_{A}$ }\\
 \makecell{$\cX_{\liset}$\\ (resp. $\cX_{\et}$)}& category of sheaves on the lisse-\'etale (rep. \'etale) site of an algebraic (resp. Deligne-Mumford) stack $\cX$&not needed\\
  \makecell{$\Qcoh(\cX_{\liset})$\\ (resp. $\Qcoh(\cX_{\et})$)}& category of quasi-coherent sheaves on the lisse-\'etale (rep. \'etale) site of an algebraic (resp. Deligne-Mumford) stack $\cX$&not needed\\
 \makecell{$\D(\cX_{\liset})$ \\(resp. $\D(\cX_{\et})$)} &unbounded derived category of $\OO_{\cX}$-modules in $\cX_{\liset}$ (resp. $\cX_{\et}$) & $\Rhom_{\OO_{\cX}}$\\[-.4cm]
 & \textit{derived global hom functor (valued in the derived category of $\Gamma(\cX, \OO_{\cX})$-modules)}. &\makecell{\textit{notated}\\ $\Ghom_{\sO_{\cX}}$ }\\
 \makecell{$\Dqc(\cX_{\liset})$\\ (resp. $\Dqc(\cX_{\et})$)} & full subcategory of $\D(\cX_{\liset})$ (resp. $\D(\cX_{\et})$) on objects with quasi-coherent cohomology sheaves&$\Rhom^{\qc}_{\OO_{\cX}}$\\
\end{longtable}
\end{center}

\begin{center}
\begin{longtable}{c|m{10cm}}
\multicolumn{2}{c}{\textit{Operations on sheaves on topoi and algebraic stacks}} \\ \hline
Notation & \makecell{Meaning} \\\hline\hline
$H^i(\cF)$ & The $i^{th}$ cohomology sheaf of a complex $\cF$\\
$\R \Gamma(\cF)$& The derived global sections functor applied to a complex $\cF$, commonly notated $\R\Gamma(X, \cF)$ where $X$ is a topos \\
$(f^{-1}, f_*)$ & Adjoint functors defined by a morphism of topoi $f: \sC \rightarrow \sD$\\
$f^* $ & For $f:(\sC, \OO_{\sC}) \rightarrow (\sD, \OO_{\sD})$ a morphism of ringed topoi set $f^*(-):=f^{-1}(-) \otimes_{f^{-1}\OO_{\sD}}\OO_{\sC}$\\
$\L f^*$ & For $f: \cX \rightarrow \cY$ a morphism of algebraic stacks, we denote by $\L f^*: \Dqc(\cY_{\liset}) \rightarrow \Dqc(\cX_{\liset})$ the functor $\L f^*_{\qc}$ in \cite[Sec~1.3]{HR17}.\\
$\R f_*$ & For $f:(\sC, \OO_{\sC}) \rightarrow (\sD, \OO_{\sD})$ a morphism of ringed topoi, this is the usual direct image functor $\D(\OO_{\sC}) \to \D(\OO_{\sD})$.\\
$\R f_*$ & For $f: \cX \rightarrow \cY$ a \textit{concentrated} morphism of algebraic stacks, this is the functor $\R (f_{\qc})_*: \Dqc(\cX_{\liset}) \rightarrow \Dqc(\cY_{\liset})$ of \cite[Sec~1.3]{HR17} that is right adjoint to $\L f^*$. By \cite[Thm~2.6(2)]{HR17} it agrees with the restriction of the usual direct image functor $\R (f_{\liset})_*: \D(\cX_{\liset}) \rightarrow \D(\cY_{\liset})$.
\end{longtable}
\end{center}

\subsection{Acknowledgements}
I am grateful to Bhargav Bhatt for explaining the characterization of an obstruction theory in Lemma \ref{lem:BF}. I am also thankful for many helpful conversations with Martin Olsson. This project was partially supported by an NSF Postdoctoral Research Fellowship, award number 200213.
\section{A formal framework for dualizing objects and trace maps}
\label{sec:formal}
\subsection{Closed symmetric monoidal categories}\label{sec:csmcats}
Because notation is simpler in an abstract setting, we work for a moment with closed symmetric monoidal categories. If $\sC$ is such a category, we will write $\OO_{\sC}$ for the unit, $\otimes$ for the product, and $\Hom$ for internal hom, using $\sC(X, Y)$ to denote the set of morphisms between two objects $X, Y \in \sC$ and $1_X$ to denote the identity morphism on $X$. We will suppress mention of the associativity, commutativity, and identity isomorphisms that are part of the definition of $\sC$.\note{instead remembering that if the diagrams (C1)-(C4) of Kelly-MacLane, 'coherence in closed categories" commute, then any diagram of assoc, commutativity, and unit  commutes by Kelly, "on maclane's conditions for coherence"} 
If $\sC$ and $\sD$ are any two categories and $R: \sC \rightarrow \sD$ is a functor with a left adjoint $L$, then for $X \in \sD$ and $Y \in \sC$ we will denote the unit and counit of the adjunction by
\[
\eta^L_X: X \rightarrow RL(X) \quad \quad \quad \epsilon^L_Y: LR(Y) \rightarrow Y
\]
omitting the decorations on $\eta$ and $\epsilon$ when there is no risk of confusion.

We will use many specific instances of the following abstract situation.
\begin{situation}\label{basic-situation}
We are given $\sC$, $\sD$ be closed symmetric monoidal categories with $f^*: \sD \rightarrow \sC$ strong monoidal and $f_*$ a right adjoint.\footnote{\label{ft:lipman}
Moreover, $f_*$ is lax monoidal by \cite[(3.2)]{FHM} and so our setup is consistent with that used in \cite[Sec~3.5]{lipman}.} This means we have natural isomorphisms 
\begin{gather}\label{eq:strong-monoidal}
f^*(X) \otimes f^*(Y) \rightarrow f^*(X\otimes Y)\\
\OO_{\sC} \rightarrow f^*\OO_{\sD}. \label{eq:strong-monoidal-unit}
\end{gather}
\end{situation}
When we are in Situation \ref{basic-situation} we have the following three morphisms at our disposal. The first we recall from \cite[(3.4)]{FHM}: given $Y\in \sC$ and $X \in \sD$ there is a functorial isomorphism
\begin{equation}\label{eq:simplicial1}
\Hom(X, f_*Y) \xrightarrow{\sim} f_*(\Hom(f^*(X), Y)).
\end{equation}
The second is the composition
\begin{equation}
\label{eq:sheafy-push}
f_*\Hom(X, Y) \rightarrow f_*\Hom(f^*f_*X, Y) \xleftarrow[\sim]{\text{\eqref{eq:simplicial1}}} \Hom(f_*X, f_*Y)
\end{equation}
where the first morphism is induced by the counit of the adjunction. The third is 
\begin{equation}\label{eq:sheafy-pullback}
\Hom(X, Y) \rightarrow \Hom(X, f_*f^*Y) \xrightarrow[\sim]{\text{\eqref{eq:simplicial1}}}f_*\Hom(f^*X, f^*Y)
\end{equation}
where the first morphism is induced by the unit of the adjunction; it is an isomorphism if $f^*$ is fully faithful. One can check that \eqref{eq:sheafy-pullback} is functorial in $X$, $Y$, and in the adjoint pair $(f^*, f_*)$ (see \cite[Exercise~3.7.1.1]{lipman}).

We present Example \ref{ex:modules} as the first instance of Situation \ref{basic-situation}; more instances can be found in Examples \ref{ex:general}, \ref{ex:change-rings}, \ref{ex:change-topoi}, and \ref{ex:change-simplicial}.
\begin{example}\label{ex:modules}
The following is an example of Situation \ref{basic-situation}. Let $B' \rightarrow B$ be a homomorphism of rings, and set $\sC = \Mod{B'}$ and $\sD = \Mod{B}$. Define $f^*$ to be the extension of scalars functor $-\otimes_{B'} B$ and define $f_*$ to be restriction of scalars $(-)_{B'}$. The functor $-\otimes_{B'} B$ is strong symmetric monoidal. One can check from the definition in \cite[(3.4)]{FHM} that \eqref{eq:simplicial1} sends $X \rightarrow (Y)_{B'}$ to its adjoint arrow $X\otimes_{B'} B \rightarrow Y$, that \eqref{eq:sheafy-push} sends a $B$-module homomorphism $g: X \rightarrow Y$ to $(g)_{B'}: (X)_{B'} \rightarrow (Y)_{B'}$, and that \eqref{eq:sheafy-pullback} sends a $B'$-module homomorphism $h: X \rightarrow Y$ to $h\otimes_{B'} B: X\otimes_{B'} B \rightarrow Y\otimes_{B'} B.$
\end{example}

We recall a formal framework for basechange. We do not need monoidal structures here.
\begin{situation}\label{basic-situation-basechange}
We have a diagram of categories and functors
\begin{equation}\label{eq:basechange1}
\begin{tikzcd}
\sS  \arrow[d, "g_*"] \arrow[r, "m_*'"]& \sC \arrow[d, "f_*"]\\
\sT  \arrow[r, "m_*"]& \sD
\end{tikzcd}
\end{equation}
where the functors $f_*, g_*, m_*, m'_*$ have left adjoints $f^*, g^*, m^*, m'^*$, and we are given a natural transformation $m'^*f^* \simeq g^*m^*$.
\end{situation}
In this situation we get get a unique natural transformation $m_*g_*\simeq f_*m'_*$ such that the adjunctions for $(m'^*f^*, f_*m'_*)$ and $(m_*g_*, g^*m^*)$ are compatible (see \cite[Sec~3.6]{lipman}). We define the basechange map 
\begin{equation}\label{eq:basechange}
m^*f_*X \rightarrow g_*m'^*X \quad \quad \quad \text{for}\; X \in \sS
\end{equation}
as in \cite[Prop~3.7.2(i)]{lipman}.\note{this basechange map can be defined can be defined without $m_*$---but you still need to know $m_*$ exists to know that all my diagrams commute, because my refernces all assume the existence of $m_*$.} It may not be an isomorphism in general.

\begin{lemma}\label{lem:bc-unit}
For $Y \in \sD$ and $X \in \sC$ there are commuting diagrams 
\begin{equation}\label{eq:bc-unit1}
\begin{tikzcd}
m'^*f^*f_*X \arrow[d, equal] \arrow[r, "m'^*\epsilon"] & m'^*X\\
g^*m^*f_*X \arrow[r, "\text{\eqref{eq:basechange}}"] & g^*g_*m'^*X \arrow[u, "\epsilon"]
\end{tikzcd} \quad \quad \quad \quad
\begin{tikzcd}
m^*Y \arrow[r, "m^*\eta"] \arrow[d, "\eta"] & m^*f_*f^*Y \arrow[d, "\text{\eqref{eq:basechange}}"]\\
g_*g^*m^*Y \arrow[r, equal] & g_*m'^*f^*Y
\end{tikzcd}
\end{equation}
\end{lemma}
\begin{proof}
We show that the first diagram commutes; the second one may be checked similarly. Commutativity of the first follows from the following commuting diagram.
\begin{equation}\label{eq:bc-counit2}
\begin{tikzcd}
g^*m^*f_*X \arrow[r, "g^*m^*f_*\eta"] \arrow[d, equal] &[1.5em] g^*m^*f_*m'_*m'^*X \arrow[d, equal] \arrow[r, equal] &[-1.5em] g^*m^*m_*g_*m'^*X \arrow[r, "g^*\epsilon^{m^*}"]& g^*g_*m'^*X \arrow[dd, "\epsilon^{g^*}"] \\
m'^*f^*f_*X \arrow[r, "m'^*f^*f_*\eta"] \arrow[d, "m'^*\epsilon"] & m'^*f^*f_*m'_*m'^*X \arrow[d, "m'^*\epsilon"] & &&\\
 m'^*X \arrow[r, "m'^*\eta"]&m'^*m'_*m'^*X \arrow[rr, "\epsilon^{m'^*}"] && m'^*X
\end{tikzcd}
\end{equation}
The perimeter of the diagram from $m'^*f^*f_*X$ to $m'^*X$ along the bottom is equal to $m'^*\epsilon$ using a triangle identity, while the composition along the top is equal to the composition of the other three arrows in the desired square. The commutativity of the big cell in \eqref{eq:bc-counit2} is compatibility of the $(m'^*f^*, f_*m'_*)$ and $(m_*g_*, g^*m^*)$ adjunctions---see \cite[(3.6.2)]{lipman}.\note{
a proof of the other diagram:
We apply the $(g^*, g_*)$ adjunction to the two arrows $m^*Y \to g_*g^*m^*Y$. The adjoint of the left vertical arrow in \eqref{eq:bc-unit1} is the identity. The adjoint of the composition of the remaining arrows in \eqref{eq:bc-unit1} is the composition
\[
g^*m^*Y \xrightarrow{g^*m^*\eta} g^*m^*f_*f^*Y \xrightarrow{=} m'^*f^*f_*f^*Y \xrightarrow{m'^*\epsilon_{f^*Y}} m'^*f^*Y \xrightarrow{=} g^*m^*Y
\]
using the second description of \eqref{eq:basechange} in \cite[Lem~A.7(2)]{hall-GAGA}. One may use a triangle identity to show that this composition is the identity.}
\end{proof}

\subsection{An ideal setup}
We recall the formal framework of \cite[Rmk~5.10]{FHM}.

\begin{situation}\label{situation} We are given closed symmetric monoidal categories $\sC, \sD$ and functors $f_*, f_!: \sC \rightarrow \sD$ and $f^*, f^!: \sD \rightarrow \sC$ such that $(f^*, f_*)$ and $(f_!, f^!)$ are adjoint pairs. Moreover, these functors satisfy
\begin{itemize}
\item $f^*$ is strong symmetric monoidal
\item $f_* = f_!$
\item The canonical projection formula morphism 
\begin{equation}\label{eq:projection}
\pi: Y \otimes f_*(X) \rightarrow f_*(f^*Y\otimes X)
\end{equation}
defined in \cite[Appendix~A]{hall-GAGA} is an isomorphism
\item The object $C := f^! \sO_{\sD}$ is invertible, and
\item The canonical morphism 
\begin{equation}\label{eq:piso}
\piso: f^*Y \otimes f^!\sO_\sD \rightarrow f^!Y
\end{equation}
defined in \cite[(5.5)]{FHM} is an isomorphism.
\end{itemize}
\end{situation}

\begin{example}

Let $p:\cC \rightarrow T$ be a family of prestable curves on a quasi-separated Noetherian scheme $T$, in the sense of \cite[Tag~0E6T]{tag}. Let $\sC = \Dqc(\cC_{\et})$, $\sD = \Dqc(T_{\et})$, $f_* = \R p_*$, and $f^*=\L p^*$. Then $f_*$ has a right adjoint $f^!$ and $f^! \OO_T$ is equal to the relative dualizing sheaf $\omega_{\cC/T}[1]$. These data are an example of Situation \ref{situation}.
We will extend this example to families $\cC \to T$ of twisted curves in Example \ref{ex:situation}.

\end{example}

\begin{lemma}\label{lem:coev}
For every $X \in \sC$, the $(\otimes, \Hom)$-unit
\begin{equation}\label{eq:coev}
\eta_X: X \rightarrow \Hom(C, X\otimes C)
\end{equation} is an isomorphism.
\end{lemma}
\begin{proof}
Since $C$ is invertible, it follows from \cite[Lem~2.9]{may} that $C$ is dualizable and that the coevaluation map defined there is an isomorphism. It follows from the definition of the coevaluation map that the unit \eqref{eq:coev} is an isomorphism when $X=\OO_{\sC}$. For general $X$, there is a commuting square
\[
\begin{tikzcd}
\OO_{\sC} \otimes X \arrow[r, "\eta_{\OO_{\sC}}\otimes 1_X"] \arrow[d, equal]&[3em] \Hom(C, \sO_{\sC}\otimes C) \otimes X\arrow[r, equal] &\Hom(C, C) \otimes X \arrow[d, "\nu", "\sim"'] \\
X \arrow[r, "\eta_X"] &\Hom(C, X \otimes C) \arrow[r, equal] & \Hom(C, C\otimes X)
\end{tikzcd}
\]
where the map labeled $\nu$ (defined in \cite[120]{LMS}) is an isomorphism since $C$ is dualizable (see \cite[Prop~III.1.3(ii)]{LMS}). This implies that $\eta_X$ is an isomorphism. The commutativity
\note{$H = \Hom$ uncomment commutative diagram below
composition of right vertical arrows is definitily $\nu$, including the messy middle arrow; square with equalities commutes by coherence ( or check it); for the composition of the bottom arrows, pull the commutativity of $X$ and $C$ out and then use a triangle/zig-zag identity .}
of the square follows immediately from the definition of $\nu$ and the functoriality of $\eta$.\note{$\nu$ here is defined to be the composition $\Hom(C, C) \otimes X \rightarrow \Hom(C, \Hom(C, C)\otimes X \otimes C) \rightarrow \Hom(C, C\otimes X)$ where the first arrow is $\eta$ and the second is evaluation (counit). This is tricky because $\eta_\OO$ really maps to $\Hom(C, C\otimes \OO)$.}
\end{proof}

Following \cite[Def~5.6]{FHM}, we define twisted functors $f^!_C(X) := \Hom(C, f^!(X))$ and $f_!^C(X):=f_!(X\otimes C)$.\note{a priori, in the desired application, $\cC$ and $\cD$ are $\Dqc$ categories, so $\Hom$ is $\Hom_{qc}$ which might not agree with $\Hom$ in general---but it does if $C$ is perfect. so i can compute the twisted functor using usual hom.} We have an isomorphism $f^*Y \rightarrow f^!_C(Y)$ for $Y \in \sC$ equal to the composition
\begin{equation}\label{eq:sheaves1}
f^*(Y) \xrightarrow{\text{\eqref{eq:coev}}} \Hom(C,  f^*Y\otimes C) \xrightarrow{\piso} \Hom(C, f^!Y)
\end{equation}
where \eqref{eq:coev} is an isomorphism by Lemma \ref{lem:coev} and $\piso$ is an isomorphism by assumption.
\note{maybe $\eta$ is an isomorphism in some generality for formal reasons, see 0E47(3)---something to check here though. NO I think that if $\eta$ is an isomorphism for ALL $X$, then $C$ is invertible! Because by the commuting square in Lemma 2.1.2 the map $\nu$ is an isomorphism for ALL $X$, which I think (because of \cite[Example~5.5]{FHM}) implies that $C$ is dualizable. Also if $\eta$ is an isomorphism then $coev$ is an isomorphism, and since we already know $C$ is dualizable this implies $C$ is invertible. I think the best I could get is that $\eta_X$ is an isomorphism for dualizable $X$. }
 Since $(f^C_!, f^!_C)$ is an adjoint pair, we've realized $f^C_!$ as a left adjoint to pullback. Moreover, there is a projection isomorphism
\[
\pi_C: Y \otimes f_!^C(X) \xrightarrow{\sim} f_!^C(f^*(Y) \otimes X)
\]
 defined by replacing $X$ with $X \otimes C$ in $\pi$.
 
 In this setting, we prove commutativity of some diagrams which will be useful to us.
\begin{lemma}\label{lem:projection-counit}
There is a commuting diagram
\begin{equation}\label{eq:situation1}
\begin{tikzcd}
X \otimes f^C_!f^*(Y)  \arrow[r, "\text{\eqref{eq:sheaves1}}", "\sim"'] \arrow[d, "\pi_C", "\sim"']& X \otimes f_!^C f^!_C(Y) \arrow[r, "\epsilon"] & X \otimes Y\\
f_!^C(f^*(X)\otimes f^*(Y))& \arrow[l, "\sim"]f^C_!f^*(X\otimes Y) \arrow[r, "\text{\eqref{eq:sheaves1}}", "\sim"'] & f_!^C f^!_C(X\otimes Y) \arrow[u, "\epsilon"]
\end{tikzcd}
\end{equation}
where the arrows labeled $\epsilon$ are counits for the $(f^C_!, f_C^!)$ adjunction. 
\end{lemma}
\begin{proof}
 For any $Z \in \sD$, the composition
\[
f^C_!f^*(Z) \xrightarrow{\text{\eqref{eq:sheaves1}}} f^C_!f^!_C(Z) \xrightarrow{\epsilon^{f^C_!}} Z
\]
is equal to
\[
f^C_!f^*(Z)=f_*(f^*(Z) \otimes C) \xleftarrow[\sim]{\pi} Z \otimes f_*(C) \xrightarrow{\epsilon^{f_!}_{\OO_{\sD}}} Z
\]
To see this, expand the $(f_!^C, f^C_!)$-counit in terms of the $(\otimes, \Hom)$-counit and the $(f_!, f^!)$-counit; commute the morphism $\piso$ in the definition of \eqref{eq:sheaves1} with the $(\otimes, \Hom)$-counit; and finally use the triangle identity $1_{f^*Z\otimes C} = \epsilon^{\otimes}_{f^*Z\otimes C}\circ (\eta^{\otimes}_{f^*Z}\otimes 1_C)$, where $\eta$ and $\epsilon$ here denote the unit and counit of the $(\otimes, \Hom)$ adjunction. Now \eqref{eq:situation1} is equivalent to the diagram
\[
\begin{tikzcd}
X \otimes f_*(f^*(Y) \otimes C) \arrow[d, "\pi"] & \arrow[l, "\pi"] X \otimes Y \otimes f_*(C) \arrow[d, "\pi"] \arrow[r, "\epsilon"] & X\otimes Y\\
f_*(f^*(X) \otimes f^*(Y) \otimes C) \arrow[r, equals] & f_*(f^*(X\otimes Y) \otimes C)
\end{tikzcd}
\]
whose commutativity is proved in \cite[Lem~3.4.7(iv)]{lipman}.
\end{proof}

\begin{lemma}\label{lem:triangle}
Suppose we are in Situation \ref{situation}. Then there is an isomorphism $\gamma: f_*\Hom(f^*X, f^!_C(Y)) \rightarrow \Hom(f^C_!(f^*X), Y)$ making the following diagram commute.
\begin{equation}\label{eq:triangle1}
\begin{tikzcd}
\Hom(f^C_!f^!_C(X), Y) & \arrow[l, "\Hom{(\epsilon, 1_Y)}"'] \Hom(X, Y) \arrow[d, "\text{\eqref{eq:sheafy-pullback}}"] \\
& \arrow[ul, "\gamma"] f_*\Hom(f^*X, f^*Y)
\end{tikzcd}
\end{equation}
Here, $\epsilon$ is the counit for the $(f_!^C, f^!_C)$-adjunction. 
\end{lemma}
\begin{proof}
The definition of $\gamma$ will come out in the course of the proof: it will be ``conjugate'' to $\pi$ via various adjoints (see also \cite[(4.1)]{FHM} and \cite[Tag~0A9Q]{tag}). For future reference, we summarize it in the final paragraph of the proof. To simplify notation, when there is no risk of confusion, if $F$ is a functor between categories and $\alpha$ is a morphism of the source category, we will notate $F(\alpha)$ by $\alpha$. For example, we may use $\epsilon$ as the label for the horizontal arrow in \eqref{eq:triangle1}.

To show the commutativity of \eqref{eq:triangle1} we use the Yoneda embedding: for an arbitrary $T \in \sD$, it suffices to show the commutativity of 
\[
\begin{tikzcd}
{}^\gray{3}\sD(T, \Hom(f_!^Cf^!_C(X), Y) & \arrow[l, "\epsilon^{f^C_!}"']  {}^\gray{1}\sD(T, \Hom(X, Y)) \arrow[d, "\text{\eqref{eq:sheafy-pullback}}"] \\
&{}^\gray{2}\sD(T, f_*\Hom(f^*X, f^*Y)) \arrow[ul, "\gamma"].
\end{tikzcd}
\]
We do this by demonstrating that it is equivalent to the commutativity of \[
\begin{tikzcd}
{}^\gray{6}\sD(T\otimes f^C_!f^!_C(X), Y) & \arrow[l, "\epsilon^{f^C_!}"']  {}^\gray{4}\sD(T\otimes X, Y) \arrow[d, "\epsilon^{f^C_!}"] \\
&{}^\gray{5}\sD(f^C_!f^!_C(T\otimes X),Y) \arrow[ul, "\tilde\gamma"].
\end{tikzcd}
\]
where $\tilde\gamma$ is defined to equal the isomorphisms in \eqref{eq:situation1}. This second diagram commutes by Lemma \ref{lem:projection-counit}.

There are isomorphisms $\gray{1}=\gray{4}$ and $\gray{3}=\gray{6}$ given by $(\otimes, \Hom)$ adjunction, and an isomorphism
$\gray{2}=\gray{5}$ given by
\begin{equation}
\begin{aligned}\label{eq:long1}
{}^\gray{2}\sD(T, &f_*\Hom(f^*X, f^*Y)) = \sC(f^*T, \Hom(f^*X, f^*Y)) = \sC(f^*T \otimes f^*X, f^*Y)\\
&= {}^{\gray{7}}\sC(f^*(T\otimes X), f^*Y) = \sC(f^!_C(T\otimes X), f^!_CY)= {}^\gray{5}\sD(f^C_!(f^!_C(T\otimes X)), Y)
\end{aligned}
\end{equation}
where the equalities are $(f^*, f_*)$-adjunction, $(\otimes, \Hom)$-adjunction, the isomorphism \eqref{eq:strong-monoidal}, the isomorphism \eqref{eq:sheaves1}, and $(f^C_!, f^!_C)$-adjunction.
The square with corners $\gray{1}, \gray{3}, \gray{4},$ and $\gray{6}$ commutes by functoriality of the $(\otimes, \Hom)$ adjunction. We take the commutativity of the square with corners $\gray{2}, \gray{3}, \gray{5},$ and $\gray{6}$ as the definition of $\gamma$. The final square commutes as follows. By the definition of the vertical map in \eqref{eq:triangle1} and adjunction, the composition $\gray{1} \rightarrow \gray{2} \rightarrow \gray{7}$ is equal to 
\begin{align*}
\sD(T, \Hom(X,Y)) &\xrightarrow{f^*}\sC(f^*T, f^*\Hom(X,Y)) \xrightarrow{\otimes f^*X} \sC(f^*T\otimes f^*X, f^*\Hom(X,Y)\otimes f^*X)\\
&\xlongequal{\text{\eqref{eq:strong-monoidal}}} \sC(f^*(T\otimes X), f^*(\Hom(X,Y)\otimes X)) \xrightarrow{f^*\epsilon^{\otimes}} \sC(f^*(T\otimes X), f^*Y)
\end{align*}
By functoriality of \eqref{eq:strong-monoidal}, this composition is equivalent to
\[
{}^\gray{1}\sD(T, \Hom(X,Y)) \xrightarrow{\otimes X}\sD(T\otimes X, \Hom(X,Y)\otimes X) \xrightarrow{\epsilon^\otimes} {}^\gray{4}\sD(T\otimes X, Y) \xrightarrow{f^*}{}^\gray{7}\sC(f^*(T\otimes X), f^*Y).
\]
The first two arrows are precisely the $(\otimes, \Hom)$ adjunction $\gray{1} = \gray{4}.$ Finally, the composition
\[
{}^\gray{4}\sD(T\otimes X, Y) \xrightarrow{f^*}{}^{\gray{7}}\sC(f^*(T\otimes X), f^*Y) \xlongequal{\text{\eqref{eq:sheaves1}}} \sC(f^!_C(T\otimes X), f^!_CY)= {}^\gray{5}\sD(f_!(f^!_C(T\otimes X)), Y),
\]
where the arrow comes from the previous formula and the two equalities come from \eqref{eq:long1}, is induced by the counit $\epsilon^{f^C_!}$ as desired.

To conclude, we summarize the definition of $\gamma$ as promised. After cancelling assorted isomorphisms with their inverses, we see that it is given under the Yoneda embedding by
\begin{align*}
\sD(T, &f_*\Hom(f^*X, f^*Y)) = \sC(f^*T, \Hom(f^*X, f^*Y)) = \sC(f^*T \otimes f^*X, f^*Y) \\
&= \sC(f^*T \otimes f^*X, f^!_CY) = \sD(f_!^C(f^*T \otimes f^*X), Y) \xrightarrow{\sim} \sD(T \otimes f_!^Cf^*X, Y)\\
&= \sD(T ,\Hom( f_!^Cf^*X, Y)) = \sD(T ,\Hom( f_!^Cf^!_CX, Y)).
\end{align*}
Where the equalities are $(f^*, f_*)$ adjunction, $(\otimes, \Hom)$ adjunction, the isomorphism \eqref{eq:sheaves1}, $(f^C_!, f^!_C)$ adjunction, the projection formula, $(\otimes, \Hom)$ adjunction, and finally the isomorphism \eqref{eq:sheaves1} again. In particular, if we follow $\gamma$ with the inverse of the last equality, we have defined a natural isomorphism
\begin{equation}\label{eq:newgamma}
f_*\Hom(Z, f^*Y) \xrightarrow{\sim} \Hom( f_!^CZ, Y)
\end{equation}
that is functorial in both arguments. (The content of this statement is that to define \eqref{eq:newgamma} it is not necessary for $Z$ to be of the form $f^*X$.) 
\end{proof}

\subsection{A modification of the ideal situation}\label{sec:modification}
When $\cC \rightarrow \cX$ is a family of twisted curves on an algebraic stack $\cX$, we would like to apply Situation \ref{situation} by setting $\sD=\Dqc(\cX_{\liset})$ and $\sC = \Dqc(\cC_{\liset})$. Unfortunately we do not know a proof that the right adjoint $f^!$ in Situation \ref{situation} exists in the generality we would like (see Lemma \ref{lem:f!-exists} and the discussion following it). Instead, we work in the following weaker situation, replacing $f^!$ with a dualizing complex and trace map.

\begin{situation}\label{weaker-situation}
We are given closed symmetric monoidal categories $\sC$, $\sD$, a functor $f_*:\sC \rightarrow \sD$ with a right adjoint $f^*$, and an invertible object $C \in \sC$ with a trace map $tr: f_*C \rightarrow \OO_{\sD}$. These data satisfy
\begin{itemize}
\item $f^*$ is strong symmetric monoidal
\item The canonical morphism $\pi: Y\otimes f_*(X) \rightarrow f_*(f^*Y \otimes X)$ is an isomorphism
\end{itemize}
\end{situation}

In this situation we define an adjunction-like map $a: \sC(X, f^*Y) \rightarrow \sD(f_*(X\otimes C), Y)$ as follows (see also \cite[Sec~A.2.1]{CJW}).\footnote{Informally we think of $a$ as realizing $f^C_!=f_*(\cdot \otimes C)$ as a left adjoint to $f^*=f^!_C$.} Given $\phi' \in \sC(X, f^*Y)$, define $a(\phi')$ to be the composition
\begin{equation}\label{eq:adjunction-like}
f_*(X \otimes 1_C) \xrightarrow{f_*(\phi' \otimes C)} f_*(f^*Y \otimes C) \xleftarrow[\sim]{\pi} Y \otimes f_*C \xrightarrow{id \otimes tr_C} Y.
\end{equation}
Observe that $a$ is functorial in both arguments, by which we mean the following:
\begin{enumerate}
\item Given $X' \in \sC$ and $\psi \in \sC(X', X)$, we have $a(\phi' \circ \psi) = a(\phi') \circ f_*(\psi \otimes 1_C)$.
\item Given $Y' \in \sD$ and $\psi \in \sD(Y, Y')$, we have $a(f^*\psi \circ \phi') = \psi \circ a(\phi').$
\end{enumerate}
The next example explains why we call $a$ ``adjunction-like.''

\begin{example}\label{ex:ideal}
Suppose we are in Situation \ref{situation}, with an adjoint pair $(f^*, f_*)$ and object $C = f^!\sO_{\sD} \in \sC$.  Then we have the data of Situation \ref{weaker-situation}: we can define $tr_C: f_*C \rightarrow \sO_{\sD}$ to be the counit of the $(f_!, f^!)$ adjunction. Under the isomorphism \eqref{eq:sheaves1}, the adjunction-like map $a: \sC(X, f^*Y) \rightarrow \sD(f_*(X\otimes C), Y)$ is identified with the adjunction $\sC(X, f^!_CY) \simeq \sD(f^C_!X, Y)$.  To see this, let $\phi' \in \sC(X, f^*Y)$. The $(f^C_!, f^!_C)$-adjoint of $\text{\eqref{eq:sheaves1}}\circ\phi'$ is equal to the $(f_!, f^!)$-adjoint of
\[
X \otimes C \xrightarrow{\phi' \otimes 1_C} f^*Y \otimes C \xrightarrow{\piso} f^!Y.
\]
Since $f_*=f_!$, said adjoint is equal to the composition
\[
f_*(X \otimes C) \xrightarrow{f_*(\phi' \otimes 1_C)} f_*(f^*Y \otimes C) \xrightarrow{\hat\piso} Y.
\]
where $\hat \piso$ is the $(f_!, f^!)$-adjoint of $\piso$. By the definition of $\piso$ in \cite[(5.5)]{FHM}, this last composition is equal to $a(\phi')$.
\end{example}

\subsection{Basechange}
We introduce a setting where the adjunction-like morphism $a$ is compatible with pullback.

\begin{situation}\label{situation2}
We have a diagram of closed symmetric monoidal categories as in \eqref{eq:basechange1}
such that the left adjoints $f^*, g^*, m^*, m'^*$ are strong symmetric monoidal. We are given objects $S \in \sS, C \in \sC$ and morphisms $tr_S: g_*S \rightarrow \OO_{\sT}, tr_C: f_*C \rightarrow \OO_{\sD}$ such that the data for each column of \eqref{eq:basechange1} are in Situation \ref{weaker-situation}, and these data are compatible as follows.
\begin{itemize}
\item The basechange map \eqref{eq:basechange} is an isomorphism.
\item We are given an isomorphism $\alpha: m'^*C \rightarrow S$ making this diagram commute:
\begin{equation}\label{eq:tr-functoriality}
\begin{tikzcd}
m^*f_*C \arrow[d, "\text{\eqref{eq:basechange}}"] \arrow[r, "m^*tr_C"]\arrow[d] & \OO_{\sT}\\
g_*m'^*C \arrow[r, "g_*\alpha"] & g_*S \arrow[u, "tr_S"]
\end{tikzcd}
\end{equation}
\end{itemize}
\end{situation}

In this situation, Lemma \ref{lem:technical} explains a precise sense in which $m^*a(\phi') = a(m'^*\phi')$.

\begin{lemma}\label{lem:technical}
Suppose we are in Situation \ref{situation2}. Let $\phi': X\rightarrow f^*(Y)$ be an arrow in $\sC$. Then we have $m'^*\phi': m'^*X \rightarrow m'^*f^*Y = g^*m^*Y$, and the following diagram commutes:
\[
\begin{tikzcd}
m^*f_*(X\otimes C) \arrow[r, "m^*a(\phi')"] \arrow[d, "\fa_{m^*}"', "\sim"] & m^*Y \\
g_*(m'^*X\otimes S) \arrow[ur, "a(m'^*\phi')"']
\end{tikzcd}
\]
The isomorphism $\fa$ is equal to \eqref{eq:basechange} followed by \eqref{eq:strong-monoidal} and finally $\alpha$, and in particular it is functorial in $X$. Moreover, suppose we have a diagram 
\begin{equation}\label{eq:situation3}
\begin{tikzcd}
\sC_1  \arrow[d, "f_{1*}"] & \arrow[l, "m'^*_{12}"]\sC_2  \arrow[d, "f_{2*}"]  &\arrow[l, "m'^*_{23}"] \sC_3 \arrow[d, "f_{3*}"]\\
\sD_1 & \arrow[l, "m^*_{12}"] \sD_2  &\arrow[l, "m^*_{23}"] \sD_3
\end{tikzcd}
\end{equation}
together with distinguished objects $C_i \in \sC_i$ and trace maps $tr_i: f_{i*}C_i \rightarrow \sO_{\sD_i}$ for each $i$, and isomorphisms $\alpha_{ij}: m^{\prime *}_{ij}C_j \rightarrow C_i$ for $i<j$. Suppose that with these data both squares and the outer rectangle of \eqref{eq:situation3} are in Situation \ref{situation2}, and that $\alpha_{13} = \alpha_{12} \circ m'^*_{12}(\alpha_{23}).$ Then $\fa_{m^*_{12}\circ m^*_{23}} = \fa_{m^*_{12}} \circ m'^*_{12}(\fa_{m^*_{23}}).$
\end{lemma}
\begin{proof}
The proof of \cite[Lem~A.2.1]{CJW} works in this more general situation.
\end{proof}

\begin{lemma}\label{lem:strange}
Suppose we are in Situation \ref{situation2}, but that the left column of \eqref{eq:basechange1} is actually in Situation \ref{situation}: this means we are given a right adjoint $g^!$ for $g_*=:g_!$ with $g^!\OO_\sT=S$ and $tr_S: g_*S \rightarrow \OO_{\sT}$ equal to the counit for the $(g_!, g^!)$ adjunction.
Let $\phi': X\rightarrow f^*(Y)$ be an arrow in $\sC$ and set $\phi:=a(\phi')$. Let $I$ be an object of $\sT$. Then the following diagram commutes.
\begin{equation}\label{eq:strange}
\begin{tikzcd}
\Hom(m^*f_*(X\otimes C), I) & \Hom(m^*Y, I) \arrow[l, "m^*\phi"] \arrow[d]\\
g_*\Hom(m'^*X, g^*I) \arrow[u, equal] &\arrow[l, "m'^*\phi'"] g_*\Hom(m'^*f^*Y, g^*I)
\end{tikzcd}
\end{equation}
In this diagram, the equality is comprised of \eqref{eq:basechange}, $\alpha$, and \eqref{eq:newgamma}. The right vertical arrow is \eqref{eq:sheafy-pullback} followed by \eqref{eq:basechange}.
\end{lemma}

\begin{proof}
The desired commuting diagram is derived from the composition of two. On the left, we have
\begin{equation}\label{eq:strange2}
\begin{tikzcd}
\Hom(g_!^S(m'^*X), I) & \arrow[l]\Hom(g_!^S(m'^*f^*Y), I) \arrow[r, equal, "\text{\eqref{eq:basechange}}"] & \Hom(g_!^Sg^*(m^*Y), I)\\
g_*\Hom(m'^*X, g^*I) \arrow[u, "\text{\eqref{eq:newgamma}}"] & g_*\Hom(m'^*f^*Y, g^*I) \arrow[l] \arrow[u, "\text{\eqref{eq:newgamma}}"] \arrow[r, equal, "\text{\eqref{eq:basechange}}"] & g_*\Hom(g^*m^*Y, g^*I) \arrow[u, "\text{\eqref{eq:newgamma}}"]
\end{tikzcd}
\end{equation}
Here, the arrows pointing left are induced by $m'^*\phi'$. On the right, we have
\begin{equation}\label{eq:strange3}
\begin{tikzcd}
\Hom(g_!^Sg^*(m^*Y), I) \arrow[r, "\pi_S","\sim"'] \arrow[d, equal, "\text{\eqref{eq:sheaves1}}"] & \Hom(m^*Y\otimes g_!^S(\OO_\sT), I) & \Hom(m^*Y, I) \arrow[l, "tr_S"']\arrow[dll, "\epsilon"] \arrow[ddll, "\text{\eqref{eq:sheafy-pullback}}"]\\
\Hom(g_!^Sg^!_S(m^*Y), I) \\
g_*\Hom(g^*m^*Y, g^*I) \arrow[u, "\gamma"]
\end{tikzcd}
\end{equation}
The commutativity of the top triangle is Lemma \ref{lem:projection-counit} with $X=m^*Y$ and $Y=\OO_{\sT}$, and the commutativity of the bottom triangle is Lemma \ref{lem:triangle}. We note that definition \eqref{eq:newgamma} is equal to $\gamma$ followed by the equality induced by \eqref{eq:sheaves1}.
Finally, by Lemma \ref{lem:technical}, the top row of \eqref{eq:strange2} followed by the top row of \eqref{eq:strange3} agrees with the top row of \eqref{eq:strange} (after inserting a copy of \eqref{eq:basechange}).
\end{proof}

\section{Duality for twisted curves}\label{sec:duality}

We explain how which the formal discussion of Section \ref{sec:formal} will generally be used in the remainder of this article. In this section and the remainder of the paper, we define pseudo-coherent and perfect objects of lisse-\'etale sites as in \cite[Tag~08FT]{tag} and \cite[Tag~08G5]{tag}, respectively. Note that if $X$ is an algebraic stack, pseudo-coherent and perfect objects of $\D({\cX_{\liset}})$ are always in $\Dqc(\cX_{\liset})$.\note{for perfect this is stated in HR17. for ps-coh, say $\cF$ is ps-coh. Then for each $m$, there is a smooth cover $U \to \cX$ such that $\cF|_U$ has a map to a perfect complex $E$ that is surjective in degree $m$ and an isomorphism in higher degrees. Since the restriction functor $|_U$ is exact, and perfect obejcts have qcoherent cohomology, we see that $H^i(\cF)$ is qcoherent.}

\begin{example}\label{ex:general}If $f: \cX \rightarrow \cY$ is a morphism of algebraic stacks, we have closed symmetric monoidal categories $\sC = \Dqc(\cX_{\liset})$ and $\sD = \Dqc(\cY_{\liset})$ and a strong monoidal functor $\L f^*: \sD \rightarrow \sC$. If $f$ is concentrated we also have $\R f_*: \sC \rightarrow \sD$ that is a right adjoint to $\L f^*$. In this context, the functor notated $\Hom$ in Section \ref{sec:formal} translates to internal hom for $\Dqc(\cX_{\liset})$. However, we note that by \cite[Lem~4.3(2)]{HR17}, for any algebraic stack $\cX$, we have equality \[\Rhom^{\qc}_{\sO_{\cX}}(\cP, \cF) \simeq \Rhom_{\sO_{\cX}}(\cP, \cF)\]
for any $\cF \in \Dqc(\cX_{\liset})$ and any perfect complex $\cP\in \Dqc(\cX_{\liset})$.
\end{example}
\subsection{Background on twisted curves}
Recall from Section \ref{sec:notation} that a morphism $p: \cC \rightarrow \cM$ of algebraic stacks is a \textit{family of twisted curves} if smooth-locally on $\cM$ it is a twisted curve in the sense of \cite[Def~2.1]{AOV11}. In particular, $p$ is flat\note{this should follow from the description in (iv) and (v). Let's write out the smooth case. Let $\bar p \to C$ be a geometric point. We will show that $\cC \to S$ is flat at $\bar p$. We have a commuting diagram
\[
\begin{tikzcd}[ampersand replacement=\&]
E \arrow[d, "sm"] \&\& E' \arrow[ll, "fl"]\arrow[d, "sm"]\\
\cC \arrow[d] \& {[D^{sh}/\mu_r]} \arrow[l, "fl"] \arrow[d] \& D^{sh} \arrow[l, "fl"] \arrow[ddl] \arrow[r, "fl","07QM"'] \& \Spec(\OO_{S, f(\bar p)}[z]) \arrow[ddll, "sm"]\\
C \arrow[d] \& \Spec(\OO_{C, \bar p}) \arrow[l, "fl"] \arrow[d] \\
S \& \arrow[l, "fl","07QM"'] \Spec(\OO_{S, f(\bar p)})
\end{tikzcd}
\]
where $fl$ means a map is flat and $sm$ means it is smooth. The top rectangle is fibered as is the square just underneath. we conclude from info in the diagram that $D^{sh} \to S$ is flat. So $E' \to S$ is flat. Since $E' \to E$ is flat, the proof of $036K$ shows that $E \to S$ is flat. This is the definition for $\cC \to S$ to be flat at a point mapping to $\bar p$.} and proper and the diagonal $\cC \to \cC \times_{\cM} \cC$ is quasi-finite.\note{use lemma \ref{lem:local-curves} and 06UF(4). also, this is implied by p.18 l.13. this line also says $p$ is flat.} 
If $r: \cC \to C$ is the coarse moduli map, and $\bar c \to C$ is a geometric point, the fiber product $\Spec(\OO_{C, \bar c}) \times_C \cC$ is moreover required to have a certain description (see the full definition in \cite[Def~2.1]{AOV11}). We recall some properties of families of twisted curves. 

Our first lemma ``spreads out'' the local quotient description of 
a twisted curve at a geometric point to an \'etale neighborhood of that point.
\begin{lemma}\label{lem:local-curves}
Let $\cC \to T$ be a family of twisted curves over an affine scheme $T$ and let $q: C \to T$ be the coarse moduli space. Let $\bar c \to C$ be a geometric point. Then there is an integer $n \geq 1$ and an affine scheme $V = \Spec(A)$ with an action of $\mu_n$ such that, if $R = A^{\mu_n}$ is the ring of invariants and $U := \Spec(R)$, there is a commuting diagram
\begin{equation}
\label{eq:local_duality_new}
	\begin{tikzcd}
	V \arrow[r, "\sigma"] & {[V/\mu_n]} \arrow[r, "\tau"] \arrow[d] & U \arrow[d] \\
	& \cC \arrow[r, "r"] & C 
 	\end{tikzcd}
	\end{equation}
where the square is fibered and the vertical maps are \'etale.\note{$\mu_n$ is finite and flat but NOT \'etale in general (it is \'etale if $n$ is relatively prime to the characteristic of a ground field). So $\sigma$ is finite and flat but NOT \'etale.} Moreover, one of the following  holds:
\begin{enumerate}
		\item $A = R[x]/(x^n-t)$ for some $t \in R$ and $\mu_n$ acts by $\zeta \cdot p(x) = p(\zeta x)$.
		\item $A = R[x,y]/(xy-t, x^r-u, y^r-v)$ for some $t, u, v \in R$ and $\mu_n$ acts by $\zeta \cdot p(x,y) = p(\zeta x, \zeta^{-1}y)$.
	\end{enumerate}
\end{lemma}
\begin{proof}
We prove the lemma when $\bar c$ maps to a node of $C$; the case when $\bar c$ maps to a smooth point is similar. By definition \cite[Def~2.1(v)]{AOV11}, there is a fiber square
\begin{equation}\label{eq:local-curves}
\begin{tikzcd}
{[(\Spec(\OO_{T, q(\bar c)}[x,y]/(xy-t))/\mu_n]} \arrow[d] \arrow[r] & \cC \arrow[d]\\
\Spec(\OO_{C, \bar c}) \arrow[r] & C
\end{tikzcd}
\end{equation}
for some $t \in \OO_{T, q(\bar c)}$, where $\zeta \in \mu_n$ acts by $x \mapsto \zeta \cdot x$ and $y \mapsto \zeta^{-1}\cdot y$.\note{The ring $\OO_{C, \bar c}$ is defined to be the limit over \'etale neighborhoods factoring the geometric point $\bar c \to C$. Probably this is the only local ring that makes sense for an algebraic space $C$. If $C$ happens to be a scheme, it is somethign extra to show that this is the (strict?) henselization of the usual local ring.} Since $\cC$ is tame, formation of the coarse space commutes with arbitrary base change \cite[Cor~3.3]{AOV08}, and we have\note{I'm thinking of using the formula for computing the coarse modulispace of a finite group action in the sponge article. that lemma doesn't need $n$ to be invertible on the base.}\note{Let $A = \OO_{T, q(\bar c)}[x,y]/(xy-t)$. To make sense of $A^{\mu_n}$, since $\mu_n$ is a group scheme, we should define an action of $\mu_n$ on all $\ZZ$-algebras $R$. This is given by compatible group actions (in the usual sense) $\mu_n(R)$ acting on $A\otimes_{\ZZ}R$. An element $v \in A$ is invariant if $v\otimes 1$ is invariant in each ring $A \otimes_{\ZZ} R$.}
\[\OO_{C, \bar c} \simeq (\OO_{T, q(\bar c)}[x,y]/(xy-t))^{\mu_n} \simeq \OO_{T, q(\bar c)}[x^n, y^n]/(x^ny^n-t^n).\]
If we set $u:= x^n$ and $v:= y^n$ in $\OO_{C, \bar c}$, then we may write the top left corner of \eqref{eq:local-curves} as the stack
\begin{equation}\label{eq:new-corner}
[(\Spec(\OO_{C, \bar c}[x,y]/(x^n-u, y^n-v, xy-t))/\mu_n].
\end{equation}
Now write $\OO_{C, \bar c}$ as the inverse limit of affine schemes $\Spec(R_i)$ with \'etale maps to $C$. Since $\Spec(\OO_{C, \bar c}[x,y]/(x^n-u, y^n-v, xy-t)) \to \Spec(\OO_{C, \bar c})$ is finitely presented, there is an index $i_0$ and elements $u, v, t \in R_{i_0}$ such that $\Spec(R_{i_0}[x,y]/(x^n-u, y^n-v, xy-t))$ pulls back to the affine scheme in \eqref{eq:new-corner}. Define $\mu_n$ to act on $\Spec(R_{i_0}[x,y]/(x^n-u, y^n-v, xy-t))$ by the same rule $x \mapsto \zeta \cdot x$ and $y \mapsto \zeta^{-1}\cdot y$.

Let $\cC_{R_i}$ denote the pullback of $\cC$ to $\Spec(R_i)$. Observe that for $i \geq i_0$, we have two stacks $[(\Spec(R_{i}[x,y]/(x^n-u, y^n-v, xy-t))/\mu_n]$ and $\cC_{R_i}$ defined over $\Spec(R_i)$ and an isomorphism between their pullbacks to $\Spec(\OO_{C, \bar c})$. By \cite[Prop~4.18(i)]{LMB} there is an index $ j\geq i_0$ and an isomorphism $[(\Spec(R_{j}[x,y]/(x^n-u, y^n-v, xy-t))/\mu_n] \simeq \cC_{R_j}$. We may set $R:= R_j$.
\end{proof}

We refer the reader to Section \ref{sec:notation} for definitions of the direct and inverse image functors in the next lemma.

\begin{lemma}\label{lem:properties-of-p}
Let $p: \cC \rightarrow \cM$ be a family of twisted curves on an algebraic stack $\cM$.
\begin{enumerate}
\item The morphism $p$ has cohomological dimension $\leq 1$ (in the sense of \cite[Def~2.1]{HR17}).
\item The morphism $p$ is concentrated (in the sense of \cite[Def~2.4]{HR17}).
\item For $\cF \in \Dqc(\cC_{\liset})$ and $\cG \in \Dqc(\cM_{\liset})$, the projection morphism $\cG \otimes \R p_*(\cF) \xrightarrow{\text{\eqref{eq:projection}}} \R p_*(p^*\cG\otimes\cF)$ is an isomorphism.
\item Given a fiber square of algebraic stacks
\[
\begin{tikzcd}
\cC' \arrow[r, "m'"] \arrow[d, "p'"] & \cC \arrow[d, "p"]\\
\cM' \arrow[r, "m"] &\cM
\end{tikzcd}
\]
and $\cF \in \Dqc(\cC_{\liset})$, the basechange map $\L m^*\R p_* \cF \xrightarrow{\text{\eqref{eq:basechange}}} \R p'_*\L{m'}^* \cF$ is an isomorphism.
\item If $\cM$ is locally Noetherian, then the functor $\R p_*$ sends perfect complexes to perfect complexes.
\end{enumerate}
\end{lemma}
\begin{proof}
For part (1), by flat base change \cite[Lem~1.2(4)]{HR17} we may assume $\cM$ is an affine scheme,\note{this uses that $m^*$ is exact if $m: T \rightarrow \cM$ is smooth.} but this is \cite[Prop~2.6]{AOV11}. Now (1) implies part (2) by definition, part (3) by \cite[Cor~4.12]{HR17}, and part (4) by \cite[Cor~4.13]{HR17}. For part (5), we recall that perfection is a flat-local property of complexes in the sense of \cite[Lem~4.1]{HR17}, so we may use basechange \cite[Cor~4.13]{HR17} to reduce to the case when $\cM$ is a Noetherian affine scheme. Now the result follows from Lemma \ref{lem:pushing-ps-per} below.
\end{proof}

\note{there is basechange lemma 0fn1}

\begin{lemma}\label{lem:pushing-ps-per}
Let $p: \cC \to T$ be a family of twisted curves over a Noetherian affine scheme $T$ and let $r: \cC \to C$ be the map to the coarse moduli space.
\begin{enumerate}
\item The exact functor $r_*$ sends pseudo-coherent objects in $\Dqc(\cC_{\liset})$ to pseudo-coherent objects in $\Dqc(C_{\liset})$.
\item The functor $\R p_*$ sends perfect objects in $\Dqc(\cC_{\liset})$ to perfect objects in $\Dqc(T_{\liset})$.
\end{enumerate}
\end{lemma}

\begin{remark}\label{rmk:lazy-me}
We expect that the locally Noetherian hypothesis can be removed using absolute Noetherian approximation for algebraic stacks as in \cite[Tag~0CN4]{tag} (see the proof of \cite[Tag~01AH]{tag}). We do not, however, know a reference that allows us to assume the approximating morphism has properties (1) and (2) of Lemma \ref{lem:pushing-ps-per}. Since we are not aware of an application of the non-Noetherian setting we omit this investigation.
\end{remark}

\note{We would like to remove the Noetherian hypothesis by copying the proof of 01AH. I can use 0CN4 to approximate the stack, but I am not sure about the analog of 09RF in the stacky situation. The referee says that (1) holds for any proper tame morphism (with a Noetherian base??) and (2) holds for morphisms of finite tor dimension. So I want to say that I can increase i such that the morphism is proper, tame, and finite tor-dimension. It seems nontrivial to show you can do this for families of stacks. }
\begin{proof}[Proof of Lemma \ref{lem:pushing-ps-per}]
We will repeatedly use the fact that if $X$ is a scheme, there are equivalences of categories $\Qcoh(X_{\liset}) \simeq \Qcoh(X_{\zar}) $ and $\Dqc(X_{\liset}) \simeq \Dqc(X_{\zar})$, where $X_{\zar}$ is the category of sheaves on the small Zariski site of $X$, and that these equivalences preserve coherence, pseudo-coherence, and perfection.\note{part of ps-coh is 08he; for the analog when going between lisse-etale and etale, you know pullback always preserves ps-coh. The converse in this case is "obvious" (follows straight from the definition) because covers are the same in the two sites.}\note{ps-coh always correspond, but maybe coherent sheaves only correspond when the scheme is locally Noetherian---you need to know that the pullback of a coherent sheaf is coherent, for example to say that pullack from $X_{zar}$ to $X_{\et}$ (or ($X_{\liset}$) preserves coherent objects}

To prove (1), let $\cF \in \Dqc(\cC_{\liset})$ be pseudo-coherent. Let $f: U \to \cC$ be a smooth cover by a scheme. Since $f$ defines a morphism of lisse-\'etale sites and $f^*$ is exact, $f^*\cF$ is pseudo-coherent by \cite[Tag~08H4]{tag}. By \cite[Tag~08E8]{tag}\note{uses Noetherian}, the sheaves $H^i(f^*\cF)$ are coherent and vanish for $i \gg 0$.\note{Let $\R \epsilon: \Dqc(U_{\liset} \to \Dqc(U_{zar})$. its inverse equivalence is $\epsilon^*$. see 071Q. a priori i know $H^i(\R \epsilon_* f^*\cF)$ has these properties. so $\epsilon^*$ applied to these sheaves has these properties, but compute (since $\epsilon^*$ is exact) $\epsilon^*H^i(\R \epsilon_* f^*\cF) = H^i(\epsilon^*\R \epsilon_* f^*\cF) = H^i(f^*\cF)$. }
It follows from \cite[Rmk~6.10, Prop~6.12]{olsson-sheaves} that the sheaves $H^i(\cF)$ are coherent and vanish for $i\gg 0$. By \cite[Thm~4.16(x)]{alper13}\note{uses Noetherian} the sheaves $r_*H^i(\cF)$ have these same properties, but since $r_*$ is exact we know $r_*H^i(\cF)=H^i(r_*\cF).$ Hence by \cite[Tag~08E8]{tag} again, the object $r_*\cF$ is pseudo-coherent.\note{we need to check that $\R \epsilon_* r_*\cF$ is ps-coh. To do this it is enough to show that $H^i(\R \epsilon_* r_*\cF)$ are coherent and vanish when $i$ is large.but these are coherent and vanish iff their pullbacks under $\epsilon$ do, and we already computed that the pullbacks under $\epsilon$ are $H^i(r_*\cF)$.}

To prove (2), let $\cF \in \Dqc( \cC_{\liset})$ be perfect---by \cite[Tag~08G8]{tag} this is equivalent to pseudo-coherent and locally of finite tor dimension. By part (1) of this lemma and \cite[Tag~0CTL]{tag}, the pushforward $\R p_*\cF$ is pseudo-coherent.\note{this equates lisse-\'etale and etale topologies and ps-coh objects in them.} To see that $\R p_*\cF$ locally has finite tor dimension, by \cite[Tag~08EA]{tag}\note{uses qsep} it suffices to show that for $\cG \in \Qcoh(T)$ the sheaves $H^i(\R p_*\cF \lotimes \cG)$ vanish for $i$ outside a finite range.\note{really I'm proving that $\R \epsilon_* \R p_* \cF$ is perfect.} By the projection formula \cite[4.12]{HR17} and flatness of $p$ these are equal to $H^i(\R p_*(\cF \lotimes p^*\cG))$. Since $\cF$ is a perfect complex on a quasi-compact space, it has finite tor amplitude,\note{$\cF$ perfect means that it has locally finite tor amplitude, which means that for every $U \to \cX$ smooth there is a cover $U_i \to U$ such that $\cF|_{U_i}$ has finite tor dimension. If $\cX$ is quasi-compact, then we can take $U$ to be a cover and take the union of all the intervals $[a_i, b_i]$ to get a smooth cover $V \to \cX$ such that $\cF|_V$ has tor amplitude $[a,b]$. We show that $\cF$ has the same tor amplitude. Let $\cG$ be an $\OO_{\cX}$-module. Question is, does $H^i(\cF \lotimes \cG)$ vanish. The restriction functor is exact, so the restriction of $H^i(\cF\lotimes \cG)$ equals $H^i(\cF|_V \lotimes \cG|_V)$ and this vanishes (for $i \not \in [a,b]$). So the restriction to $V_{\et}$ vanishes. So by prop:descent-olsson the sheaf $H^i(\cF\lotimes \cG)$ vanishes.} 
so the spectral sequence
\[
\R^mp_*H^n(\cF) \implies \R^{m+n}p_*\cF
\]
of \cite[Tag~015J]{tag} and the fact that $p$ is concentrated finish the proof.\note{what's really happening: I want to show $\R \epsilon_* \R p_*\cF$ is perfect and I already know it is pscoh, so suffice to show $H^i(\R \epsilon_* \R p_*\cF \lotimes \cG)$ vanishes for qcoh $\cG$ and $i$ outside a finite range. To check if this sheaf vanishes, apply $\epsilon^*$ and get $H^i(\R p_*\cF \lotimes \epsilon^*\cG)$. happily $\epsilon^*\cG$ is still qcoh, which is important for applying the projection formula. For the projection formula and spectral sequence we stay in the lisse-\'etale topology.}
\end{proof}

\subsection{Background on right adjoint to pushforward}\label{sec:right-adj}
We recall some statements about right adjoint to pushforward that hold for purely formal reasons.

\begin{lemma}\label{lem:f!-exists}
Let $f: \cX \rightarrow \cY$ be a concentrated morphism of algebraic stacks. Then a right adjoint $f^!$ to $\R f_*:\Dqc(\cX_{\liset}) \to \Dqc(\cY_{\liset})$ exists, and for dualizable $\cG \in \Dqc(\cY_{\liset})$ the canonical morphism $f^*\cG \otimes f^!\OO_{\cY} \to f^!\cG$ defined in \eqref{eq:piso} is an isomorphism.
 Moreover, for $\cF \in \Dqc(\cX_{\liset})$ and $\cG \in \Dqc(\cY_{\liset})$ there is a functorial isomorphism
\begin{equation}\label{eq:sheaf-adjii}
\R f_*\Rhom_{\sO_{\cX}}^{\qc}(\cF, f^!\cG) \rightarrow \Rhom_{\sO_{\cY}}^{\qc}(\R f_*\cF, \cG).
\end{equation}
\end{lemma}
\begin{proof}
Existence of $f^!$ is \cite[Thm~4.14(1)]{HR17} and that \eqref{eq:piso} is an isomorphism follows from \cite[Prop~5.4]{FHM}. The isomorphism \eqref{eq:sheaf-adjii} is \cite[Prop~4.3]{FHM} (see also \cite[Tag~0A9Q]{tag}).
\end{proof}


We now explain what it means for $f^!$ to be compatible with basechange. While lemma \ref{lem:f!-exists} applies to arbitrary families of twisted curves $\cC \to \cM$, we will see that we need additional assumptions for the basechange property to hold.

Suppose we have a fiber square of algebraic stacks as below with $m$ and $f$ tor-indpendent (see \cite[Sec~4.5]{HR17}) and $f$ concentrated.
\begin{equation}\label{eq:duality10}
\begin{tikzcd}
\cX' \arrow[r, "m'"] \arrow[d, "g"]& \cX \arrow[d, "f"] \\
\cY' \arrow[r, "m"]& \cY
\end{tikzcd}
\end{equation}
Then $f^!$ and $g^!$ exist as recalled in Lemma \ref{lem:f!-exists}.
By \cite[Cor~4.13]{HR17}, the base change map \eqref{eq:basechange} is an isomorphism (we take the closed symmetric monoidal categories in \eqref{eq:basechange1} to be $\Dqc(\cX_{\liset})$, etc). This lets us define the functorial base change map $\L{m'}^*f^! \rightarrow g^!\L m^*$ to be the composition\note{this should agree with the basechange map in the stacks project. my idea for why: the maps $L{m'}^*f^! \rightarrow g^!Lm^*$ and $Rm'_*g^! \rightarrow f^!Rm_*$ are conjugate in the sense of Lipman lemma-definition 3.3.5. from here I hope you can find a commuting diagram.}
\begin{equation}\label{eq:basechangeii}
\L{m'}^*f^! \rightarrow g^!\R g_*\L{m'}^*f^! \xrightarrow{\text{\eqref{eq:basechange}}} g^!\L m^*\R f_*f^! \rightarrow g^!\L m^*.
\end{equation}
We are interested in when \eqref{eq:basechangeii} is an isomorphism.



\begin{lemma}\label{lem:f!-basechange}
Suppose we have a tor-independent fiber square \eqref{eq:duality10} of quasi-compact algebraic stacks with quasi-finite and separated diagonals,\note{quasi-finite implies quasi-compact (0G2M), so these stacks are in particular qcqs (04YW)} and suppose that $\cY'$ and $\cY$ are concentrated\note{if $\cY'$ and $\cY$ are qcqs algebraic spaces then they are concentrated} with quasi-affine diagonals.\note{I need BOTH diagonals q-affine becuase in the proof below, $\cY'$ plays the role of $\cY$ at one point.} If $f$ is concentrated and $\R f_*$ sends perfect complexes to perfect complexes, then \eqref{eq:basechangeii} is an isomorphism.\note{generalize to arbitrary characteristic: if the stacks are algebraic, need to add assumption that $\cY$ has quasi-affine diagonal. in fact i think I only need this lemma when the $\cX$'s are algebraic. I need a condition that guarantees that the derived category of $\cX$ has a single (or finite set) compact (? i have a hunch this is needed but I don't know why) perfect generator.}\todo{why is quasi-finite needed? why is tor-independent needed? in fact can I check where each hypothesis is used?}
\end{lemma}
\begin{remark}
The hypotheses of the lemma are satisfied if all the stacks in \eqref{eq:duality10}  are quasi-compact tame Deligne-Mumford with separated diagonals, with additional conditions on $f$ as above.\note{This is because the diagonal of $\cY$ is quasi-affine and you can do the cancellation theorem. The diagonal is quasi-affine because of 0418 (this lets you conclude that the diagonal is representable by schemes, not just alg spaces) and 02LR.}
\end{remark}
\begin{remark}
The preprint \cite{Nee17} proves that \eqref{eq:basechangeii} is an isomorphism under very general conditions. Compared with \cite{Nee17}, our Lemma \ref{lem:f!-basechange} imposes stricter conditions on the stacks $\cX$, $\cX'$, $\cY$, $\cY'$, and on the morphism $f$, but we allow $m$ to be arbitrary, whereas \cite{Nee17} requires $m$ to be flat.
\end{remark}

\begin{remark}\label{rmk:new-generators}
The proof of Lemma \ref{lem:f!-basechange} relies on our ability to find a compact generator for the algebraic stack $\cX$. By \cite[Thm~A]{HR17}, our assumptions that $\cX$ is quasi-compact with quasi-finite and separated diagonal imply that $\Dqc(\cX_{\liset})$ is compactly generated by a single perfect complex $\cP$. This means that for any $\cF \in \Dqc(\cX_{\liset})$, we have $\cF=0$ if and only if $\Hom_{\Dqc(\cX_{\liset})}(\cP[n], \cF)=0$ for every $n \in \ZZ$ (here, $\Hom_{\Dqc(\cX_{\liset})}$ denotes the hom-set in the (additive) category ${\Dqc(\cX_{\liset})}$). Since $\Dqc(\cX_{\liset})$ is a full subcategory of $\D(\cX_{\liset})$, we may compute the hom set in the larger category. But these hom sets are computed by the cohomology of the derived global hom functor. We conclude that for any morphism $f: \cF \rightarrow \cG$, we have that $f$ is an isomorphism if and only if $\R\mathrm{Hom}_{\sO_{\cX}}(\cP, f)$ is an isomorphism. 
\end{remark}

\begin{proof}[Proof of Lemma \ref{lem:f!-basechange}]
We explain why the proof of \cite[Cor~4.4.3]{lipman} also works in this setting. 

The first step is to reduce to the case where $m$ is quasi-affine. Indeed, by \cite[Prop~4.6.8]{lipman} the morphism \eqref{eq:basechangeii} satisfies a cocycle condition for squares stacked horizontally. This implies that it is enough to prove the Lemma when $\cY'$ is an affine scheme and $m$ is smooth, or when $\cY'$ and $\cY$ are both affine (see \cite[182-4]{lipman} for more details).
By assumption $\cY$ has quasi-affine diagonal,\note{tricky words! $\cY$ here might actually be $\cY'$.} so in either case the morphism $m$ is quasi-affine.

Now we assume $m$ is quasi-affine. Let $\cF \in \Dqc(\cY_{\liset})$ and let $\cP$ be a perfect, compact generator for $\Dqc(\cX_{\liset})$ (see Remark \ref{rmk:new-generators}). Since $\cY'$ and $\cY$ are concentrated by assumption, the morphisms $m$ and $m'$ are also concentrated \cite[Lem~2.5]{HR17} and we have functors $\R m_*$ and $\R m'_*$. To show that \eqref{eq:basechangeii} is an isomorphism, we claim that it suffices to show the induced map
\begin{equation}\label{eq:bc3}
\R f_*\R m'_* \Rhom_{\sO_{\cX'}}^{\qc}(\L m'^*\cP, \L m'^*f^! \cF) \rightarrow \R f_*\R m'_* \Rhom_{\sO_{\cX'}}^{\qc}(\L m'^*\cP, g^!\L m^* \cF)
\end{equation}
is an isomorphism. First, $\L {m'}^*\cP$ is perfect,\footnote{
To see this, note that the equivalence in Proposition \ref{prop:descent-olsson} sends perfect objects to perfect objects---one reason is because this is an equivalence of symmetric monoidal categories and the perfect objects are the dualizables \cite[Tag~0FPP]{tag}. Now apply \cite[Tag~08H6]{tag} to the morphism of strictly simplicial \'etale topoi.
}\note{Claim 1: if $E$ is perfect in $\cX_{\liset}$, then its restriction is perfect in $X^+_{\bullet, \et}$. Proof: the pullbac of perfect is perfect (for morphisms of topoi) so the pullbac of $E$ to $X^+_{\bullet, \liset}$ is perfect. Now the restriction (pushforward) to $X^+_{\bullet, \et}$ is perfect by the definition in 08G5: it is super important that covers in the lisse-etale site are the same as covers in the etale site. There are just more objects. 
}\note{pullback of perfect is perfect in lisse-etale topology? 
perfect is the same as being dualizable (0FPP), which is a category-theoretic notion---so dualizables in $\cX_{\liset}$ and $X^+_{\bullet, \et}$ coincide since the categories are equivalent. HELPFUL POInT: these categories are equivalent, and if I can show one of the equivalences is strong monoidal then so is the other, see \url{https://math.stackexchange.com/questions/183285/is-the-inverse-to-a-monoidal-equivalence-also-monoidal} or hall's paper p.20. SO the categories of qc sheaves on $\cX_{\liset}$ and $X^+_{\bullet, \et}$ are equivalent via strong symmetric monoidal functors.  but why is a dualizable object of $X^+_{\bullet, \et}$ represented by a perfect complex on $X^0$ (hence it is perfect on every level)? ah, because dualizable means locally dualizable (we have the dualizing data locally.) so it is dualizable on $X_0$ hence perfect. now usu. pullback of perfect is perfect, hence dualizable---it's all good.} so we may replace the functors $\Rhom_{\sO_{\cX'}}^{\qc}$ with
$\Rhom_{\sO_{\cX'}}$ (see Example \ref{ex:general}). Next, if \eqref{eq:bc3} is an isomorphism we get an isomorphism of global derived homs by applying the global sections
functor.\note{ideas are in Lipman 4.3.7, could add a reference to the stacks project definition of global derived hom}\note{-I'm just working WITHIN the lisse-etale topology, so this sentence is still just about derived categories of sites.} But $\L m'^*\cP$ is a perfect generator for
$\Dqc(\cX'_{\liset})$ by \cite[Cor~2.8]{HR17}---this is where we use that $m$ (hence $m'$) is quasi-affine. We conclude that \eqref{eq:basechangeii} is an isomorphism (see Remark \ref{rmk:new-generators}).

To show that \eqref{eq:bc3} is an isomorphism, we cite the bottom two cells of the commuting diagram on \cite[182]{lipman} to reduce to proving a certain morphism
\[
\R m_*\R g_*\Rhom_{\sO_{\cX'}}^{\qc}(\L m'^*\cP, \L m'^*f^!\cF) \xrightarrow{\R m_*(4.4.1)^*_{\mathrm{pc}}} \R m_*\Rhom_{\sO_{\cY'}}^{\qc}(\R g_*\L m'^*\cP, \L m^*\cF)
\]
is an isomorphism.\footnote{One may check that Lipman's discussion of the relevant commuting diagram here and in \eqref{eq:bc4} uses only formal properties of adjoint symmetric functors as discussed in \cite[Sec~3.5]{lipman}. The setup in \cite[Sec~3.5]{lipman} is compatible with our Situation \ref{basic-situation} by footnote \ref{ft:lipman}.} (In the cited diagram, the map notated $u^*\delta$ is, in our notation, equal to $\L m^*$ applied to the isomorphism \eqref{eq:sheaf-adjii}.) We will not bother to write the definition of $\R m_*(4.4.1)^*_{\mathrm{pc}}$ because \cite[Lem~4.6.4]{lipman} gives a commuting diagram
\begin{equation}\label{eq:bc4}
\begin{tikzcd}
\R f'_*\Rhom_{\sO_{\cX'}}^{\qc}(\L m'^*\cP, \L m'^*f^!cF) \arrow[r, "{(4.4.1)^*_{\mathrm{pc}}}"] & \Rhom_{\sO_{\cY'}}^{\qc}(\R f'_*\L m'^*\cP, \L m^*\cF)\\
\R f'_*\L m'^*\Rhom_{\sO_{\cX}}^{\qc}(\cP, f^!\cF) \arrow[u, "\rho"] & \Rhom_{\sO_{\cY'}}^{\qc}(\L m^*\R f_*\cP, \L m^*\cF) \arrow[u, "\text{\eqref{eq:basechange}}"]\\
\L m'^*\R f_*\Rhom_{\sO_{\cX}}^{\qc}(\cP, f^!\cF) \arrow[u, "\text{\eqref{eq:basechange}}"] \arrow[r, "\L m^*\text{\eqref{eq:sheaf-adjii}}"] & \L m^*\Rhom_{\sO_{\cY}}^{\qc}(\R f_*\cP, \cF) \arrow[u, "\rho"]
\end{tikzcd}
\end{equation}
The arrows labeled \eqref{eq:basechange} are isomorphisms by \cite[Cor~4.13]{HR17}. The arrows labeled $\rho$ are defined in \cite[(3.3)]{FHM}, and by \cite[Prop~3.2]{FHM} all instances in this diagram are isomorphisms since $\cP$ and $\R f_*\cP$ are perfect complexes by assumption. We note that our definition of $\rho$ agrees with the definition in \cite[(3.5.4.5)]{lipman} by \cite[Exercise~3.5.6(a)]{lipman}. Finally, in diagram \eqref{eq:bc4}, we know that \eqref{eq:sheaf-adjii} is an isomorphism; this concludes the proof.

\end{proof}

\subsection{Example of Situation \ref{situation}}
We realize Situation \ref{situation} as duality for families of twisted curves on Noetherian algebraic spaces. 

\begin{example}\label{ex:situation}
Let $p:\cC \rightarrow T$ be a family of twisted curves on a quasi-separated Noetherian algebraic space $T$.\note{why i need each property: i need qc (resp qs) so $\cC$ is qc (resp ABSOLUTE diagonal of $\cC$ is sep), so thm A applies. \textit{i need locally noetherian so that $\R p_*$ preserves perfect complexes.}} Let $\sC = \Dqc(\cC_{\et})$, $\sD = \Dqc(T_{\et})$, $f_* = \R p_*$, and $f^*=\L p^*$.\note{I have to use categories with qc cohomology because in the next paragraph I need to know that they are compactly generated.} We will write $p^*$ for $\L p^*$ since $p$ is flat (this is justified by \cite[(1.9)]{HR17}). The projection map \eqref{eq:projection} is an isomorphism by Lemma \ref{lem:properties-of-p}.

The right adjoint $p^!$ exists by Lemma \ref{lem:f!-exists}. Moreover, by \cite[Thm~A]{HR17}, the category $\Dqc(\cC_{\liset})$ is compactly detected by a single perfect complex. Since $\R p_*$ preserves perfect complexes (by Lemma \ref{lem:properties-of-p}) and perfect objects in $\Dqc(T_{\et})$ are also compact,\note{this is because an algebraic space (used to say affine scheme) is concentrated so perfect=compact, see \cite[4]{HR17}. an algebraic space has trivial stabilizer groups which are certainly affine, so an algebraic space is concentrated.} it follows from \cite[Thm~8.4]{FHM} and \cite[Lem~7.4]{FHM} that \eqref{eq:piso} is an isomorphism for all $Y$.

It remains to show that $p^!\OO_{T}$ is invertible. This follows from Lemma \ref{lem:f!-basechange} and Lemma \ref{lem:invertible} below.
\end{example}

\begin{lemma}\label{lem:invertible}
Let $p:\cC \rightarrow T$ be a family of twisted curves on a Noetherian affine scheme $T$. Then $p^!\OO_{T}$ is represented by a rank one locally free sheaf in degree -1.
\end{lemma}
\begin{proof}
Let $q:C\rightarrow T$ be the coarse moduli space of $\cC$ and let $r: \cC \rightarrow C$ be the coarse moduli map. 
By \cite[Tag~0E6P,~0E6R]{tag} we know $q^!\OO_T$ is invertible and supported in degree -1. In particular it is dualizable, so we have\note{Because $Rp_* = Rq_*Rr_*$} 
	\[p^!\OO_T = r^!q^!\OO_T =  r^*q^!\OO_T \otimes r^!\OO_C\]
	where the second equality uses \cite[Thm~8.4]{FHM} and the fact that $q^!\OO_T$ is dualizable. Hence to prove the lemma it suffices to show that $r^!\OO_C$ is invertible and supported in degree 0. 
	
	Let $\bar c \to C$ be a geometric point. By Lemma \ref{lem:local-curves} we have a local description of $\cC \to C$ near $\bar c$ given by the diagram \eqref{eq:local_duality_new}. Note that a right adjoint to pushforward exists for every horizontal map in \eqref{eq:local_duality_new}.
	It follows from \cite[Lem~0.1]{Nee17} that the pullback of $r^!\OO_C$ to $[V/\mu_n]$ is equal to $\tau^!\OO_U$ (note that \cite[Lem~0.1]{Nee17} applies since $U \to C$ is \'etale and $r_*$ preserves pseudo-coherent objects by Lemma \ref{lem:pushing-ps-per}).
	Since $[V/\mu_n] \to \cC$ is flat, the complex $r^!\OO_C$ is represented by a quasi-coherent sheaf if and only if $\tau^!\OO_U$ is, and by \cite[Tag05B2]{tag} (applied on strictly simplicial \'etale sites as in Proposition \ref{prop:descent-olsson}) $r^!\OO_C$ is invertible if and only if $\tau^!\OO_U$ is.
	\note{here's the argument. Let $f: \cX \to \cY$ be a fpqc morphism of algebraic stacks, $E$ a complex on $\Dqc(\cY_{\liset})$ such that the pullback to $\cX$ is an invertible sheaf in degree 0. Since $f$ is flat, $f^*$ is exact, and the cohomology sheaves of $E$ must vanish in every degree except 0. So assume $E$ is a quasi-coherent sheaf. Now let $U \to \cX$ and $ V \to \cY$ be smooth covers by schemes fitting in a commuting (not necessarily fibered) square, with $U \to B \times_{\cY} \cX$ smooth. We know $(f^*E)|_{U_{\bullet,\et}^+}$ is invertible, so $f_\bullet^*(E_{V_{\bullet, \et}^+})$ is invertible. But $f_{\bullet}$ is still fpqc (at every level), so by 05B2 $E_{V_{\bullet, \et}^+}$ is invertible (at every level). So $E$ is inverible.}
	
To compute $\tau^!\OO_U$, set $\rho = \tau \circ \sigma$ and we observe that we have an equality
	\[
	\rho^!\OO_U = \sigma^*\tau^!\OO_U \otimes \sigma^!\OO_{[V/\mu_n]}
	\]
	so it suffices to show that $\rho^!\OO_U$ and $\sigma^!\OO_{[V/\mu_n]}$ are both invertible and supported in degree 0. 
	In Lemma \ref{lem:compute-adjoints} below, we prove the statement about $\rho^!\OO_U$, as well as the statement that $pr^!\OO_V$ is a line bundle in degree 0, where $pr: \mu_n \times V \to V$ is the projection. The statement about $pr^!\OO_V$ is equivalent to the statement about $\sigma^!\OO_{[V/\mu_n]}$ by an arument identical to the one used in the previous paragraph.

\end{proof}

\begin{lemma}\label{lem:compute-adjoints}
The complexes $\rho^!\OO_U$ and $pr^!\OO_V$ are represented by line bundles supported in degree 0.
\end{lemma}

\begin{proof}
We use the statement of finite duality in \cite[Tag~0AX2]{tag}\note{see also \cite[Thm~4.14(3)]{HR17}}, which we translate to a statement about rings using \cite[Tag~06Z0]{tag}. These results imply that for a morphism of affine schemes $\Spec(B) \rightarrow \Spec(A)$, the image of $\OO_{\Spec(A)}$ under the right adjoint to pushforward is induced by the complex of $B$-modules\note{some details: uncomment out below
}
\begin{equation}\label{eq:compute!}
\Rhom_A(B, A).
\end{equation}

For $pr$, the relevant ring map is the diagonal $A \rightarrow \prod_{g \in G} A$, and $B = \prod_{g \in G} A$ is a free $A$-module so \eqref{eq:compute!} is supported in degree 0. One checks that there is an isomorphism $B \rightarrow \Rhom_A(B, A)$ given by sending $1_B$ to the projection to the identity factor.\note{More detail?
The map $pr$ is induced by the diagonal ring map $A \rightarrow \prod_{g \in G} A$. The complex $pr^!\OO_V$ is represented by the complex $
	\Rhom_{A}(\prod_{g \in G}A, A)$ of $(\prod_{g \in G}A)$-modules. Since $\prod_{g \in G} A$ is a free $A$-module this complex is concentrated in degree 0. Let $e_1^\vee: \prod_{g \in G} A$ denote projection to the factor indexed by the identity $1 \in G$. Then $
	\Rhom_{A}(\prod_{g \in G}A, A)$ is generated as an $(\prod_{g \in G}A)$-module by $e_1$, hence the map
	\[
	\prod_{g \in G}A \rightarrow \Rhom_{A}(\prod_{g \in G}A, A)
	\]
	is an isomorphism, since it is an isomorphism of $A$-modules.}
	
For $\rho$, let $R = A^G$, the ring of $G$-invariants. Lemma \ref{lem:local-curves} lists two possibilities for $A$. In case (1) the computation of \eqref{eq:compute!} is similar to that for $pr$ since in this case, $A$ is a free $R$-module with basis $1, x, \ldots, x^{r-1}$ and $
	\Rhom_{R}(A, R)$ is generated as an $A$-module by projection to the $x^{r-1}$-factor. 
	
The computation in case (2) is more involved since we have to take a free resolution of $A$. One may use the resolution
\[
 \ldots \xrightarrow{d_3} R^{\oplus 2r-2} \xrightarrow{d_2} R^{\oplus 2r-2} \xrightarrow{d_1} R^{\oplus 2r-1} \xrightarrow{d_0} A \rightarrow 0
\]
with maps given as follows. If $\{f_i, g_i\}_{i=1}^{r-1}$ denotes a free basis for $R^{\oplus 2r-2}$ and $e$ is the additional basis element of $R^{\oplus 2r-1}$, then $d_i$ is defined by
{\footnotesize
\[\begin{array}{llllll}
d_0:&e \mapsto 1 & \quad d_i,\;i\;\text{odd}:& f_i \mapsto vf_i-t^ig_{r-i} & \quad d_i,\;i>0\;\text{even:}&f_i \mapsto uf_i+t^ig_{r-i}\\
&f_i \mapsto x^i&&g_{r-i} \mapsto ug_{r-i}-t^{r-i}f_i&&g_{r-i}\mapsto t^{r-i}f_i+vg_{r-i}\\
&g_i \mapsto y^i &&&&
\end{array}
\]	
}
For details, see \cite[25-28]{webb-thesis}.
\end{proof}

\subsection{Example of Situation \ref{situation2}}\label{sec:dualizing}
We realize Situation \ref{situation2} for families of twisted curves on algebraic stacks. We use the dualizing sheaf and trace map (as in Situation \ref{weaker-situation}) as a substitute for the full duality in Example \ref{ex:situation} because we are unable to show that the basechange morphism \eqref{eq:basechangeii} is an isomorphism in general.

\begin{proposition}\label{prop:duality}
For every family $\cC \rightarrow \cM$ of twisted curves on a locally Noetherian algebraic stack $\cM$, there is a pair $(\omegabul_{\cM}, \tr_{\cM})$ with $\omegabul_{\cM} = \omega_{\cM}[1]$ where $\omega_{\cM} \in \Qcoh(\cC_{\liset})$ is locally free and $\tr_{\cM}: \R\pi_*\omegabul_{\cM} \rightarrow \OO_{\cM}$,  such that the following hold:
\begin{enumerate}
\item The pair is functorial in the following sense. Given a fiber square
\begin{equation}\label{eq:duality1}
\begin{tikzcd}
\cC_{\cN} \arrow[r, "m'"] \arrow[d]& \cC_{\cM} \arrow[d] \\
\cN \arrow[r, "m"]& \cM
\end{tikzcd}
\end{equation}
there is a canonical isomorphism
\begin{equation}\label{eq:omega-iso}
{m'}^*\omegabul_\cM \xrightarrow{\sim} \omegabul_{\cN}\end{equation}
such that the following square commutes:
\begin{equation}\label{eq:tr-functoriality1}
\begin{tikzcd}
\L m^*\R p_*\omegabul_{\cM} \arrow[r, "\L m^*\tr_{\cM}"] \arrow[d, "\text{\eqref{eq:basechange}}", "\sim"']&\OO_{\cN} \\
\R p_*\L  {m'}^*\omegabul_{\cM} \arrow[r, "\text{\eqref{eq:omega-iso}}", "\sim"'] & \R p_*\omegabul_{\cN}\arrow[u, "\tr_{\cN}"]
\end{tikzcd}
\end{equation}
Moreover, if $n: \cK \rightarrow \cN$ is a morphism of algebraic stacks and $\cC_{\cK} = \cC_{\cN} \times_{\cN} \cK$ is the pullback and $n': \cC_{\cK} \rightarrow \cC_{\cN}$ the projection, then the isomorphism $(m' \circ n')^*\omegabul_{\cM} \rightarrow \omega_{\cN}$ is equal to the composition $n'^*m'^*\omegabul_{\cM} \rightarrow n'^*\omegabul_{\cN} \rightarrow \omegabul_{\cK}.$
\item If $\cM$ is a quasi-separated Noetherian algebraic space, then $\omegabul_{\cM}=p^!\OO_{\cM}$ and $\tr_{\cM}$ is the counit of the $(\R p_*, p^!)$ adjunction.\note{reason I did the more complicated descent, the one for hypercovers:  I actually only have base change when $\cM$ is a qcqs alg space, and for the fiber product of qcqs to be qcqs I need the \textit{base} to also be qcqs---i.e. I would need to assume $\cM$ is quasi-compact, see 075S. The way to get around this is to build a special simplicial cover, a \textit{hypercover}, where each level is a disjoint union of affine schemes. but now this is not the kind of cover that Martin proved descent for. so that's a super bummer.}
\end{enumerate}
\end{proposition}
 For a general base $\cM$ we do not know if our construction of $(\omegabul_\cM, \tr_{\cM})$ agrees with the right adjoint to pushforward.
 
 \begin{remark}
 To see that Proposition \ref{prop:duality} gives an example of Situation \ref{situation2} compatible with Example \ref{ex:situation}, we use the fact that $\L m^*: \Dqc(\cM_{\liset}) \to \Dqc(\cN_{\liset})$ has a right adjoint even when $m$ is not concentrated; see \cite[Sec~1.3]{HR17}.
 \end{remark}

\begin{proof}[Proof of Proposition \ref{prop:duality}]
The idea as follows. We will define the pair $(\omegabul_{\cM}, tr_{\cM})$ when $\cM$ is an algebraic space as required by part (2) of the proposition. When $\cM$ is an algebraic stack, we will take this as the smooth-local definition of $(\omegabul_{\cM}, tr_{\cM})$, and using the notion of a very smooth hypercover explained in Appendix \ref{app:descent} we will show that these local pairs ``glue'' to a global one with the correct properties.

We now proceed with the proof. When $\cM$ is a quasi-separated Noetherian algebraic space,  we define $\omegabul_{\cM}$ and $tr_{\cM}$ as required in part (2) of the proposition (see Example \ref{ex:situation}). When both $\cN$ and $\cM$ are both quasi-separated Noetherian algebraic spaces, we define \eqref{eq:omega-iso} to be the base change map \eqref{eq:basechangeii} (it is an isomorphism by Lemma \ref{lem:f!-basechange}). The commuting diagram \eqref{eq:tr-functoriality1} follows from the definition of \eqref{eq:basechangeii}, see \cite[Rmk~4.4(d)]{lipman}. The cocycle condition on \eqref{eq:omega-iso} is \cite[Prop~4.6.8]{lipman}.

Let $\cM$ be a locally Noetherian algebraic stack. In this paragraph we define $\omega_\cM$. Let $M_\bullet \rightarrow \cM$ be a very smooth hypercover (see Definition \ref{def:very-smooth-hypercover}) and let $\cC_{M,\bullet}$ be its pullback to $\cC_{\cM}$ (see Remark \ref{rem:hyper-pullback}). We have associated categories of quasi-coherent sheaves $\Qcoh(M_{\bullet, \liset})$ and $\Qcoh(\cC_{M, \bullet, \liset})$ as in Section \ref{sec:descent-topos}. By Remark \ref{rmk:hyper-affine} we may assume that each $M_i$ is a disjoint union of affine schemes (each Noetherian by \cite[Tag~06R6]{tag}). In particular, each $M_i$ is a disjoint union of qcqs Noetherian schemes. For each $n \in \ZZ_{\geq 0}$ we have families of twisted curves $\cC_{M, n} \rightarrow M_n$, and hence the system of locally free sheaves $\omega_{M_n}$ (defined by applying the construction in the previous paragraph to the Noetherian components of $M_n$) together with the isomorphisms \eqref{eq:omega-iso} defines an object $\omega_{M, \bullet}$ of $\Qcoh(\cC_{M, \bullet, \liset})$. By Proposition \ref{prop:descent} the sheaf $\omega_{M, \bullet}$ corresponds to a unique quasi-coherent sheaf $\omega_\cM$ in $\Qcoh(\cC_{\cM, \liset})$ whose restriction to $\cC_{M_i}$ is $\omega_{M_i}$. Let $\omegabul_{\cM} = \omega_{\cM}[1].$

In this paragraph we define $tr_{\cM}$. By Remark \ref{rmk:descent} the complex $\R p_*\omegabul_{\cM}$ is represented by the element of $\Dqc(M_{\bullet, \liset})$ whose $n^{th}$ component is $\R p_*\omegabul_{M_n}$ (see also \cite[Tag~0D9P]{tag}). We have trace maps $\tr_{M_n}:\R p_*\omegabul_{M_n} \to \OO_{M_n}$ for each $n$, and these are compatible with the transition maps of $M_\bullet$ by \eqref{eq:tr-functoriality1}. Now from Proposition \ref{prop:descent} combined with the argument in \cite[Tag~0DL9]{tag} we obtain $tr_{\cM}: \R p_*\omegabul_{\cM} \rightarrow \OO_{\cM}$ (the required Ext groups vanish since $\R p_*\omegabul_{\cM}$ is a complex in degrees [-1,0] by Lemma \ref{lem:properties-of-p}).\note{note that \cite[Tag~0D9S]{tag} and \cite[Tag~0D9T]{tag} apply to $X_{\bullet, \liset}$} 

Now we check that the pair $(\omega_{\cM}, tr_{\cM})$ has the properties required in part (1) of the proposition. Suppose we have a fiber square \eqref{eq:duality1} where $\cN$ and $\cM$ are algebraic stacks. Let $N_\bullet \to \cN$ and $M_\bullet \to \cM$ be very smooth hypercovers with $M_i$ and $N_i$ disjoint unions of affine schemes, with a morphism $N_\bullet \to M_\bullet$ commuting with the augmentations and $m:\cN \to \cM$ (see Remark \ref{rem:hyper-double-affine}). Let $\cC_{M,\bullet}$ and $\cC_{N,\bullet}$ be the pullbacks of $M_{\bullet}$ and $N_{\bullet}$ to $\cC_{\cM}$ and $\cC_{\cN}$, respectively.
For each $n \in \ZZ_{\geq 0}$, the twisted curve $\cC_{N, n} \rightarrow N_n$ is the pullback of $\cC_{M, n} \rightarrow M_n$, and we have isomorphisms $m'^*_n\omega_{M_n} \xrightarrow{\text{\eqref{eq:omega-iso}}}\omega_{N_n}$. 
Under the identifications $(a^*m'^*\omega_{\cM})|_{M_n} \simeq m'^*_n\omega_{M_n}$ of Remark \ref{rmk:descent}, these isomorphisms are compatible with the transition maps for the
sheaves $a^*{m'}^*\omega_{\cM}$ and $a^*\omega_{\cN}$ in $\Qcoh(\cC_{N, \bullet, \liset})$ because \eqref{eq:omega-iso} satisfies the cocycle condition. By descent we get an isomorphism $m'^*\omegabul_{\cM} \rightarrow \omegabul_{\cN}$. To check that this definition makes \eqref{eq:tr-functoriality1} commute, apply the equivalences $a^*$ and use Remark \ref{rmk:descent} to get a collection of commuting diagrams indexed by $n \in \ZZ_{\geq 0}$.

\end{proof}


\section{Obstruction theories via the Fundamental Theorem}\label{sec:main-proof}
\subsection{Some Picard categories}\label{sec:picard-stacks}
Let $\cS$ be a site. We recall the notion of Picard stacks from \cite[Sec~XVIII.1.4.5]{SGA4}, and observe that a Picard category is just a Picard stack on the punctual site (see also \cite[Def~XVIII.1.4.2]{SGA4}). If $f: \cP \rightarrow \cQ$ is a morphism of Picard stacks on $\cS$, we define the \textit{kernel} to be the fiber product $\cK = \bullet \times_{e, \cQ,f} \cP$ where $\bullet$ is the trivial Picard stack (a constant sheaf with all its fibers equal to a single point), and $e: \bullet \rightarrow \cQ$ is the identity. 


\begin{example}\label{ex:ch}
Let $\D(\cS)$ be the unbounded derived category of abelian sheaves on $\cS$. As in \cite[Sec~XVIII.1.4.11]{SGA4} we have a functor $ch$ from the subcategory $\D^{[-1,0]}(\cS)$ to the category of Picard stacks on $\cS$ (in the latter category, arrows are isomorphism classes of morphisms of stacks).
Suppose $A$ is a sheaf of rings on $\cS$ and $\D(A)$ is the unbounded derived category of sheaves of $A$-modules. For two complexes $F \in \D^{[-\infty, a]}(A)$ and $G\in \D^{[a-1, \infty]}(A)$,\note{confused here: do I want $G \in \D^{[a-1, \infty]}(A)$? because I might want to take an injective resolution, too...and i want the output of $\Rhom$ to be in degrees -1 and above.} we define
\begin{equation}\label{eq:def-ext}
\Estackbase{F}{G}{A} := ch(\tau_{\leq 0}\Ghom_A(F, G)) = ch(\tau_{\leq 0}\R \Gamma\Rhom_A(F, G)) 
\end{equation}
where $\Ghom_A$ is derived global hom for $\D(A)$ and we have omitted the pushforward from the derived category of $\Gamma(\sS, A)$-modules to the category of abelian groups. Observe that $ch$ is applied here over the site with one object and one morphism, so $\Estackbase{F}{G}{A}$ is actually a Picard category (and the prestack $pch(\tau_{\leq 0}\R \Gamma \Rhom_A(F, G)$ of \cite[Sec~XVIII.1.4.11]{SGA4} is actually a stack). \note{where did i want to use that $pch$ is a stack? I don't think that's quite true. But if I replace $F$ with flat and $G$ iwth injective, so $\Rhom$ is injective, applying $\R\Gamma$ leaves it injective. On the other hand, if the site is trivial, then what can stackification do?? isn't every prestack a stack in this case? SGA4 doesn't say you \textit{have} to have an injective object for $pch$ to be a stack, only that an injective object guarantees a stack. there could be other ways to get a stack i guess.} If the ring $A$ is clear we will omit it from the notation.\note{this is equivalent to the global fiber of what I had before. I mean that $ch( K)(\sS) = ch( \tau_{\leq 0}\R \Gamma K)$.
Proof: by 1.4.16 if $K^{-1}$ is injective then $pch$ is a stack. So replace $K^{-1}\to K^0$ with a complex $I^{-1}\to I^0 \xrightarrow{d} I^1 \to \ldots$ of injectives and then truncate to get a complex $I^{-1} \to \ker(d)$ that is qiso to
$K^{-1}\to K^0$ but its first term is injective. Now the objects of $ch(K)(U)$ are just $ker(d)(U)$ and arrows are $I^{-1}(U)$. On the other hand, to compute $\R\Gamma(K)$ apply $\Gamma$ to the
injective resolution $I^\bullet$. The truncation to two terms is $I^{-1}(U) \to \ker(d|_U))$. Now I guess $I^{-1}(U)$ is still injective, so $pch=ch$ for this complex, and in particular its objects are $\ker(d|_U)$. The point is that $\ker(d|_U) = \ker(d)(U)$ because global sections is left exact.}
\note{If you compute Rhom over $\ZZ$ instead of over $\OO$, you WILL get a different sheaf, and hence a different Picard stack, I think.} It follows from \cite[(XVIII.1.4.11.1)]{SGA4} that isomorphism classes of objects of $\Estackbase{F}{G}{A}$ are equal to $\Ext^0_A(F, G)$ and from \cite[(XVIII.1.4.11.2)]{SGA4} that automorphisms of the identity element are $\Ext^{-1}_A(F, G)$.
\end{example}

\begin{example}\label{ex:exal}
Let $\cX \rightarrow \cY$ be a representable morphism of algebraic stacks and let $I$ be a quasi-coherent sheaf on $\cX$. We recall from \cite[Sec~2.2, 2.12]{olsson-deformation} the Picard category $\Exal_{\cY}(\cX, I)$ on $\cX_{\et}$: objects are square-zero extensions $\cX \hookrightarrow \cX'$ of stacks over $\cY$, together with an isomorphism $I \rightarrow \ker(\OO_{\cX'} \rightarrow \OO_\cX)$ (see \cite[Sec~2.2]{olsson-deformation} for details, e.g. arrows). 

Now suppose we have the following commuting diagram of algebraic stacks where $q:\cX \hookrightarrow \cX'$ is a square-zero extension by a quasi-coherent sheaf $I$, the maps $f$ and $g$ are representable, and we have fixed 2-morphism $\gamma: r\circ f \rightarrow g \circ q$.
\begin{equation}\label{eq:defdiagram}
\begin{tikzcd}
\cX \arrow[r, "f"] \arrow[d, "q"'] & \cY \arrow[d, "r"] \\
\cX' \arrow[r, "g"] & \cZ
\end{tikzcd}
\end{equation}
The morphism $r$ induces a morphism $\uR: \Exal_{\cY}(\cX, I) \to \Exal_{\cZ}(\cX, I)$, and the perimeter of \eqref{eq:defdiagram} defines an element of $\Exal_\cZ(\cX, I)$ (i.e. a functor $\bullet \to \Exal_\cZ(\cX, I)$, where $\bullet$ is the groupoid with one object and one arrow).
We define the Picard category $\Def(f)$ to be the fiber product
\begin{equation}\label{eq:fiber1}
\begin{tikzcd}
\Def(f) \arrow[r] \arrow[d] & \Exal_{\cY}(\cX, I) \arrow[d, "\uR"] \\
\bullet \arrow[r] & \Exal_{\cZ}(\cX, I)
\end{tikzcd}
\end{equation}
where the bottom arrow $\bullet \to \Exal_{\cZ}(\cX, I)$ is the section induced by \eqref{eq:defdiagram}.
We use $\Defset(f)$ to denote the set of isomorphism classes of $\Def(f)$. Explicitly, objects of $\Def(f)$ are triples $(k, \epsilon, \delta)$ such that $k: \cX' \rightarrow \cY$ is a 1-morphism, and $\epsilon: f \rightarrow k \circ q$ and $\delta: r \circ k \rightarrow g$ are 2-morphisms satisfying $q^*(\delta) \circ r(\epsilon) = \gamma$. A morphism from $(k_1, \epsilon_1, \delta_1)$ to $(k_2, \epsilon_2, \delta_2)$ is a natural transformation $\tau: k_1 \rightarrow k_2$ such that $q^*(\tau)\circ \epsilon_1 = \epsilon_2$ and $\delta_1=\delta_2 \circ r(\tau)$ (for details see \cite[Lem~2.4.3]{webb-thesis}).
\end{example}

\begin{example}\label{ex:triv-diagram}
As an example of the diagram \eqref{eq:defdiagram}, let $\cX \xrightarrow{f} \cY \xrightarrow{r} \cZ$ be morphisms of algebraic stacks with $\cX$ an algebraic space, and let $I \in \Qcoh(\cX_{\et})$. Define $q: \cX \rightarrow \cX'$ to be the trivial extension by $I$, so we have $q': \cX' \rightarrow \cX$ such that $q' \circ q = 1_{\cX}$. Now $g := r \circ f \circ q'$ is representable, and $k = f \circ q'$ defines an element of $\Def(f)$.
\end{example}
\subsection{The Fundamental Theorem}
The fundamental property of the cotangent complex is that it provides a description of the Picard category in Example \ref{ex:exal} in terms of the construction in Example \ref{ex:ch}.

\begin{theorem}[{\cite{olsson-deformation}}]\label{thm:ft}Let $\cX \rightarrow \cY$ be a representable morphism of algebraic stacks. Then
there is an isomorphism of Picard categories:
\begin{equation}\label{eq:ft}
\Exal_{\cY}(\cX, I) \xrightarrow{\sim} \Estackbase{\LL_{\cX/\cY}}{I[1]}{\OO_{\cX}}
\end{equation}
\end{theorem}
The definition of \eqref{eq:ft} is technical and we defer it to Section \ref{sec:defiso}. For us, the key property of \eqref{eq:ft} is that it is functorial under pullback and basechange as stated in the next two lemmas. 

\begin{lemma}\label{lem:ft-functoriality1}
Suppose we have maps $\cZ \xrightarrow{f} \cW \xrightarrow{g} \cY$ with $f$ and $g \circ f$ representable. Then given $I \in \Qcoh(\cZ_{\liset})$, there is a commuting diagram of Picard categories:
\begin{equation}\label{eq:wrap2}
\begin{tikzcd}
\Estack{\LL_{\cZ/\cW}}{ I[1]} \arrow[r, "A"] & \Estack{\LL_{\cZ/\cY}}{I[1]} \\
\Exal_{\cW}(\cZ, I) \arrow[r, "B"] \arrow[u, "{\text{\eqref{eq:ft}}}"]& \Exal_{\cY}(\cZ, I)\arrow[u, "{\text{\eqref{eq:ft}}}"']
\end{tikzcd}
\end{equation}
Here $A$ is induced by the canonical map $\LL_{\cZ/\cY} \rightarrow \LL_{\cZ/\cW}$ and $B$ is induced by composition with $g$.
\end{lemma}

Lemma \ref{lem:ft-functoriality1} is a special case of \cite[(2.33.3)]{olsson-deformation}, but that result is stated only for isomorphism classes of objects. We will prove Lemma \ref{lem:ft-functoriality1} in Appendix \ref{sec:ft-functoriality}.
For the second functoriality lemma, suppose we have a fiber square of algebraic stacks
\begin{equation}\label{eq:ft1}
\begin{tikzcd}[cramped]
\cZ \arrow[r, "p"] \arrow[d] & \cX \arrow[d]\\
\cW \arrow[r] & \cY
\end{tikzcd}
\end{equation}
where the map $\cW \rightarrow \cY$ is flat and $\cX \rightarrow \cY$ is representable. Then given a quasi-coherent sheaf $I \in \Qcoh(\cX_{\liset})$, there is a morphism of Picard categories
\begin{equation}\label{eq:ft2}
\Exal_\cY(\cX, I) \rightarrow \Exal_\cW(\cZ, p^*I)
\end{equation}
sending $\cX'\rightarrow \cY$ to the pullback $\cZ':= \cX' \times_\cY \cW \rightarrow \cW$ (observe that, since \eqref{eq:ft1} is fibered, we have an induced map $\cZ \hookrightarrow \cZ'$ with the desired kernel).
\note{I need to assume $a$ is flat because otherwise I would need to use derived pullback $La^*I$ in the definition of Exal, and I don't know how to interpret the right hand side when $I$ is a complex in negative degrees. For simplicity I'm just assuming the bottom map is flat.}

\begin{lemma}\label{lem:ft-functoriality2}
Given the fiber square \eqref{eq:ft1} and $I \in \Qcoh(\cX_{\liset})$, there is a commuting diagram of Picard categories:
\begin{equation}\label{eq:idunno}
\begin{tikzcd}
\Estack{\LL_{\cX/\cY}}{I[1]} \arrow[r, "C"] & \Estack{p^*\LL_{\cX/\cY}}{p^*I[1]} & \Estack{\LL_{\cZ/\cW}}{p^*I[1]} \arrow[l, "\sim", "D"'] \\
\Exal_{\cY}(\cX, I) \arrow[u, "\text{\eqref{eq:ft}}"] \arrow[rr, "E"] && \Exal_{\cW}(\cZ, p^*I)\arrow[u,"\text{\eqref{eq:ft}}"]
\end{tikzcd}
\end{equation}
Here $C$ is induced by \eqref{eq:sheafy-pullback} in the context of Example \ref{ex:change-topoi}\footnote{Example \ref{ex:change-topoi} differs from Example \ref{ex:general} because it uses general sheaves of $\OO$-modules and hence the $\Rhom$ functor instead of $\Rhom^{\qc}$.}, the arrow $D$ is induced by the canonical map of cotangent complexes (an isomorphism in this case), and $E$ is \eqref{eq:ft2}.
\end{lemma}

We will prove Lemma \ref{lem:ft-functoriality2} in Appendix \ref{sec:ft-functoriality}.
We conclude this section with a corollary to Theorem \ref{thm:ft} that may be read as a relative version of the same theorem.

\begin{corollary}\label{cor:ft2}
Consider a diagram \eqref{eq:defdiagram} of algebraic stacks where $\cX \rightarrow \cX'$ is a square-zero extension with ideal sheaf $I$, and $f$ and $g$ are representable.
\begin{enumerate}
\item There is an obstruction $o(f) \in \Ext^1(\L f^*\LL_{\cY/\cZ}, I)$ whose vanishing is necessary and sufficient for the set $\Defset(f)$ to be nonempty.
\item If $o(f)=0$, then there is an isomorphism $\Def(f) \simeq \Estack{\L f^*\LL_{\cY/\cZ}}{I}$.
\end{enumerate}
\end{corollary}
\begin{remark}
It follows from the Corollary that if $o(f)=0$, we get an isomorphism of groups between $\Ext^{-1}(\L f^*\LL_{\cY/\cZ})$ and the automorphism group of any element of $\Def(f)$. One can extract from the proof of the Corollary that $\Defset(f)$ is a torsor for $\Ext^0(\L f^*\LL_{\cY/\cZ}, I).$
\end{remark}
\begin{proof}

Applying Lemma \ref{lem:ft-functoriality1} to the maps $
\cX \rightarrow \cY \xrightarrow{r} \cZ$
we get a commuting diagram
\begin{equation}\label{eq:special2}
\begin{tikzcd}
\Estack{\LL_{\cX/\cY}}{I[1]} \arrow[r] \arrow[d, "\sim"] & \Estack{\LL_{\cX/\cZ}}{I[1]} \arrow[d, "\sim"]\\
\Exal_{\cY}(\cX, I) \arrow[r, "\uR"] &\Exal_{\cZ}(\cX, I)
\end{tikzcd}
\end{equation}
where $\uR$ is the same as the map $B$ in the Lemma.
When we restrict \eqref{eq:special2} to isomorphism classes of objects, we get the commuting square in the diagram below.
\begin{equation}\label{eq:defob}
\begin{tikzcd}
\Ext^1(\LL_{\cX/\cY},I) \arrow[r] \arrow[d, "\sim"] & \Ext^1(\LL_{\cX/\cZ},I) \arrow[d, "\sim", "\alpha"'] \arrow[r, "ob"] & \Ext^1(\L f^*\LL_{\cY/\cZ}, I)\\
\Exalset_{\cY}(\cX, I) \arrow[r, "R"]& \Exalset_{\cZ}(\cX, I)
\end{tikzcd}
\end{equation}
The top row of the diagram comes from applying $\Ext^1(-, I)$ to the distinguished triangle
\begin{equation}\label{eq:special0}
\L f^*\LL_{\cY/\cZ} \rightarrow \LL_{\cX/\cZ} \rightarrow \LL_{\cX/\cY}.
\end{equation}
The set $\Defset(f)$ is nonempty if and only if, in \eqref{eq:defob}, the fiber of $R$ over the element $[g]\in \Exalset_{\cZ}(\cX, I)$ defined by \eqref{eq:defdiagram} is nonempty. From the long exact sequence for $\Ext^i(-, I)$ applied to \eqref{eq:special0}, we see that this happens if and only if the image of $[g]$ in $\Ext^1(\L f^*\LL_{\cY/\cZ}, I)$ (under the maps given in \eqref{eq:defob}) is 0. We define
\begin{equation}\label{eq:special3}
o(f) = ob(\alpha^{-1}([g])).
\end{equation}

If $\Defset(f)$ is not empty, then by Lemma \ref{lem:pic-a-kernel} below $\Def(f)$ is isomorphic to the kernel of the morphism of Picard categories
\[
\uR: \Exal_{\cY}(\cX, I) \rightarrow \Exal_{\cZ}(\cX, I).
\]
It follows from \eqref{eq:special2} and \cite[Lem~2.29]{olsson-deformation} applied to the distinguished triangle
\[
\Ghom(\LL_{\cX/\cZ}, I[1]) \xrightarrow{\beta} \Ghom(\LL_{\cX/\cY}, I[1]) \rightarrow \Ghom(\L f^*\LL_{\cY/\cZ}, I[1]) \rightarrow
\]
induced from \eqref{eq:special0} that this kernel is canonically isomorphic to $ch((\tau_{\leq -1}Cone(\tau_{\leq 0}\beta))[-1])$, where $Cone$ denotes the mapping cone of a morphism. But we compute
\[(\tau_{\leq -1}Cone(\tau_{\leq 0}\beta)))[-1] = (\tau_{\leq -1}Cone(\beta))[-1]=\tau_{\leq 0}Cone(\beta[-1])\] 
so we get that this kernel is isomorphic to $\Estack{\L f^*\LL_{\cY/\cZ}}{I}$.
\end{proof}

\begin{lemma}\label{lem:pic-a-kernel}
Let $f: \cP \rightarrow \cQ$ be a morphism of Picard stacks on a stack $\cX$, and let $\cK$ denote the kernel. Let $q: \cX \rightarrow \cQ$ be a section and $\cF = \cP \times_{\cQ, q} \cX$ the fiber product. If the set of global objects of $\cF$ is not empty, then $\cF$ is non-canonically isomorphic to $\cK$.
\end{lemma}
\begin{proof}
A global object of $\cF$ defines a section $\sigma: \cX \rightarrow \cF$. One can check that the composition $\cK \times_{\cX} \cF \rightarrow \cP \times_{\cX} \cP \xrightarrow{\mu} \cP$, where $\mu$ is the group operation, factors through $\cF$. We obtain a morphism
\begin{equation}\label{eq:pica1}
\cK \xrightarrow{(1_{\cK}, \sigma)} \cK \times_{\cX} \cF \rightarrow \cF.
\end{equation}
On the other hand, we have the composition 
\begin{equation}\label{eq:pica2}
\cF \xrightarrow{(pr_1, -\sigma)} \cP \times_{\cX} \cP \xrightarrow{\mu} \cP \xrightarrow{f} \cQ
\end{equation}
where $pr_1: \cF \rightarrow \cP$ is the canonical morphism and $-\sigma$ is $\sigma$ followed by the inverse morphism. The composition \eqref{eq:pica2} factors through the identity $e: \cX \rightarrow \cQ$, so we get an induced map $\cF \rightarrow \cK$. One may check that this is inverse to \eqref{eq:pica1}.
\end{proof}
\subsection{Equivalent definitions of an obstruction theory}\label{sec:def-ot}
Let $\cY \rightarrow \cZ$ be a morphism of algebraic stacks. If $\phi: E \rightarrow F$ is a morphism in $\Dqc(\cY_{\liset})$, let $H^i(\phi): H^i(E) \rightarrow H^i(F)$ denote the induced morphism on cohomology sheaves. The following definition generalizes \cite[Def~4.4]{BF97}.

\begin{definition}
A morphism $\phi: E \rightarrow \LL_{\cY/\cZ}$ in $\Dqc(\cY_{\liset})$ is an \textit{obstruction theory} if $H^{-1}(\phi)$ is a surjection and $H^0(\phi)$, $H^1(\phi)$ are isomorphisms. 
\end{definition}

Given a morphism $\phi: E \rightarrow \LL_{\cY/\cZ}$ in $\Dqc(\cY_{\liset})$, for every diagram \eqref{eq:defdiagram} we have induced homomorphisms of groups (computed a priori in the lisse-\'etale topology)
\begin{equation}\label{eq:Phi}
\Phi_i: \Ext^i(\L f^*\LL_{\cY/\cZ}, I) \rightarrow \Ext^i(\L f^*E, I).\end{equation} 
We now present a well-known local criterion for a morphism $\phi$ to be an obstruction theory. Similar criteria have appeared in \cite[Thm~4.5]{BF97}, \cite[Cor~8.5]{AP19} and \cite[Thm~3.5]{poma}. However, we found the wording in these criteria to be vague in that they do not explicitly require compatibility between various morphisms. Since proving said compatibility is a major part of the paper (it comprises the functoriality computations in Appendix \ref{sec:ft-functoriality}), we give the precise statement of the local criterion and a fully detailed proof.
\begin{lemma}\label{lem:BF}
The following conditions are equivalent.
\begin{enumerate}
\item The morphism $\phi$ is an obstruction theory.
\item For every diagram \eqref{eq:defdiagram} with $\cX$ a scheme, the following hold:
\begin{enumerate}
\item the element $\Phi_1(o(f)) \in \Ext^1(\L f^*E, I)$ vanishes if and only if $\Def(f)$ is nonempty
\item if $\Phi_1(o(f))=0$ then $\Phi_0$ and $\Phi_{-1}$ are isomorphisms
\end{enumerate}
\item For every affine scheme $\cX$, and smooth map $\cX \rightarrow \cZ$, the following hold:
\begin{enumerate}
\item  for every ambient diagram \eqref{eq:defdiagram} using $\cX$, the element $\Phi_1(o(f)) \in \Ext^1(\L f^*E, I)$ vanishes if and only if $\Def(f)$ is nonempty
\item for every $I \in \Qcoh(\cX_{\liset})$, the maps $\Phi_0$ and $\Phi_{-1}$ are isomorphisms
\end{enumerate}
\end{enumerate}
\end{lemma}

\begin{remark}
In Lemma \ref{lem:BF}, conditions (2) and (3) may be computed in $\cX_{\et}$---so in (3b), one checks every $I \in \Qcoh(\cX_{\et})$ (see e.g. \cite[Prop 9.2.16]{olsson-book}).
\end{remark}


\begin{proof}[Proof of Lemma \ref{lem:BF}]
The proof of this lemma seems to be well-known; many parts were explained to me by Bhargav Bhatt.
Let $C$ be the mapping cone of $\phi: E \rightarrow \LL_{\cY/\cZ}$. Then condition (1) is equivalent to
\[
 \text{(1')}\;\;\; H^i(C)=0 \;\text{for} \;i\geq -1.
\]

Assume (1'). Then $H^i(\L f^*C)$ also vanish for $i\geq -1$, so a spectral sequence \cite[Tag~07AA]{tag} for $\Ext^i(-, I)$ implies $\Ext^i(\L f^*C, I)=0$ for $i \leq 1$ and any $I$.\note{in general if $h^i(A^\bullet)=0$ for $i\geq a$ then $Ext^i(A^\bullet, F)=0$ for $i\leq -a$, as should be believable by thinking about degrees of homs} Now the long exact sequence of Ext groups arising from the distinguished triangle
\begin{equation}\label{eq:mydt}
\L f^*E \rightarrow \L f^*\LL_{\cY/\cZ} \rightarrow \L  f^*C \rightarrow
\end{equation}
implies that $\Phi_1$ is injective and $\Phi_0$ and $\Phi_{-1}$ are isomorphisms. Combined with Corollary \ref{cor:ft2}, this proves (2) (with $\cX$ an arbitrary scheme).
Now (2) implies (3) using Example \ref{ex:triv-diagram}.

Assume (3). Condition (1') may be checked smooth-locally on $\cY$, so let $f: \cX \rightarrow \cY$ be a smooth morphism from an affine scheme and let $I \in \Qcoh(\cX_{\liset})$ be arbitrary. We will show that if $i \geq -1$ then $\Ext^0(H^i( \L f^*C), I)=0$, which implies $f^* H^i(C) = H^i(\L f^*C) =0$ (the first equality is \cite[(1.9)]{HR17} and uses flatness of $f$). 

By assumption (3b), the morphisms  $\Phi_0$ and $\Phi_{-1}$ are isomorphisms. We show that $\Phi_1$ is injective. It follows from Corollary \ref{cor:ft2} and assumption (3a) that if $\Phi_1(o(f))=0$ then $o(f)=0$, so it suffices to show that every element of $\Ext^1(\L f^*\LL_{\cY/\cZ}, I)$ is equal to $o(f)$ for some diagram \eqref{eq:defdiagram}, or equivalently that the map $ob$ in \eqref{eq:special3} is surjective. This follows from the long exact sequence
\[
\rightarrow \Ext^1(\LL_{\cX/\cZ}, I) \xrightarrow{ob} \Ext^1(\L f^*\LL_{\cY/\cZ}, I) \rightarrow \Ext^2(\LL_{\cX/\cY}, I)\rightarrow
\]
 since $\LL_{\cX/\cY}=\Omega^1_{\cX/\cY}[0]$ is a locally free sheaf in degree 0.

Since $\Phi_1$ is injective and $\Phi_0$ and $\Phi_{-1}$ are isomorphisms, the long exact sequence of Ext groups for \eqref{eq:mydt} shows that $\Ext^i(\L f^*C, I)=0$ for every $i\leq 1$. By \cite[Tag~07AA]{tag} there is a spectral sequence whose second page is 
\[
\Ext^i(H^{-j}(\L f^*C), I) \implies \Ext^{i+j}(\L f^*(C), I).
\]
A priori we know $H^i(\L f^*C)=0$ for $i \geq 2$. By the above spectral sequence, the group $\Ext^0(H^1(\L f^*C), I)$ is equal to $\Ext^{-1}(\L f^*C, I)$ which vanishes for every $I$. This forces $H^1(\L f^*C)$ to vanish. Inductively applying the same argument to $\Ext^0(\L f^*C, I)$ and then $\Ext^1(\L f^*C, I)$ shows that $H^0(\L f^*C)$ and $H^{-1}(\L f^*C)$ vanish as well.

\end{proof}

\subsection{Moduli of sections}\label{sec:moduli-of-sections}

Consider a tower of algebraic stacks
\[
\cZ \rightarrow \cC \xrightarrow{p} \cM
\]
as in Section \ref{sec:intro}. There we defined the moduli of sections $\Sec_{\cM}({\cZ}/{\cC})$. By \cite[Thm~1.3]{HR19} and our assumption that $\cM$ is locally Noetherian, the stack $\Sec_{\cM}({\cZ}/{\cC})$ is also locally Noetherian.
The stack $\Sec_{\cM}({\cZ}/{\cC})$ has a universal curve $\cC_{\Sec_{\cM}(\cZ/\cC)}$ and a universal section $f_{\Sec_{\cM}(\cZ/\cC)}\in \Hom_{\cC}( \cC_{\Sec(\cZ)} ,\cZ)$ (we will omit the subscript on $f$ when possible).

Now suppose we have a tower of algebraic stacks
\[
\cZ \rightarrow \cW \rightarrow \cC \xrightarrow{p} \cM
\]
where $\cZ, \cC$, and $\cM$ are as before and $\cW \rightarrow \cM$ is locally finitely presented, quasi-separated, and has affine stabilizers. To simplify the notation, let $\SZ := \Sec_{\cM}({\cZ}/{\cC})$ and $\SW:=\Sec_{\cM}({\cW}/{\cC}) $. We have an induced map $\SZ \rightarrow \SW$, and over this map we have a canonical relative obstruction theory defined as follows. We have a morphism in $\Dqc(\cC_{\SZ})$ consisting of canonical morphisms of cotangent complexes:
\begin{equation}\label{eq:build1}
\L f^*\LL_{\cZ/\cW} \rightarrow \LL_{\cC_{\SZ}/\cC_{\SW}} \xleftarrow{\sim} p^*\LL_{\SZ/\SW}.
\end{equation}
Using the pair $(\omegabul_{\SZ}, tr_{\SZ})$ defined in Proposition \ref{prop:duality}, we may apply the adjunction-like morphism $a$ defined in Section \ref{sec:modification} to \eqref{eq:build1}, obtaining
\begin{equation}\label{eq:generalpot}
\phi_{\SZ/\SW} : \EE_{\SZ/\SW} \rightarrow \LL_{\SZ/\SW},  \quad \EE_{\SZ/\SW}:= \R p_*(\L f^*\LL_{\cZ/\cW} \otimes \omegabul_{\SZ}).
\end{equation}
For example, when $\cW=\cC$ we have $\SW = \cM$ and we obtain an obstruction theory on $\SZ$ relative to $\cM$. We refer the reader to \cite[Appendix~A]{CJW} for functoriality properties of $\Sec_{\cM}(\cZ/\cC)$ and the obstruction theories \eqref{eq:generalpot}.

The main theorem of this article is the following.
\begin{theorem}\label{thm:main}
The morphism \eqref{eq:generalpot} is an obstruction theory.
\end{theorem}

\subsubsection{Proof of Theorem \ref{thm:main}}
We prove condition (3) of Lemma \ref{lem:BF}. To begin, fix a solid commuting diagram
\begin{equation}\label{eq:main1}
\begin{tikzcd}
T \arrow[r, "m"] \arrow[d] & \SZ\arrow[d]\\
T' \arrow[r] &\SW
\end{tikzcd}
\end{equation}
with $m$ a smooth morphism and $T \rightarrow T'$ a square-zero extension of affine schemes with ideal sheaf $I\in\Qcoh(T_{\liset})$. Let $\cC_T$ (resp. $\cC_{T'}$) denote the pullback of the universal curve to $T$ (resp. $T'$). We first observe that from the definition of the moduli stacks we have a commuting diagram of algebraic stacks
\begin{equation}\label{eq:ofspaces}
\begin{tikzcd}
&\cC_T \arrow[r, "f_T"] \arrow[dl, gray]& \cZ \arrow[r] & \cW \\
\cC_{T'}\arrow[urrr, gray] \arrow[d, equal]& \cC_T \arrow[u, equal] \arrow[d, "p"]\arrow[r, " m'"] & \cC_{\SZ} \arrow[d, "p"]\arrow[u, "f"'] \arrow[r] & \cC_{\SW} \arrow[u, "f_{\SW}"']\arrow[d, "p"]\\
\cC_{T'}  \arrow[d]&T \arrow[r, "m"] \arrow[dl, gray]& \SZ \arrow[r] & \SW\\
T' \arrow[urrr, gray]  &&&
\end{tikzcd}
\end{equation}

\begin{claim}[Step 1]\label{cl:ofexals} The diagram \eqref{eq:ofspaces} leads to a commuting diagram of Picard categories
\begin{equation}\label{eq:ofexals}
\begin{tikzcd}
\Def(f_T) \arrow[r] & \Exal_\cZ(\cC_T, p^*I) \arrow[r, "B"] & \Exal_\cW(\cC_T, p^*I) \\
\Def(m) \arrow[u, "\Psi"]\arrow[r] &\arrow[u, "B \circ E"] \Exal_{\SZ}(T, I) \arrow[r, "B"]& \Exal_{\SW}(T, I)\arrow[u, "B\circ E"]
\end{tikzcd}
\end{equation} 
where the arrows $B$ and $E$ are as in Lemmas \ref{lem:ft-functoriality1} and \ref{lem:ft-functoriality2}, the terms in the leftmost column are fibers of the top and bottom horizontal maps and $\Psi$ is an isomorphism.
\end{claim}
\begin{proof}
The right square in the diagram follows from the bottom two (fibered) squares of \eqref{eq:ofspaces} and the definitions of $B$ and $E$. Moreover, the element of $\Exal_{\SW}(T, I)$ defined by $m$ and its horizontal square in \eqref{eq:ofspaces} maps under $B\circ E$ to the element of $\Exal_{\cW}(\cC_T, p^*I)$ defined by $f_T$ and its horizontal square. By the definition \eqref{eq:fiber1} of $\Def$ we get the left square of \eqref{eq:ofexals}.

To prove that $\Psi$ is an isomorphism suffices to check \'etale-locally on $T$; i.e., it suffices to show that $\Psi$ induces an equivalence of categories $\Def(m) \rightarrow \Def(f_T)$. For this we construct an inverse functor. Let $(k, \epsilon, \delta)$ be an element of $\Def(f_T)$. We get an arrow $k_\delta: T' \rightarrow \SZ$ determined by $k$ and $\delta$, making the resulting triangle over $\SW$ strictly commutative. The 2-morphism $\epsilon$ determines a 2-morphism (also denoted $\epsilon$) from $m$ to the composition $T \rightarrow T' \xrightarrow{k_\delta} \SZ$. Hence our functor sends the object $(k, \epsilon, \delta)$ to the object $(k_\delta, \epsilon, id)$. We leave it to the reader to check that this is inverse to $\Psi$.
\end{proof}

\begin{claim}[Step 2]\label{cl:ofestacks} The diagram \eqref{eq:ofspaces} leads to a morphism of distinguished triangles
\begin{equation}\label{eq:ofestacks1}
\begin{tikzcd}[column sep=scriptsize]
\R p_*\Rhom({\L f_T^*\LL_{\cZ/\cW}},{p^*I}) \arrow[r] &[-7pt] \R p_*\Rhom({\LL_{\cC_T/\cZ}},{p^*I[1]}) \arrow[r] &[-5pt] \R p_*\Rhom({\LL_{\cC_T/\cW}},{p^*I[1])} \\
\Rhom({\L m^*\LL_{\SZ/\SW}},{I}) \arrow[u]\arrow[r] & \Rhom({\LL_{T/\SZ}},{I[1]}) \arrow[u]\arrow[r] & \Rhom({\LL_{T/\SW}},{I[1]})\arrow[u]
\end{tikzcd}
\end{equation}
where the leftmost vertical arrow has the property that there exists a composition
\begin{equation}\label{eq:ofestacks2}
\Rhom({\L m^*\LL_{\SZ/\SW}},{I}) \rightarrow \R p_*\Rhom({\L f_T^*\LL_{\cZ/\cW}},{p^*I}) \xrightarrow[\sim]{\exists} \Rhom(\L m^*(\R p_*\LL_{\cZ/\cW}\otimes \omegabul_{\SZ}), I)
\end{equation} equal to the map induced by $\phi_{\SZ/\SW}: \R p_*\LL_{\cZ/\cW}\otimes \omegabul_{\SZ}\rightarrow \LL_{\SZ/\SW}$. Applying the functor $ch \circ \tau_{\leq 0}\circ\R \Gamma$ to \eqref{eq:ofestacks1} yields a commuting diagram of Picard categories
\begin{equation}\label{eq:ofestacks}
\begin{tikzcd}[column sep=scriptsize]
\Estack{\L f_T^*\LL_{\cZ/\cW}}{p^*I} \arrow[r] &[-7pt] \Estack{\LL_{\cC_T/\cZ}}{p^*I[1]} \arrow[r, "A"] &[-5pt] \Estack{\LL_{\cC_T/\cW}}{p^*I[1]} \\
\Estack{\L m^*\LL_{\SZ/\SW}}{I} \arrow[u, "\Phi"]\arrow[r] & \Estack{\LL_{T/\SZ}}{I[1]} \arrow[u, "A \circ D^{-1} \circ C"]\arrow[r, "A"] & \Estack{\LL_{T/\SW}}{I[1]}\arrow[u, "A\circ D^{-1} \circ C"]
\end{tikzcd}
\end{equation}
where the arrows $A,$ $D$, and $C$ are defined as in Lemmas \ref{lem:ft-functoriality1} and \ref{lem:ft-functoriality2}, the terms in the leftmost column are the kernels of the top and bottom horizontal maps, and if $\Phi$ is an isomorphism then $\Phi_0$ and $\Phi_{-1}$ (defined in \eqref{eq:Phi}) are isomorphisms.
\end{claim}
\begin{proof}
There is a morphism of distinguished triangles (see \cite[Lem~2.2.12]{webb-thesis})
\[
\begin{tikzcd}
\LL_{\cC_T/\cW} \arrow[r] \arrow[d] & \LL_{\cC_T/\cZ} \arrow[r] \arrow[d] & \L f_T^*\LL_{\cZ/\cW}[1] \arrow[r] \arrow[d] & {}\\
p^*\LL_{T/\SW} \arrow[r] & p^*\LL_{T/\SZ} \arrow[r] & p^*\L m^*\LL_{\SZ/\SW}[1] \arrow[r] & {}
\end{tikzcd}
\]
(note that the vertical arrows are only defined in the derived category). Applying\\
$\R p_*\Rhom^{\qc}(-, p^*I[1])$ to this diagram and composing with the morphism \eqref{eq:sheafy-pullback} yields \eqref{eq:ofestacks1}, but with $\Rhom^{\qc}$ in place of $\Rhom$. Now Lemma \ref{lem:strange} (applied in the context of Example \ref{ex:general}) produces the composition \eqref{eq:ofestacks2} that is isomorphic to the map induced by $\phi_{\SZ/\SW}$, but still with $\Rhom^{\qc}$ in place of $\Rhom$. To replace $\Rhom^{\qc}$ with $\Rhom$, we observe that
all stacks in \eqref{eq:ofspaces} are locally Noetherian and all morphisms are locally of finite type,\note{uses that $m$ is smooth, hence locally finitely presented.}\note{\url{http://indico.ictp.it/event/a0255/session/14/contribution/9/material/0/0.pdf } says that for $\LL_{X/Y}$ to be pseudo-coherent, you need $X \to Y$ to be locally of finite type nad $X, Y$ to be locally Noetherian schemes. My idea of a proof: reduce to the cotangent complex of a map $A \to B$ of Noetherian rings such that $B$ is a finite $A$-module. Now at least the differentials $\Omega_{B/A}$ are finitely generated, hence coherent since $B$ is Noetherian. } so by \cite[Tag~08PZ]{tag} all cotangent complexes are pseudo-coherent (in fact, in the derived category $\D^{-}_{Coh}$ of the appropriate topos),\note{to get from bounded above with coherent cohomology to pseudo-coherent, tag 08e8 requires noetherian hypotheses so i need to be careful. OK reduce to the base being an affine scheme and use 066E.} and we may make the replacement by \cite[Tag~0A6H]{tag} (recall that we are working on an affine scheme $T$). Now \eqref{eq:ofestacks} is produced by applying $ch\circ \tau_{\leq 0}\circ \R \Gamma$ and using \cite[Tag~08J6]{tag}, and arguing as at the end of the proof of Corollary \ref{cor:ft2}. The map $\Phi$ being an isomorphism implies $\Phi_0$ (resp. $\Phi_{-1}$) is an isomorphism by restricting $\Phi$ to isomorphism classes of objects (resp. automorphisms of the identity).
\end{proof}

\begin{claim}[Step 3]
Condition (3) in Lemma \ref{lem:BF} holds.
\end{claim}

\begin{proof}
We study the commuting cube formed by mapping the right square of \eqref{eq:ofestacks} (on the top floor) to the right square of \eqref{eq:ofexals} (on the ground) via \eqref{eq:ft} (vertical maps): 
\begin{equation}\label{eq:D}
  \begin{tikzcd}[row sep=tiny, column sep=tiny]
& \Estack{\LL_{\cC_T/\cZ}}{p^*I[1]}\arrow[rr] \arrow[dd] & [-5em]&  \Estack{\LL_{\cC_T/\cW}}{p^*I[1]}\arrow[dd]  \\
\Estack{\LL_{T/\SZ}}{I[1]} \arrow[rr, crossing over, ]\arrow[ur]  \arrow[dd]& & \Estack{\LL_{T/\SW}}{I[1]} \arrow[ur]\\
& \Exal_\cZ(\cC_T, p^*I)  \arrow[rr] & & \Exal_\cW(\cC_T, p^*I) \\
\Exal_{\SZ}(T, I) \arrow[rr]\arrow[ur] & & \Exal_{\SW}(T, I) \arrow[ur]\arrow[from=uu, crossing over]\\
\end{tikzcd}
\end{equation}
This cube commutes by Lemmas \ref{lem:ft-functoriality1} and \ref{lem:ft-functoriality2}. We note that Theorem \ref{thm:ft} applies because the maps $\cC_T \rightarrow \cZ$ and $\cC_T \rightarrow \cW$ are representable: for example, representability of $\cC_T \rightarrow \cZ$ follows from the fact that $m'$ is representable and \cite[Tag~04Y5]{tag}.

To prove (3a), restrict the diagram \eqref{eq:D} to isomorphism classes of objects. As in \eqref{eq:defob} we extend this diagram by the obstruction maps, obtaining a commutative diagram
\[
\begin{tikzcd}
\Exalset_{\cW}(\cC_T, p^*I) & \Ext^1(\LL_{\cC_T/\cW}, p^*I) \arrow[l, "\text{\eqref{eq:ft}}"', "\sim"] \arrow[r, "ob"] & \Ext^1(\L f_T^*\LL_{\cZ/\cW}, p^*I) \\
\Exalset_{\SW}(T, I) \arrow[u, "B\circ E"] & \Ext^1(\LL_{T/\SW}, I) \arrow[l, "\eqref{eq:ft}"', "\sim"] \arrow[u] \arrow[r, "ob"] & \Ext^1(\L m^*\LL_{\SZ/\SW}, I) \arrow[u, "\Phi_1'"]
\end{tikzcd}
\]
where the left square is a side of our cube and the right square is obtained by applying the derived global sections functor $\R \Gamma$ to \eqref{eq:ofestacks1} and then taking cohomology.\note{this uses $\R \Gamma \circ \R p_* = \R (\Gamma \circ p_*) = \R \Gamma$.} By defintion of $B\circ E$ and commutativity of the diagram, the map labeled $\Phi_1'$ sends $o(m)$ to $o(f_T)$. By \eqref{eq:ofestacks2} the map $\Phi_1'$ is quasi-isomorphic to $\Phi_1$, where $\Phi_1$ is defined as in Lemma \ref{lem:BF}. By Corollary \ref{cor:ft2}, the element $o(f_T)$ (resp $o(m)$) vanishes if and only if $\Defset(f_T)$ (resp $\Defset(m)$) is nonempty. Since the map $\Psi: \Defset(m)\rightarrow\Defset(f_T)$ from \eqref{eq:ofexals} is an isomorphism we see that (3a) holds.

To prove (3b), we may assume the diagram \eqref{eq:D} was formed from the trivial example of \eqref{eq:main1} (see Example \ref{ex:triv-diagram}). In this case the terms in the left column of \eqref{eq:ofexals} are kernels (not just fibers) of the horizontal maps, so \eqref{eq:D} induces the following commuting square of kernels:
\[
\begin{tikzcd}
\Estack{\L f_T^*\LL_{\cZ/\cW}}{p^*I} \arrow[r,"\sim"']  & \Def(f_T)\\
\Estack{\L m^*\LL_{\SZ/\SW}}{I} \arrow[u, "\Phi"] \arrow[r, "\sim"'] & \Def(m) \arrow[u, "\Psi"]
\end{tikzcd}
\]
The horizontal maps are induced by the instances of \eqref{eq:ft} in the diagram \eqref{eq:D} and they are isomorphisms because \eqref{eq:ft} is an isomorphism.
Since $\Psi$ is an isomorphism, $\Phi$ is an isomorphism as well.
\end{proof}

\appendix
\section{Descent theorems for lisse-\'etale sheaves on algebraic stacks}\label{app:descent}
In this section, we recall the unbounded cohomological descent theorem in \cite[Ex~2.2.5]{LO08} for quasi-coherent sheaves in the lisse-\'etale site of an algebraic stack (Proposition \ref{prop:descent-olsson}), and then we use it to prove a new descent theorem (Propositions \ref{prop:descent}) that is needed in this paper. In this section, if $\cX$ is an algebraic stack we use $\Le(\cX)$ to denote the lisse-\'etale site, and if $U$ is an algebraic space we use $\Et(U)$ to denote its small \'etale site (\cite[Tag~03ED]{tag}).

\subsection{Morphisms from \'etale to lisse-\'etale sites}\label{sec:cocontinuous-functor}
If $U$ is an algebraic space and $m: U \to \cX$ is a smooth morphism, there is an induced functor of sites $\Et(U) \to \Le(\cX)$ (also denoted $m$) that sends a scheme $V$ with an \'etale map $V \to U$ to the composition $V \to U \to \cX$. 

\begin{remark}\label{rmk:cocontinuous-functor}We make the following observations about the functor $m$:
\begin{enumerate}
\item \label{cocontinuous-functor1}The functor $m: \Et(U) \to \Le(\cX)$ is cocontinuous and hence induces a morphism of topoi $m: U_{\et} \to \cX_{\liset}$ by \cite[Tag~00XI]{tag}. The functor $m^{-1}: \cX_{\liset} \to U_{\et}$ is just restriction.
\item \label{cocontinuous-functor1.5} Since $m^{-1}$ is restriction, we have $m^{-1}\OO_{\cX} = \OO_U$ and $m^{-1} \cF = m^*\cF$ when $\cF$ is a sheaf of $\OO_\cX$-modules.
\item \label{cocontinuous-functor2}The functor $m: \Et(U) \to \Le(\cX)$ is also continuous, and hence $m^{-1}$ has a left adjoint by \cite[Tag~04BG]{tag}. Since $m$ commutes with fiber products and equalizers, the left adjoint is exact by \cite[Tag~04BH]{tag}. In particular $m^{-1}$ preserves injectives.
\end{enumerate}
Suppose we have the following commuting diagram of algebraic stacks where $V$ and $U$ are algebraic spaces and $m$ is smooth.
\[
\begin{tikzcd}
V \arrow[d, "m'"'] \arrow[r, "f'"] & U\arrow[d, "m"]\\
X \arrow[r, "f"] & \cX
\end{tikzcd}
\]
\begin{enumerate}[resume]
\item \label{cocontinuous-functor3}
If $f$ is representable and $V = U \times_{\cX} X$, then 
for $\cF \in X_{\liset}$ we have a canonical identification $f'_*m'^{-1}\cF = m^{-1}f_*\cF$, where $f'_*:V_{ \et} \to U_{\et}$ (resp. $f_*: X_{ \liset}\to \cX_{\liset}$) is the usual pushforward of \'etale (resp. lisse-\'etale) sheaves induced by a continuous functor of sites. (Note that $f_*$ may not have an exact left adjoint.\note{ and it may not preserve quasi-coherent sheaves}) Indeed, if $W$ is a scheme and $W \to U$ is \'etale, then we have
\[
(f'_*m'^{-1}\cF)(W) = \cF(W \times_U V) \quad \quad \quad \quad (m^{-1}f_*\cF)(W) = \cF(W \times_{\cX} X)
\]
but there is a natural identification of algebraic spaces $W \times_U V \simeq W \times_{\cX} X$.
\item\label{cocontinuous-functor4} If $f$ is smooth, then $f_*$ has an exact left adjoint and we let $f^*: \Mod{\OO_{\cX_{\liset}}} \to \Mod{\OO_{X_{\liset}}}$ be the induced pullback of $\OO$-modules. In this case, $f'^*m'^*\cF = m^*f^*\cF$ for $\cF \in \Qcoh(\cX_{\liset}).$ Indeed, by part \eqref{cocontinuous-functor1.5} above (since $f$ is representable) the functors $m^*$ and $m'^{*}f^*$ are just restriction, but $f'^*$ is the pullback functor from $\Qcoh(U_{\et})$ to $\Qcoh(U_{X,\et})$. Hence the desired equality holds by the Cartesian property of $\cF$. 
\end{enumerate}
\end{remark}

\subsection{The first descent theorem}\label{sec:descent-olsson} \label{sec:descent1}
In this section, we recall Lazslow-Olsson's theorem for unbounded cohomological descent for lisse-\'etale sheaves on an algebraic stack (Proposition \ref{prop:descent-olsson}).
To begin we recall the following general construction (which will be used multiple times in this appendix).
\begin{construction}\label{const:simplicial-topos}
Let $I$ be a category and let $\cC$ be a functor from $I$ to the 2-category of categories (see \cite[Tag~003N]{tag}); that is, for each $i \in I$ we have a category $\cC_i$, and for each morphism $\phi: i \to j$ in $I$ we have a functor $\phi^*_{\cC}: \cC_i \to \cC_j$ and these are compatible with compositions. We define a \emph{category of systems} $\cC_{total}$ whose objects are tuples $\cF := (\cF_i, \cF(\phi))$ with $\cF_i \in \cC_i$ and $\cF(\phi): \phi^*_{\cC} \cF_i \to \cF_j$, such that the following diagrams commute.\note{this is similar to [{cite SGA4 Vbis.1.2.1}]}
\[
\begin{tikzcd}
\phi^*_\cC \psi^*_{\cC} \cF_\ell \arrow[rr, "\cF(\psi \circ \phi)"] \arrow[dr, "\phi^*_{\cC} \cF(\psi)"'] && \cF_n\\
& \phi^*_{\cC} \cF_m \arrow[ur, "\cF(\phi)"']
\end{tikzcd}
\]
A morpism from $(\cF_i, \cF(\phi))$ to $(\cG_i, \cG(\phi))$ in $\cC_{total}$ is a collection of morphisms $\alpha_i: \cF_i \to \cG_i$ compatible with the $\cF(\phi)$ and $\cG(\phi)$. The \emph{category of Cartesian systems} $\cC^{cart}_{total}$ is the full subcategory of $\cC_{total}$ whose objects have the property that every $\cF(\phi)$ is an isomorphism.
\end{construction}
\begin{remark}\label{rmk:simplicial-topos-functor}
Suppose we are given two functors $\cC, \cD$ from $I$ to the 2-category of categories, and suppose we have functors $\Lambda_i: \cC_i \to \cD_i$ such that the squares
\[
\begin{tikzcd}
\cC_i \arrow[r, "\Lambda_i"] \arrow[d, "\phi^*_C"] & \cD_i \arrow[d, "\phi^*_D"] \\
\cC_j \arrow[r, "\Lambda_j"] & \cD_j
\end{tikzcd}
\]
 2-commute, and the 2-morphisms respect (vertical) compositions of squares. Then we have a functor $\Lambda: \cC_{total} \to \cD_{total}$ given by the rule $\Lambda(\cF_i, \cF(\phi)) = (\Lambda_i(\cF_i), \Lambda_j(\cF(\phi)))$.\note{could check this more careully. check}
\end{remark}

Let $\cX$ be an algebraic stack and let $U \to \cX$ be a smooth cover by an algebraic space. Let $U_\bullet$ be the simplicial algebraic space that is equal to the 0-coskeleton of $U \to \cX$. We apply Construction \ref{const:simplicial-topos} to the category $I := \Delta^+$, were $\Delta^+$ is the subcategory of the simplicial category $\Delta$ with the same objects but only the injective morphisms. For $i \in \Delta^+$ we set $\cC_i := U_{i, \et}$ and for $\phi: i \to j$ we let $\phi^*: U_{i, \et} \to U_{j, \et}$ be the usual inverse image functor for this morphism of topoi. The resulting category of systems is called the \textit{strictly simplicial topos} in \cite[Sec~2.1]{olsson-sheaves} and \cite[Ex~2.1.5]{LO08}, and we notate it $U^+_{\bullet, \et}$. The structure sheaves $\OO_{U_i}$ define a distinguished ring object $\OO_{U^+_{\bullet}}$ in $U^+_{\bullet, \et}$. A \emph{quasi-coherent sheaf} in $U^+_{\bullet, \et}$ is an $\OO_{U^+_{\bullet}}$-module $(\cF_i, \cF(\phi))$\note{you can define a ring object in any category, and a module in any category. unwinding the definitions, a system is an O-module if each $\cF_i$ is an O-module, and if the action maps $O_i \times \cF_i \to \cF_i$ fit into commuting squares with the transition maps $\phi^*\cF_j \to \cF_i$ (where $\phi^*$ is really $\phi^{-1}$)} such that each $\cF_i$ is in $\Qcoh(U_{i, \et})$ and the morphism $\phi^*\cF_i \otimes_{\phi^*\OO_{U_i}} \OO_{U_j} \to \cF_j$ induced by $\cF(\phi^*)$ is an isomorphism.
Observe that the category of quasi-coherent sheaves $\Qcoh(U_{\bullet, \et}^+)$ is equal to the category of Cartesian systems with $\cC_i = \Qcoh(U_{i, \et})$ and $\phi^*$ equal to the usual pullback of quasi-coherent sheaves. 

There is a functor $\varpi^*: \Mod{\OO_{\cX_{\liset}}} \to \Mod{\OO_{U^+_{\bullet, \et}}}$ given as follows: for $\cF \in \cX_{\liset}$ set $(\varpi^*\cF)_i = m_i^{-1} \cF \otimes_{m_i^{-1}\OO_{\cX}} \OO_{U_i}$ where $m_i^{-1}: \cX_{\liset} \to U_{i, \et}$ is defined using the projection $U_i \to \cX$ and Remark \ref{rmk:cocontinuous-functor}.\ref{cocontinuous-functor1}, and let $\cF(\phi)$ be the identity for each $\phi$. Note that $\varpi^*$ is exact and sends quasi-coherent sheaves to quasi-coherent sheaves.\note{this is maybe not completely obvious from what I've said so far?? but Martin proves it all.}
The following proposition is due to Laszo-Olsson.

\begin{proposition}[Laszlo-Olsson]\label{prop:descent-olsson}
The morphism $\varpi^*: \Qcoh(\cX_{\liset}) \to \Qcoh(U^+_{\bullet, \et})$ is an exact equivalence of categories. We use $\varpi_*$ to denote the quasi-inverse. Moreover, $\varpi^*: \Dqc(\cX_{\liset}) \to \Dqc(U^+_{\bullet, \et})$ is an equivalence, and we use $\R \varpi_*$ to denote the quasi-inverse.
\end{proposition}
\begin{remark}
The equivalence of categories of quasi-coherent sheaves is proved in \cite[Prop~9.2.13]{olsson-book}. The equivalence of unbounded derived categories is proved in \cite[Ex~2.2.5]{LO08} (using \cite[Thm~6.14]{olsson-sheaves}) under the assumption that $\cX$ is quasi-separated (a standing assumption for both \cite{olsson-sheaves} and \cite{LO08}). This assumption is not needed for Proposition \ref{prop:descent-olsson}. Indeed, \cite[Thm~2.2.3]{LO08} appears as \cite[Tag~0D7V]{tag} without the quasi-separated hypothesis, and one may check directly that the necessary portions of \cite{olsson-sheaves} (namely Proposition 4.4, Lemma 4.5, and Lemma 4.8) do not use this hypothesis.\note{Note HR17 doesn't assume quasi-separated. (Quasi-separated is used in the proof of Lemma 6.5(i) in olsson-sheaves, for example.)}
\end{remark}
\begin{remark}\label{rmk:functorial-varpi}

The equivalences $(\varpi^*, \varpi_*)$ are functorial as follows.
Let $X \to \cX$ be a smooth morphism of algebraic stacks (inducing a morphism of lisse-\'etale topoi), let $V \to X$ be a smooth surjective morphism from a scheme $V$, and let $V \to U$ be a morphism commuting with maps to $\cX$. 
It follows from Remark \ref{rmk:cocontinuous-functor}.\ref{cocontinuous-functor4} that there is an identification  $f^*_{\bullet,\et} \varpi^*\simeq \varpi^* f^*$ (where $f^*_{\bullet, \et}$ is given by $f^*_i$ at level $i$), and since $\varpi^*$ is an equivalence we also have $\varpi_* f^*_{\bullet,\et} \simeq f^* \varpi_*$. 

\end{remark}

\subsection{The second descent theorem: hypercovers} In this section we prove an unbounded cohomological descent theorem in the lisse-etale topology for very smooth hypercovers of algebraic stacks (Proposition \ref{prop:descent}).

\subsubsection{Very smooth hypercovers}
Recall that if $\cU \to \cX$ and $\cV \to \cX$ are representable morphisms of algebraic stacks, then the category $\Hom_{\cX}(\cU, \cV)$ is isomorphic to a set.

\begin{definition}
The \emph{enlarged smooth site} $\Es(\cX)$ of $\cX$ is the category with objects given by morphisms $f: \cU \to \cX$ where $\cU$ is an algebraic stack and $f$ is smooth and representable, and with arrows from $\cU \to \cX$ to $\cV \to \cX$ given by the set $\Hom_{\cX}(\cU, \cV)$. A covering is a set of smooth maps $\{\cU_i \to \cU\}_{i \in I}$ that are jointly surjective.
\end{definition}

\begin{remark}
The site $\Es(\cX)$ contains $id: \cX \to \cX$ as the final object.
\end{remark}

\begin{remark}\label{rmk:es-representable}
The morphisms in $\Es(\cX)$ are all representable.\note{because: representable iff faithful (04y5), and if $U \to \cX$ is injective on morphisms and it factors as $U \to V \to \cX$ then $U \to V$ is also injective on morphisms.}
\end{remark}

\begin{definition}
A \emph{smooth hypercover} of $\cX$ is a simplicial object $X_\bullet$ in $\Es(\cX)$ such that
\begin{enumerate}
\item $X_0 \to \cX$ is surjective (note that it will also be smooth)
\item $X_{n+1} \to (\mathrm{cosk}_n\mathrm{sk}_n X_\bullet)_{n+1}$ is smooth and surjective for $n\geq 0$.
\end{enumerate}
\end{definition}

\begin{remark}\label{rmk:agrees-with-fppf}
A smooth hypercover of $\cX$ is a hypercover of the final object in $\Es(\cX)$ in the sense of \cite[Tag~01G5]{tag}.\note{The stacks project tag defines a hypercover to be a simplicial object in the category of semi-representable objects. Martin says that semi-representable objects are a kind of hack to make sure that arbitrary disjoint unions are allowed in your category (your category here is the site). I don't need to worry about that because my site has arbitrary disjoint unions. At any rate, my statement here is true: a single object of the site is also a semi-representable object of the site. Also, if I use a stacks project tag that produces a semi-representable cover (like 0DAV), I can just take the union in my site and get an actual representble hypercover like I've defined.}
Moreover, if $\cX$ is an algebraic space, then a smooth hypercover of $\cX$ is also an fppf hypercover in the sense of \cite[Tag~0DH4]{tag}.\note{This tag has a separate requirement for (2). If it didn't, perhaps the meaning would be to require $U_1 \to U_0 \times_S U_0$ to be a covering, but that's not what you want: you want $U_1 \to U_0 \times_X U_0$. For larger $n$ this misunderstanding won't matter.}
\end{remark}

\begin{definition}\label{def:very-smooth-hypercover}
A \emph{very smooth hypercover} of $\cX$ is a smooth hypercover $X_\bullet$ such that every degeneracy and face map $X_i \to X_j$ is smooth.
\end{definition}


If $X_\bullet$ is a smooth hypercover of $\cX$ and $f: \cY \to \cX$ is a morphism of algebraic stacks, we can pullback $X_\bullet$ to a simplicial object $Y_\bullet$ in $\Es(\cY)$: define $Y_i = X_i \times_{\cX} \cY$.



\begin{remark}\label{rem:hyper-pullback}
If $\cY \to \cX$ is a morphism of algebraic stacks and $X_\bullet$ is a (very) smooth hypercover of $\cX$, then $Y_\bullet$ is a (very) smooth hypercover of $\cY$. This follows from \cite[Tag~0DAZ]{tag}.
\end{remark}

\begin{remark}\label{rmk:hyper-affine}
From \cite[Tag~0DEQ]{tag} and the proof of \cite[Tag~0DAV]{tag} it follows that if $\cX$ is an algebraic stack then a very smooth hypercover of $\cX$ exists. In fact, we may take $X_i$ to be a disjoint union of affine schemes.
\end{remark}

\begin{remark}\label{rem:hyper-double-affine}
Let $\cX \to \cY$ be a morphism of algebraic stacks. Then we can find very smooth hypercovers $X_\bullet \to \cX$ and $Y_\bullet \to \cY$ with $X_i$ and $Y_i$ disjoint unions of affine schemes, with a morphism $X_\bullet \to Y_\bullet$ commuting with the augmentations and the given morphism $\cX \to \cY$. This follows from analyzing the construction of $X_\bullet$ and $Y_\bullet$ in \cite[0DAV]{tag}, using the fact that the functors $\cosk_n$ are finite limits and hence commute with pullback (see the proof of \cite[Tag~0DAZ]{tag}).
\end{remark}

\subsubsection{The lisse-etale topos of a very smooth hypercover}\label{sec:descent-topos}

Recall that if $\cX \to \cY$ is a smooth morphism of algebraic stacks then there is a morphism of sites $\Le{\cX} \to \Le{\cY}$ (see e.g. \cite[Tag~00X1]{tag} and \cite[Sec~3.3]{olsson-sheaves}), and in fact a morphism of ringed topoi $(\cX_{\liset}, \OO_{\cX}) \to (\cY_{\liset}, \OO_{\cY})$.
We follow \cite[Tag~09WB]{tag} by defining the category of sites to be the category whose objects are sites and whose morphisms are morphsisms of sites. If $\cC_\bullet$ is a simplicial object in this category, then for each morphism $\varphi: [i] \to [j]$ of the simplicial category $\Delta$ we have a morphism of sites $f_\varphi: \cC_i \to \cC_j$.

\begin{definition}\label{def:le-site}
Let $X_\bullet$ be a very smooth hypercover of $\cX$. We construct an associated site $\Le(X_{\bullet})$ as follows: let $\cC_\bullet$ be the simplicial object in the category of sites with $\cC_i := \Le(X_{i})$ and $f_\varphi$ equal to the given morphism of sites (it is important that all the face and degeneracy maps are smooth). Define $\Le(X_{\bullet})$ to be the site $\cC_{total}$ in \cite[Tag~09WC]{tag}, and use $X_{\bullet, \liset}$ to denote the corresponding topos.
\end{definition}

\begin{remark}\label{rmk:systems}
By \cite[Tag~09WF]{tag}, a sheaf on $\Le(X_{\bullet})$ is given by a system $(\cF_i, \cF(\varphi))$ where $\cF_i$ is a sheaf on $\Le(X_{i})$ and $\cF(\varphi): f_{\varphi}^{-1}\cF_i \to \cF_j$ are compatible morphisms. 
\end{remark}

Using Remark \ref{rmk:systems}, define a sheaf $\OO_{X_{\bullet, \liset}}$ on $X_{\bullet, \liset}$ to be the sheaf equal to $\OO_{X_i}$ on $X_i$ with transition maps induced by the morphisms of ringed topoi already given. This makes $X_{\bullet, \liset}$ a ringed site. An $\OO_{X_{\bullet, \liset}}$-module $\cF$ on $X_{\bullet, \liset}$ is quasi-coherent if for each $i$ the sheaf $\cF_i$ is a quasicoherent $\OO_{X_i}$-module and if for each $\varphi: [i] \to [j]$ the induced maps
\[
f^{-1}_\varphi \cF_i \otimes_{f^{-1}\OO_{X_i}} \OO_{X_j} \to \cF_j
\]
are isomorphisms.

\subsubsection{The descent theorem}
For $X_\bullet$ a very smooth hypercover of $\cX$, let $a_i: X_i \to \cX$ denote the given (smooth) morphism of algebraic stacks. 

The morphism $X_0 \to \cX$ induces an augmentation of $\Le(X_{\bullet})$ towards $\Le(\cX)$ in the sense of \cite[Tag~0D6Z]{tag}. By \cite[Tag~0D70]{tag} we get a morphism of topoi 
\begin{equation}\label{eq:defa}
a: X_{\bullet, \liset} \to \cX_{\liset}\end{equation}
such that $a^{-1}\cF$ is given by the system with $(a^{-1}\cF)_i := a_i^{-1}\cF$ and the natural transition maps (they are all isomorphisms), and $a_*\cG$ 
is given by the equalizer of the two maps $a_{0*}\cG_0 \to a_{1*}\cG_1$. 

Using the maps $a_i^{-1}\OO_{\cX} \to \OO_{X_i}$, we get a morphism $a^{-1}\OO_{\cX} \to \OO_{X_{\bullet, \liset}}$ that makes $a$ a morphism of ringed topoi. 
Define $a^*: \Mod{\OO_{\cX}} \to \Mod{\OO_{X_{\bullet, \liset}}}$ by 
\[
a^*\cF := a^{-1} \cF \otimes_{a^{-1}\OO_{\cX}} \OO_{X_{\bullet, \liset}}
\]
It is clear that $a^*$ is exact and sends $\Qcoh(\cX_{\liset})$ to $\Qcoh(X_{\bullet, \liset})$.

\begin{proposition}\label{prop:descent}
Let $X_\bullet \to \cX$ be a very smooth hypercover. Then \begin{equation}\label{eq:descent1}a^*:\Qcoh(\cX_{\liset})\to \Qcoh(X_{\bullet, \liset})
\end{equation}
is an equivalence of categories with quasi-inverse $a_*$. Moreover, the functors $\R a_*$ and $a^*$ are inverse equivalences of $\Dqc(\cX_{\liset})$ and $\Dqc(X_{\bullet, \liset})$.
\end{proposition}

\begin{proof}
We first show that $a^*$ is an equivalence of categories of quasi-coherent sheaves with quasi-inverse $a_*$. Let $U \to \cX$ be a smooth map from an algebraic space $U$, and let $U^+_{\bullet, \et}$ be the strictly simplicial \'etale topos defined in Section \ref{sec:descent1}.
We apply Construction \ref{const:simplicial-topos} to the category $I=\Delta \times \Delta^+$. For $(i, j) \in \Delta \times \Delta^+$ we set $\cC_{i,j} = \Qcoh\big( (X_i \times_{\cX} U_j)_{\et}\big)$ (observe that the fiber product is an algebraic space) and we let $\phi^*: (X_i \times_{\cX} U_j)_{\et} \to (X_k \times_{\cX} U_\ell)_{\et}$ be the usual pullback of quasi-coherent sheaves. Let $\Qcoh((X_\bullet \times_\cX U_\bullet^+)_{\et} )$ denote the resulting category of  Cartesian systems.

Let $U_{X_i} = X_i \times_{\cX} U$. By viewing $\Qcoh(X_{\bullet, \liset})$ and $\Qcoh((X_\bullet \times_\cX U_\bullet^+)_{\et} )$ both as categories of systems with $I=\Delta$, we define functors 
\[
\begin{tikzcd}
\Qcoh(X_{\bullet, \liset}) \arrow[r, shift left=.2em, "\varpi_\bullet^*"] & \arrow[l, shift left=.2em, "\varpi_{\bullet,*}"] \Qcoh((X_\bullet \times_\cX U_\bullet^+)_{\et} )
\end{tikzcd}
\]
induced via Remark \ref{rmk:simplicial-topos-functor} by the inverse equivalences
\[
\begin{tikzcd}
\Qcoh(X_{i, \liset}) \arrow[r, shift left=.2em, "\varpi^*"] & \arrow[l, shift left=.2em, "\varpi_*"] \Qcoh(U^+_{X_i, \bullet, \et} )
\end{tikzcd}
\]
of Proposition \ref{prop:descent-olsson}. The rules $\varpi^*_\bullet$ and $\varpi_{*, \bullet}$ are indeed functors of categories of systems by Remark \ref{rmk:functorial-varpi}, and one checks that they are inverse equivalences.\note{we need to show that the square in this diagram
\[
\begin{tikzcd}[ampersand replacement=\&]
\phi^*\cF_i \arrow[r] \arrow[d] \& \phi^*\varpi_*\varpi^*\cF_i \arrow[r, equal] \arrow[d]\& \varpi_*\varpi^*\phi^*\cF_i \arrow[dl]\\
\cF_j \arrow[r] \& \varpi_*\varpi^*\cF_j
\end{tikzcd}
\]
commutes. the outer one does by naturality of $1 \to \varpi_*\varpi^*$ and the triangle does by definition of $\varpi_*\varpi^*\cF$ (how you apply a functor to the morphism-part of the $\cF$ data)}

Similarly, let $X_{U_i, \bullet}$ be the pullback of the hypercover $X_\bullet \to \cX$ to $U_i$ as in Remark \ref{rem:hyper-pullback}. By viewing $U^+_{\bullet, \et}$ and $\Qcoh((X_\bullet \times_\cX U_\bullet^+)_{\et} )$ as categories of systems with $I=\Delta^+$ we define functors \[
\begin{tikzcd}
\Qcoh(U^+_{\bullet, \et}) \arrow[r, shift left=.2em, "a_\bullet^*"] & \arrow[l, shift left=.2em, "a_{\bullet,*}"] \Qcoh((X_\bullet \times_\cX U_\bullet^+)_{\et} )
\end{tikzcd}
\]
induced via Remark \ref{rmk:simplicial-topos-functor} by the functors 
\begin{equation}\label{eq:descent100}
\begin{tikzcd}
\Qcoh(U_{i,\et}) \arrow[r, shift left=.2em, "a^*"] & \arrow[l, shift left=.2em, "a_*"] \Qcoh(X_{U_i, \bullet, \et})
\end{tikzcd}
\end{equation}
defined in analogy with \eqref{eq:defa} above. The functors $a_*$ and $a^*$ in \eqref{eq:descent100} are inverse equivalences by \cite[Tag~0DHD]{tag}.

We have constructed a diagram
\[
\begin{tikzcd}
\Qcoh(X_{\bullet, \liset}) \arrow[r, shift left = .2em, "\varpi^*_\bullet"] \arrow[d,  shift left = .2em,"a_*"] & \Qcoh((X_\bullet \times_\cX U_\bullet^+)_{\et} ) \arrow[l,  shift left = .2em,"\varpi_{\bullet, *}"] \arrow[d,  shift left = .2em,"a_{\bullet, *}"]\\
\Qcoh(\cX_{\liset}) \arrow[u,  shift left = .2em,"a^*"] \arrow[r, shift left = .2em, "\varpi^*"] & \Qcoh(U^+_{\bullet, \et}) \arrow[l, shift left = .2em, "\varpi_*"] \arrow[u, shift left = .2em, "a_{\bullet}^*"]
\end{tikzcd}
\]
where three of the four pairs of morphisms are known to be inverse equivalences. It follows from Remark \ref{rmk:cocontinuous-functor}.\ref{cocontinuous-functor4} that $\varpi^*_\bullet a^* = a_{\bullet}^* \varpi^*$,\note{since $\varpi^*$ is a kind of restriction, inverse image equals pullback. See the rmk:functorial-varpi.} so $a^*$ is an equivalence with inverse $\varpi_*a_{\bullet, *}\varpi^*_{\bullet}$. Using Remark \ref{rmk:cocontinuous-functor}.\ref{cocontinuous-functor4} and the fact that $\varpi^*$ is exact one can check that $a_{\bullet,*}\varpi^*_{\bullet} = \varpi^*a_*$, so $a_* = \varpi_*a_{\bullet, *}\varpi^*_{\bullet}$. This shows that $a^*$ and $a_*$ are inverse equivalences of quasi-coherent sheaves.

To finish the proof of the Proposition, we use \cite[Tag~0D7V]{tag}. To do so we must verify its five hypotheses. The category $\Qcoh({X_{\bullet, \liset}})$ is a weak Serre subcategory of $\Mod{\OO_{X_{\bullet, \liset}}}$ and conditions (1), (4), and (5) of \cite[Tag~0D7V]{tag} hold as in the proof of \cite[Tag~0DHF]{tag}. Condition (2) is the inverse equivalence of $a^*$ and $a_*$ that we just proved. The final condition, number (3), is the statement that
for $\cF \in \Qcoh(\cX_{\liset})$, the unit $\cF \to \R a_*a^* \cF$ is an isomorphism. Since we already know $\cF \to a_*a^*\cF$ is an isomorphism,\note{unit is an isomorphism iff left adjoint is fully faithful} it suffices to show $\R^n a_*a^*\cF=0$ for $n > 0$. 

For any smooth map $m: U \to \cX$ from a scheme $U$, let $U_\bullet \to U$ be the very smooth hypercover equal to the pullback of $X_\bullet$. We have a diagram of morphisms of topoi
\[
\begin{tikzcd}
U_{\bullet, \et} \arrow[d, "a''"] & U_{\bullet, \liset} \arrow[r, "m_{\bullet, \liset}"] \arrow[l] \arrow[d, "a'"] & X_{\bullet, \liset} \arrow[d, "a"]\\
U_{\et} & U_{\liset} \arrow[l]\arrow[r, "m_{\liset}"] & \cX_{\liset}
\end{tikzcd}
\]
where the site $U_{\bullet, \et}$ is constructed with \cite[Tag~09WC]{tag} and $a''$ and $a'$ are defined as in \eqref{eq:defa}. The top horizontal morphisms come from \cite[Tag~0DH0]{tag}.
It follows from \cite[V.5.1(1)]{SGA4}, \cite[Lem~3.5]{olsson-sheaves},
and Lemma \ref{lem:pullbacks} that $\R^n a_*a^*\cF$ is the sheafification of the presheaf that associates to an smooth map $m:U \to \cX$ from a scheme $U$ the group 
$H^n(U_{\bullet, \liset}, m^{-1}_{\bullet,\liset}a^*\cF)$. Since restriction to the \'etale site is exact and preserves injectives, this is equal to $H^i(U_{\bullet, \et}, m_{\bullet}^{*}a^*\cF)$ where $m_{\bullet}:U_{\bullet, \et} \to X_{\bullet, \liset}$ is defined as in Remark \ref{rmk:cocontinuous-functor}. Finally, by Remark \ref{rmk:cocontinuous-functor}.\ref{cocontinuous-functor4} (since 
$\cF$ is quasi-coherent)
this equals $H^n(U_{\bullet, \et}, a''^*m^*\cF).$
\note{
By \cite[V.5.1(1)]{SGA4}, $\R^na_*a^*\cF$ is the sheafification of the presheaf that associates to a smooth map $m:U \to \cX$ from a scheme $U$ the group $H^n(X_{\bullet, \liset}/a^{-1}\tilde U, a^*\cF)$ where $\tilde U$ is the sheaf represented by $U$ and $X_{\bullet, \liset}/a^{-1}\tilde U$ is the localized topos.
Define $U_i \:= U \times_{\cX} X$ and let $\Le\big(( X_i|_{U_i})_\bullet \big)$ be the site constructed as in \cite[Tag~09WC]{tag} with $\cC_i$ equal to the localized site $\Le(X_i)/U_i$. Unwinding the definitions and using \cite[Exercise~2.D]{olsson-book} and Lemma \ref{lem:pullbacks} below, we see that the localized topos $X_{\bullet, \liset}/a^{-1}\tilde U$ is the category of sheaves on the site $\Le\big(( X_i|_{U_i})_\bullet \big)$.
Now by \cite[Tag~0D7E]{tag}\todo{see research journal notes for 7/23/2021} there is a spectral sequence \todo{converging to} whose first page is
\begin{equation*}
E^{p,q}_1 = H^q(\Le(X_p)/U_p, a_p^*\cF) \quad \quad \quad d_1^{p,q}: E^{p,q}_1 \to E^{p+1,q}_1 
\end{equation*}
 By \cite[Lem~3.5]{olsson-sheaves} we have isomorphisms $H^q(\Le(X_p)/U_p, a_p^*\cF) \simeq H^p(\Le(U_p), m_{p, \liset}^{-1}a^*_p\cF)$ where ${m_{p,\liset}}: U_{p,\liset} \to X_{p,\liset}$ is the usual morphism of topoi (recall that $m$ is smooth). These isomorphisms commute with the differentials above \todo{martin's isomorphism boils down to identifying global sections of $\cF$ on $X|_U$ and of $m^{-1}\cF$ on $U$ as the set $\cF(V)$ (it is definitely important that $U \to X$ is smooth). I think the differential is also boils down to a map of global sections, the alternating sum in 0D9B (see also 0D9C).}, so we get an isomorphism of the limitting groups $H^n(\Le\big(( X_i|_{U_i})_\bullet \big), a^*\cF) \simeq H^n(\Le(U_{\bullet}), m_{\bullet, \liset}^{-1}a^*\cF)$. \todo{see also research journal notes for 7/23/2021}
}


 On the other hand, it follows from \cite[V.5.1(1)]{SGA4} and Lemma \ref{lem:pullbacks} that 
 $\R^na''_*a''^*(m^*\cF)$ is the sheafification of the presheaf that associates to an \'etale map $f:V \to U$ from a scheme $V$ the group $H^n(V_{\bullet, \et}, a''^*m^*\cF)$.\note{also uses that the restricted \'etale topos $U_{\bullet, \et}|_{V_\bullet}$ is equal to $V_{\bullet, \et}$}
 This group is equal to 
 $H^i(V_{\bullet, \et}, a''^*_V(m\circ f)^*\cF)$, where $a''_V: V_{\bullet, \et} \to V_{\et}$ is the usual morphism \eqref{eq:defa}.
 It follows that if $m: U \to \cX$ is a smooth cover by a scheme, the \'etale sheaves $m^*(\R^n a_*a^*\cF)$ and $\R^na''_*a''^*(m^*\cF)$ are the sheafification of the same presheaf,\note{the implicit fact here is that if $\cF_\bullet$ is a presheaf on a lisse-etale site, the etale sheaf on V of its sheafification is the sheafification of the etale presheaf on V. I think this is true since sheafification is a left adjoint, left adjoints preserve colimits, and restriction is a colimit } hence isomorphic. But it follows from \cite[Tag~0DHE]{tag} that $\R^na''_*a''^*(m^*\cF)=0$.
\end{proof}

\begin{lemma}\label{lem:pullbacks}
\label{lem:basechange} Consider a fiber square of algebraic stacks
\[
\begin{tikzcd}
\arrow[d]U_X\arrow[r] & \cX \arrow[d, "b"]\\
U \arrow[r, "m"] & \cY
\end{tikzcd}
\]
such that $U$ is an algebraic space.
\begin{enumerate}
\item If $m$ and $b$ are smooth and representable and $\tilde U$ is the sheaf represented by $U$ on $\cX_{\liset}$, then $b^{-1}\tilde U$ is represented by $U_X$.
\item If $\cY$ is representable, $m$ is \'etale, and $b$ is representable, and if $\tilde U$ is the sheaf represented by $U$ on $\cX_{\et}$, then the \'etale sheaf $b^{-1}\tilde U$ is represented by $U_X$.
\end{enumerate}
\end{lemma}
\begin{proof}
We first sketch the proof of (1). The sheaf $b^{-1}\tilde U$ is the sheafification of the presheaf that assigns to a scheme $T$ with a smooth map $g: T \to \cX$ the set $\colim \Hom_Y(W, U)$, where the colimit is taken over schemes $W$ fitting into diagrams
\begin{equation}\label{eq:pullbacks2}
\begin{tikzcd}
\arrow[d]T \arrow[r, "g"] & \cX \arrow[d, "b"]\\
W \arrow[r, "\text{smooth}"] & \cY
\end{tikzcd}
\end{equation}
Composition induces a map 
\begin{equation}
\label{eq:pullbacks1}
\colim \Hom_\cY(W, U) \to \Hom_\cY(T, U)
\end{equation}
which is an isomorphism since $T \to Y$ is smooth and hence defines the final object in the category over which we take the colimit. Finally we note that $\Hom_Y(T, U) = \Hom_Y(T, U_X)$. For (2), the map $T \to \cX$ is now \'etale and the colimit is over diagrams \eqref{eq:pullbacks2} with $W \to \cY$ \'etale, so $T \to \cY$ is not an object of the colimit category. However, the map \eqref{eq:pullbacks1} is still surjective. It is injective as well because an element of $\Hom_\cY(W, U)$ must be \'etale, so if we have elements of $\Hom_{\cY}(W_1, U)$ and $\Hom_{\cY}(W_2, U)$ that yield the same map $T \to U$, we may compare them via the \'etale $U$-scheme $W_1 \times_U W_2$.
\end{proof}

\begin{remark}\label{rmk:descent}
Let $X_\bullet \to \cX$ and $Y_\bullet \to \cY$ be very smooth hypercovers of algebraic stacks $\cX$ and $\cY$, and suppose we are given a morphism $f_\bullet : X_\bullet \to Y_\bullet$ of simplicial algebraic stacks and $f: \cX \to \cY$ such that these maps commute with the augmentations.  Then for $\cF \in \Dqc(\cY_{\liset})$ we have 
\[(a^*\L f^*\cF)|_{X_n} = \L f_n^*(a^*\cF)|_{X_n},\] 
and if $f$ is concentrated and $\cG \in \Dqc(\cX_{\liset})$ then 
\[(a^*\R f_*\cG)|_{Y_n} = \R f_{n*}(a^*\cG)|_{Y_n},\] where the functor $\L f_n^*$ (resp. $\R f_{n*}$) is the usual pullback functor (resp. direct image) between $\Dqc(Y_{n,\liset})$ and $\Dqc(X_{n,\liset})$. Indeed, the functor $(a^*-)|_{X_n}$ is just $a^*_n(-)$, so the desired equalities are equivalent to
\[
a_n^*\L f^*\cF = \L f_n^*a_n^*\cF \quad \quad \text{and} \quad \quad a_n^*\R f_*\cG = \R f_{n*}a_n^*\cG.
\]
These follow from naturality of derived pullback and \cite[Cor~4.13]{HR17}, respectivley.
\end{remark}

\section{Functoriality of the Fundamental Theorem}\label{sec:ft-functoriality}
In this section we prove Lemmas \ref{lem:ft-functoriality1} and \ref{lem:ft-functoriality2}. 
\subsection{Categories of algebra extensions}
In this section, $\cS$ is a site with $A \rightarrow B$ a morphism of sheaves of rings on $\cS$, and $I$ be a sheaf of $B$-modules.
\subsubsection{Categories}
 The Picard category $\Exal_A(B, I)$ was defined in \cite[Sec~III.1.1.2.3]{illusie}: an object is a surjective $A$-algebra map $E \rightarrow B$ whose kernel is (1) square-zero as an ideal of $E$, and (2) isomorphic to $I$ as a $B$-module.\note{see ega $0_{iv}$.18.2} We write these objects as short exact sequences of abelian sheaves
\begin{equation}\label{eq:start}
0 \rightarrow I \rightarrow E \rightarrow B \rightarrow 0.
\end{equation}
A morphism in $\Exal_A(B, I)$ is a commuting diagram
\[
\begin{tikzcd}
0 \arrow[r] & I \arrow[r] \arrow[d, equal]& E \arrow[r] \arrow[d, "g"]& B \arrow[r] \arrow[d, equal]& 0\\
0 \arrow[r] & I \arrow[r] & E' \arrow[r] & B \arrow[r] & 0
\end{tikzcd}
\]
where $f$ is a morphism of $B$-modules and $g, h$ are morphisms of $A$-algebras. 

\subsubsection{Functors}\label{sec:exal-functors}
If $I \rightarrow I'$ is a morphism of $B$-modules, 
$B'\rightarrow B$ is a morphism of $A$-algebras,
and $A' \rightarrow A$ is a morphism of rings, then we have natural functors
\begin{gather}
\label{eq:changeI}
\Exal_A(B, I) \rightarrow \Exal_A(B, I')\\
\label{eq:changeB}
\Exal_A(B, I) \rightarrow \Exal_A(B', I_{B'})\\
\label{eq:changeA}
\Exal_A(B, I)\rightarrow \Exal_{A'}(B, I)
\end{gather}
defined in \cite[Equ~III.1.1.5.2]{illusie}, \cite[Equ~III.1.1.5.3]{illusie}, and 
\cite[Equ~III.1.1.5.4]{illusie}, respectively. Here, $I_{B'}$ denotes the sheaf $I$ considered as a $B'$-module.
Let $\cS' \rightarrow \cS$ be a continuous morphism of sites inducing a morphism of topoi $(p^{-1}, p_*)$.\note{every morphism of topoi is induced by a continuous morphism of sites, \cite[Ex~2.N(iii)]{olsson-book}} Then we have an induced morphism
\begin{equation}\label{eq:changeC}
\Exal_A(B, I) \rightarrow \Exal_{p^{-1}A}(p^{-1}B, p^{-1}I)
\end{equation}
sending \eqref{eq:start} to its image under $p^{-1}$. We are using that $p^{-1}$ is an exact functor.

\begin{lemma}\label{lem:commutes-with-I}
The morphisms \eqref{eq:changeI}, \eqref{eq:changeB}, \eqref{eq:changeA}, and  \eqref{eq:changeC} commute pairwise.
\end{lemma}

\begin{proof}
The most involved pair to check is \eqref{eq:changeI} and \eqref{eq:changeB}. We work it out in detail and offer a few words about the remaining pairs at the end of the proof. 
When we say \eqref{eq:changeI} and \eqref{eq:changeB} commute, we mean that if $B' \rightarrow B$ is a morphism of rings and $I \rightarrow I'$ is a morphism of $B$-modules, then the diagram
\begin{equation}\label{eq:commutes-with-I1}
\begin{tikzcd}
\Exal_A(B, I) \arrow[r, "\text{\eqref{eq:changeB}}"] \arrow[d, "\text{\eqref{eq:changeI}}"] & \Exal_A(B', I_{B'}) \arrow[d, "\text{\eqref{eq:changeI}}"] \\
\Exal_A(B, I') \arrow[r, "\text{\eqref{eq:changeB}}"] & \Exal_A(B', I'_{B'})
\end{tikzcd}
\end{equation}
commutes up to a natural transformation.

Given an element \eqref{eq:start} of $\Exal_A(B, I)$, we have a diagram
\begin{equation}\label{eq:commutes-with-I3}
\begin{tikzcd}
& I' \arrow[r, "a'"] & P \arrow[dr, dashrightarrow, "{(0, b)}"]\\
0 \arrow[r] & I \arrow[u] \arrow[r, "a"] \arrow[dr, dashrightarrow, "{(a, 0)}"'] & E \arrow[u, "\iota_E"] \arrow[r, "b"] & B \arrow[r]  & 0\\
&&F \arrow[u, "p_E"'] \arrow[r, "b'"] & B' \arrow[u]
\end{tikzcd}
\end{equation}
where $P = I' \oplus_I E$ and $F = E \times_B B'$: as abelian groups, $P$ and $F$ are the colimit and limit of the usual diagrams, while the ring structures are described in \cite[$0_{\mathrm{IV}}$.18.2.8]{ega} and \cite[$0_{\mathrm{IV}}$.18.1.5]{ega}.\note{you have to sheafify to extend the definition of $I' \oplus_I E$ to sheaves because there is a quotient. note sheafification of a presheaf of rings is a sheaf of rings, 00YR.} Set $Q = I' \oplus_I F$ and $G = P \times_B B'$. Then 
\[
0 \rightarrow I' \rightarrow Q \rightarrow B' \rightarrow0
\]
is the image of \eqref{eq:start} under the composition ${}^\rightarrow \downarrow$ in \eqref{eq:commutes-with-I1}, and likewise $G$ defines the image under the composition $\downarrow{}_\rightarrow$. An arrow from $Q$ to $G$ in the groupoid $\Exal_A(B', I')$ is given by four dashed arrows so that this diagram commutes:
\begin{equation}\label{eq:commutes-with-I2}
\begin{tikzcd}
I \arrow[r] \arrow[dr] & I' \arrow[r, dashrightarrow] \arrow[dr, dashrightarrow] & P \arrow[dr]\\
&F \arrow[ur, dashrightarrow] \arrow[r, dashrightarrow] & B' \arrow[r] & B
\end{tikzcd}
\end{equation}
(To check commutativity, it suffices to check that the quadrilaterals $I'IFP$ and $I'IFB'$ and the perimeter commute.) The required collection of dotted arrows is given by $a': I' \rightarrow P$, $0: I' \rightarrow B'$, $\iota_E \circ p_E: F \rightarrow P$, and $b': F \rightarrow B'$. 

To show that the resulting arrows in $\Exal_A(B', I')$ define a natural transformation (in this groupoid), suppose we are given an arrow
\[
\begin{tikzcd}
0 \arrow[r] & I \arrow[r, "a_1"] \arrow[dr, "a_2"'] & E_1 \arrow[d, "f"]\arrow[r,"b_1"] &B\arrow[r] & 0\\
&&E_2 \arrow[ur, "b_2"']
\end{tikzcd}
\]
in $\Exal_A(B, I)$. Let $f_P: P_1 \rightarrow P_2$ and $f_F: F_1 \rightarrow F_2$ be the maps induced by $f$, where $P_i$ and $F_i$ are defined as in \eqref{eq:commutes-with-I3}. Likewise let $Q_i$ and $G_i$ be the images of $E_i$ in $\Exal_A(B', I')$ under the maps in \eqref{eq:commutes-with-I1}. We must compare two maps from $Q_1$ to $G_2$ in $\Exal_A(B', I')$. Such maps are given by diagrams of the form \eqref{eq:commutes-with-I2} with $F$ replaced by $F_1$ and $P$ replaced by $P_2$. In the situation at hand, one of the maps from $Q_1$ to $G_2$ is given by the diagram
\[
\begin{tikzcd}[column sep = 12 em]
 I' \arrow[r, dashrightarrow, "f_P\circ a_1'"] \arrow[dr, dashrightarrow, sloped, "0"', pos=0.15] & P_2 \\
F_1 \arrow[ur, dashrightarrow, sloped, "f_P\circ \iota_{E_1}\circ p_{E_1}", pos=.2] \arrow[r, dashrightarrow, "b_1'"'] & B' 
\end{tikzcd}
\]
and the other is given by the diagram
\[
\begin{tikzcd}[column sep=12em]
I' \arrow[r, dashrightarrow, "a_2'"] \arrow[dr, dashrightarrow, sloped, "0"', pos=.15] & P_2\\
F_1 \arrow[ur, dashrightarrow, sloped, "\iota_{E_2}\circ p_{E_2} \circ f_F", pos=.19] \arrow[r, dashrightarrow, "b_2'\circ f_F"'] & B'
\end{tikzcd}
\]
These are easily seen to consist of the same morphisms.

This completes the proof that \eqref{eq:changeB} and \eqref{eq:changeI} commute. Of the remaining pairs, most of the checks are trivial (in particular, the analog of \eqref{eq:commutes-with-I1} is strictly commutative). Only the pairs (\eqref{eq:changeI}, \eqref{eq:changeC}) and (\eqref{eq:changeB}, \eqref{eq:changeC}) are nontrivial. For these, one uses that $p^{-1}$ is exact and hence preserves finite limits and colimits.
\end{proof}

\subsection{Illusie's theorem}\label{sec:core}

\subsubsection{Statement}
In this section, $\cS$ is a site with $A \rightarrow B$ a morphism of sheaves of rings on $\cS$, and $I$ be a sheaf of $B$-modules.
\begin{theorem}[{\cite[Thm~A.7]{olsson-deformation}, \cite[Sec~III.1.2.2]{illusie}}] There is a canonical isomorphism
\begin{equation}\label{eq:illusie}
\beta: \Exal_{A}(B, I) \rightarrow \Estackbase{\LL_{B/A}}{I[1]}{B},
\end{equation}
where the right hand side was defined in \eqref{eq:def-ext}.
\end{theorem}
\begin{proof}
 Since the isomophism in \cite[Thm~A.7]{olsson-deformation} is defined on groupoid fibers, we may use the same definition for our morphism \eqref{eq:illusie} (written out in the proof of Lemma \ref{lem:core}), and the argument in \cite[Thm~A.7]{olsson-deformation} shows that it is an isomorphism. Note that when $\sS$ has a final object $S$, the map \eqref{eq:illusie} is the value on $S$ of the isomorphism in \cite[Thm~A.7]{olsson-deformation}.
\end{proof}

\subsubsection{Functoriality}
We will show that \eqref{eq:illusie} is compatible with the functors defined in Section \ref{sec:exal-functors}. We will use the following instances of Situation \ref{basic-situation}.
\begin{example}\label{ex:change-rings} The following is an example of Situation \ref{basic-situation}. Let $\cS$ be a site and $B$ a sheaf of rings on $\cS$. Then $\D(B)$ is a closed symmetric monoidal category with product $\lotimes_B$ and internal hom $\Rhom_B$. If $B' \rightarrow B$ is a flat morphism of sheaves of rings, then extension of scalars $- \otimes_{B'} B: B'\mathrm{-mod} \rightarrow B\mathrm{-mod}$ is strong monoidal and exact with an exact right adjoint $(-)_{B'}$ given by restriction of scalars.

Let $\sC = \D(B')$ and let $\sD = \D(B)$. By \cite[Tag~0DVC]{tag} the functors $- \otimes_{B'}$ and $(-)_{B'}$ extend to an adjoint pair for $\D(B')$ and $\D(B)$, and by \cite[Tags~07A4, 08I6]{tag} the functor $- \otimes_{B'} B: \D(B') \rightarrow \D(B)$ is still strong monoidal.\note{I was confused about 07a4, the ``constituents'': how do you get a map TO the infinite product? the answer is that the total complex is defined in 012Z uses direct sum, not product. (people also sometimes use the product total complex; they are good for different things.) you can always take the direct sum of some maps.}\note{if $M^\bullet, N^\bullet$ are complexes of $B'$-modules such that one of them is $K$-flat, the desired isomorphism of complexes $(M\otimes_{B'} B) \lotimes_B (N \otimes_{B'} B) \rightarrow (M\lotimes_{B'}N)\otimes_{B'} B$ is given in degree $n$ by the sum of sheaf maps
\[
\bigoplus_{p+q=n} \left( (M^p\otimes_{B'} B) \otimes_B (N^q \otimes_{B'} B) \rightarrow (M^p\otimes_{B'}N^q)\otimes_{B'} B\right).
\]}

In addition, it follows from \cite[Tags~08J9, 0A90, 0A5Y]{tag} that if $M^\bullet$ is $K$-flat and $N^\bullet$ is injective, the counit $\Rhom_{B'}(M, N)\lotimes_{B'} M \rightarrow N$ is given in degree $n$ by a product over $p+q+r=n$ of the sheaf maps
\[
\Shom_{B'}(M^{-p}, N^q)\otimes_{B'} M^r \rightarrow N^n,
\]
where this map is equal to the usual evaluation map if $q=n$ and it is zero otherwise. We will give an explicit description of \eqref{eq:sheafy-push} in the proof of Lemma \ref{lem:last-one}.
\end{example}

\begin{example}\label{ex:change-topoi}
If $p: (\sC, \OO_{\sC}) \rightarrow (\sD, \OO_{\sD})$ is a flat morphism of ringed topoi given by an adjoint pair $(p^{-1}, p_*)$,\note{I need to work with RINGED topoi so that I have derived categories)} then $p^*$ is exact and hence defines a strong monoidal functor $\D(\OO_{\sD}) \rightarrow \D(\OO_{\sC})$ with a right adjoint $\R p_*$.
We will give an explicit description of \eqref{eq:sheafy-pullback} in the proof of Lemma \ref{lem:other-last-guy}.
\end{example}

\begin{lemma}\label{lem:core}
The isomorphism \eqref{eq:illusie} is functorial as follows.
\begin{enumerate}
\item Let $A \rightarrow B$ be a map of sheaves of rings on $\cS$. If $I \rightarrow J$ is a morphism of $B$-modules, there is a commuting diagram
\[
\begin{tikzcd}
\Exal_A(B, I) \arrow[r, "\beta"] \arrow[d, "\text{\eqref{eq:changeI}}"] & \Estackbase{\LL_{B/A}}{I[1]}{B} \arrow[d] \\
\Exal_A(B, J) \arrow[r, "\beta"] &\Estackbase{\LL_{B/A}}{J[1]}{B} 
\end{tikzcd}
\]
\item If there is a commuting square of rings with $B' \rightarrow B$ flat\note{I do not know how to make \eqref{eq:simplicial1} explicit if $B' \to B$ is not flat, and for this square to commute I need everything to be explicit.}
\[
\begin{tikzcd}[cramped]
A' \arrow[r] \arrow[d] & A \arrow[d]\\
B' \arrow[r] & B
\end{tikzcd}
\]
then the canonical map $\LL_{B'/A'}\lotimes_{B'} B \rightarrow \LL_{B/A}$ induces a commuting square
\begin{equation}\label{eq:core2}
\begin{tikzcd}
\Exal_A(B, I) \arrow[d, "\text{\eqref{eq:changeB}}\circ\text{\eqref{eq:changeA}}"] \arrow[rr, "\beta"] &&\Estackbase{\LL_{B/A}}{I[1]}{B}\arrow[d]\\
\Exal_{A'}(B', I_{B'})  \arrow[r, "\beta"] &\Estackbase{\LL_{B'/A'}}{I_{B'}[1]}{B'}\arrow[r,  "\text{\eqref{eq:simplicial1}}"]&\Estackbase{\LL_{B'/A'}\otimes_{B'} B}{I[1]}{B}
\end{tikzcd}
\end{equation}
where \eqref{eq:simplicial1} is defined in the context of Example \ref{ex:change-rings}.
\item Let $(\cS, \OO_{\cS}) \rightarrow (\cS', \OO_{\cS'})$ be a continuous morphism of ringed sites inducing a flat morphism of topoi $(p^{-1}, p_*)$. Let $A$ and $B=\OO_{\cS}$ be sheaves of rings on $\cS$. Then if $I$ is a sheaf of $B$-modules, there is a commuting diagram
{\small
\begin{equation}\label{eq:core3}
\begin{tikzcd}[column sep=small]
\Exal_A(B, I) \arrow[rr, "\beta"] \arrow[d, "\text{\eqref{eq:changeC}}"]&& \Estackbase{\LL_{B/A}}{I[1]}{B} \arrow[d, "\text{\eqref{eq:sheafy-pullback}}"] \\
\Exal_{p^{-1}A}(p^{-1}B, p^{-1}I) \arrow[r, "\beta"] & \Estackbase{p^{-1}\LL_{B/A}}{p^{-1}I[1]}{p^{-1}B} \arrow[r, "\text{\eqref{eq:sheafy-pullback}}"] & \Estackbase{p^*\LL_{B/A}}{p^*I[1]}{p^* B}
\end{tikzcd}
\end{equation}
}
where the horizontal instance of \eqref{eq:sheafy-pullback} is defined in the context of Example \ref{ex:change-rings} and the vertical instance of \eqref{eq:sheafy-pullback} is defined in the context of Example \ref{ex:change-topoi}, and we have suppressed an isomorphism induced by $\LL_{p^{-1}B/p^{-1}A} \simeq p^{-1}\LL_{B/A}$.\note{need ringed topoi in order to work with derived categories (necessary since we have cotangent complex); then we need flatness so that $p^*I$ is a sheaf, not a complex.}
\end{enumerate}
\end{lemma}

\begin{proof}
We summarize the definition of $\beta$; see \cite[Thm~A.7]{olsson-deformation} for more details. Let $P_\bullet$ be the simplicial $A$-algebra given by the standard free resolution of the $A$-algebra $B$ \cite[Tag~08SR]{tag}. The morphism $\beta$ is defined to be the composition
{\footnotesize
\[
\Exal_A(B, I) \xrightarrow{\beta_1} \Exal_A(P_\bullet, I) \xrightarrow{\beta_2} \Extcat(\Omega_{P_\bullet/A}, I) \xrightarrow{\beta_3} \Extcat(\Omega_{P_\bullet/A}\otimes B, I) \xrightarrow{\beta_4} \Estack{\LL_{B/A}}{I[1]}
\]
}
Here, $\Extcat(\Omega_\bullet, I)$ denotes the Picard category of simplicial $\OO_{\cC}$-module extensions of $\Omega_\bullet$ by $I$ (viewed as a simplicial module); see \cite[Sec~A.1]{olsson-deformation}.
The map $\beta_1$ is given by the map \eqref{eq:changeB} applied to the augmentation $P_\bullet \rightarrow B$, the morphism $\beta_2$ is given by taking differentials, $\beta_3$ is given by tensoring with $B$, and $\beta_4$ is the functorial isomorphism in \cite[Prop~A.3]{olsson-deformation}.
\note{You should go read Olsson A.7 to understand these maps better. For $\beta_1$, you should vie the operation \eqref{eq:changeB} as happening on the simplicial topos where $I$, $B', B$, and $P_{\bullet}$ are all usual $A$-modules (where $A$ is the constant simplicial object, but it is just an ordinary ring object in the simplicial topos). The map $\beta_2$ is basically tag 00RU: a surjective map o algebras makes an exact sequence of modules of differentials. Now you apply $-\otimes_{P_\bullet} B$ to the WHOLE SEQUENCE of differentials. There are two things to check. First, the sequence remain exact because $\Omega^1_{P_\bullet/A}$ is free, hence flat, and so you can apply something like 004L. Second, the counit $I_{P_\bullet}\otimes_{P_\bullet} B \rightarrow I$ is an isomorphism because restriction of scalars for $P_\bullet \rightarrow B$ is fully faithful, because $P_\bullet \rightarrow B$ is an epimorphism (at least it is an epimorphism at every level, so I hope this means it is epi). See MSE ``does ring epimorphism induce reversed . . . ''}\\

\noindent
\textit{{Proof of (1).}}
The desired functoriality follows from a commuting diagram
{\footnotesize
\[
\begin{tikzcd}[column sep=small]
\Exal_A(B, I) \arrow[r, "{\beta_1}"]\arrow[d, "\text{\eqref{eq:changeI}}"]& \Exal_A(P_\bullet, I) \arrow[r, "{\beta_2}"]\arrow[d, "\text{\eqref{eq:changeI}}"]& \Extcat(\Omega_{P_\bullet/A}, I) \arrow[r, "{\beta_3}"] \arrow[d]& \Extcat(\Omega_{P_\bullet/A}\otimes B, I) \arrow[d] \arrow[r, "\beta_4"] & \Estack{\LL_{B/A}}{I[1]}\arrow[d]\\
\Exal_{A}(B, J) \arrow[r, "{\beta_1}"]& \Exal_{A}(P_\bullet, J) \arrow[r, "{\beta_2}"]& \Extcat(\Omega_{P_\bullet/A}, J) \arrow[r, "{\beta_3}"] & \Extcat(\Omega_{P_\bullet/A}\otimes B, J)\arrow[r, "\beta_4"] & \Estack{\LL_{B/A}}{J[1]} 
\end{tikzcd}
\]
}
The square with $\beta_1$ commutes by Lemma \ref{lem:commutes-with-I}. The square with $\beta_2$ commutes because differentials commute with colimits \cite[Tag~031G]{tag}. The square with $\beta_3$ commutes because tensor product is a left adjoint and so commutes with colimits, and the square with $\beta_4$ commutes by the naturality in \cite[Prop~A.3]{olsson-deformation}.
\\

\noindent
\textit{{Proof of (2).} }
The desired functoriality follows from two commuting diagrams. First we have
{\footnotesize
\begin{equation}\label{eq:core4}
\begin{tikzcd}
\Exal_A(B, I) \arrow[r, "{\beta_1}"]\arrow[d]& \Exal_A(P_\bullet, I) \arrow[r, "{\beta_2}"]\arrow[d]& \Extcat(\Omega_{P_\bullet/A}, I) \arrow[r, "{\beta_3}"] \arrow[d]& \Extcat(\Omega_{P_\bullet/A}\otimes B, I) \arrow[d]\\
\Exal_{A'}(B', I_{B'}) \arrow[r, "{\beta_1}"]& \Exal_{A'}(P_\bullet', I_{B'}) \arrow[r, "{\beta_2}"]& \Extcat(\Omega_{P_\bullet'/A'}, I_{P'_\bullet}) \arrow[r, "{\beta_3}"] & \Extcat(\Omega_{P_\bullet'/A'}\otimes B', I_{B'}) 
\end{tikzcd}
\end{equation}
}
which we claim commutes. Here $P'_\bullet$ is the simplicial $A'$-algebra that is the standard resolution of $B'$. The two left vertical arrows are given by \eqref{eq:changeA} and \eqref{eq:changeB}; the next two vertical arrows are given by the analog of \eqref{eq:changeB} for the $\Extcat$ categories. The first square commutes by Lemma \ref{lem:commutes-with-I}. 
The commutativity of the squares with $\beta_2$ and $\beta_3$ may be checked with the same type of computation used in Lemma \ref{lem:commutes-with-I} and we will be brief here. 

For the square with $\beta_2$, if $0 \rightarrow I \rightarrow E_\bullet \rightarrow P_\bullet \rightarrow 0$ is an object of $\Exal_A(P_\bullet, I)$, then the natural transformation is given on this object by the (iso)morphism\note{It is an isomorphism because it is a morphism in a groupoid. To show that it is a morphism in the groupoid you must check that it commutes with maps to $\Omega_{P'_\bullet/A'}$ (easy) and from $I_{P'_\bullet}$ (a little harder). check that this morphism commutes wtih maps of $E$ and that it is compatible with restriction (this last uses that $a^{-1}\Omega_{A/B} = \Omega_{a^{-1}A/a^{-1}B}$.}
\[
\Omega_{E_\bullet \times_{P_\bullet} P'_\bullet/A'} \rightarrow \Omega_{E_\bullet/A} \times_{\Omega_{P_\bullet/A}} \Omega_{P'_\bullet/A'}
\]
induced by the commuting cube
\[
  \begin{tikzcd}[row sep=tiny, column sep=tiny]
& E_\bullet  \arrow[rr] \arrow[from=dd] & & P_\bullet   \\
E_\bullet \times_{P_\bullet} P'_\bullet \arrow[rr, crossing over, ]\arrow[ur]  & & P'_\bullet \arrow[ur]\\
& A  \arrow[rr, equal] & & A  \arrow[uu]\\
A' \arrow[rr, equal]\arrow[ur]\arrow[uu] & & A' \arrow[ur]\arrow[uu, crossing over]\\
\end{tikzcd}
\]

For the square with $\beta_3$, if $0 \rightarrow I \rightarrow E_\bullet \rightarrow \Omega_{P_\bullet/A} \rightarrow 0$ is an object of $\Extcat_{P_\bullet}(\Omega_{P_\bullet/A}, I)$, then the natural transformation is given on this object by the (iso)morphism of $B'$-modules
\[
(E_\bullet\times_{\Omega_{P_\bullet/A}} \Omega_{P'_\bullet/A'}) \otimes_{P'_\bullet} B' \rightarrow (X\otimes_P B) \times_{\Omega_{P_\bullet/A}\otimes_{P_\bullet} B}(\Omega_{P'_\bullet/A'} \otimes_{P'_\bullet}B')
\]
induced by the natural map of $P'_\bullet$-modules
\[
E_\bullet\times_{\Omega_{P_\bullet/A}} \Omega_{P'_\bullet/A'} \rightarrow (E_\bullet\otimes_{P_\bullet} B) \times_{\Omega_{P_\bullet/A}\otimes_{P_\bullet} B}(\Omega_{P'_\bullet/A'} \otimes_{P'_\bullet}B').
\]

The second diagram comprising \eqref{eq:core2} is as follows.
{\footnotesize
\[
\begin{tikzcd}
\Extcat_B(\Omega_{P_\bullet/A}\otimes B, I) \arrow[r, "{\beta_4}"]\arrow[d, "\rho"]&[-10pt] \Estackbase{\LL_{B/A}}{I[1]}{B} \arrow[d, "\text{\eqref{eq:sheafy-push}}"]\arrow[dr, "\epsilon_{\LL_{B/A}}"]\\
\Extcat_{B'}((\Omega_{P_\bullet/A}\otimes B)_{B'}, I_{B'}) \arrow[r, "{\beta_4}"]\arrow[d, "\text{\eqref{eq:changeB}}"]& \Estackbase{(\LL_{B/A})_{B'}}{I_{B'}[1]}{B'} \arrow[d]\arrow[r, equal, "\text{\eqref{eq:simplicial1}}"]& \Estackbase{(\LL_{B/A})_{B'}\otimes_{B'}B}{I[1]}{B} \arrow[d]\\
\Extcat_{B'}(\Omega_{(P_\bullet)'/A'}\otimes B', I_{B'}) \arrow[r, "{\beta_4}"]& \Estackbase{\LL_{B'/A'}}{I_{B'}[1]}{B'} \arrow[r, equal, "\text{\eqref{eq:simplicial1}}"] & \Estackbase{\LL_{B'/A'}\otimes_{B'}B}{I[1]}{B}
\end{tikzcd}
\]
}
The arrow labeled $\rho$ sends an extension of $B$-modules to the extension of $B'$-modules obtained by restriction of scalars (an exact functor). One may check directly that the composition of the left vertical arrows is equal to the right vertical arrow in \eqref{eq:core4}. 
The map labeled \eqref{eq:sheafy-push} is in the context of Example \ref{ex:change-rings}\note{$Estack$ includes a map that forgets the B (or $B'$)-module structure and only remembers the complex of abelian groups)} and the triangle commutes by definition, while the top left square commutes by Lemma \ref{lem:last-one} below.
The unlabeled vertical maps are induced by the canonical map $(\LL_{B/A})_{B'} \rightarrow \LL_{B'/A'}$, so the bottom squares commute by functoriality of $\beta_4$ and \eqref{eq:simplicial1}. \\

\noindent
\textit{{Proof of (3).}}
Let $A' = p^{-1}A$, let $B' = p^{-1}B$, and let $P'_\bullet$ denote the standard resolution of $B'$ as an $A'$-algebra. The desired commuting square comes from two commuting diagrams. First, we have
{\footnotesize
\begin{equation}\label{eq:part3}
\begin{tikzcd}[column sep=small]
\Exal_A(B, I) \arrow[r, "\beta_1"] \arrow[d, "p^{-1}"] & \Exal_A(P_\bullet, I) \arrow[r, "\beta_2"] \arrow[d, "p^{-1}"] & \Extcat(\Omega_{P_\bullet/A}, I) \arrow[r, "\beta_3"]\arrow[d] & \Extcat(\Omega_{P_\bullet/A}\otimes B, I)\arrow[d]\\
\Exal_{A'}(B', a^{-1}I) \arrow[r, "\beta_1"] & \Exal_{A'}(P'_\bullet, p^{-1}I) \arrow[r, "\beta_2"] & \Extcat(\Omega_{P'_\bullet/A'}, p^{-1}I) \arrow[r, "\beta_3"]  & \Extcat(\Omega_{P'_\bullet/A'}\otimes B', p^{-1}I)
\end{tikzcd}
\end{equation}
}
The first square commutes by 
Lemma \ref{lem:commutes-with-I}.
The third vertical map is induced by $p^{-1}$ and the isomorphism $p^{-1}\Omega_{P_\bullet/A} \simeq \Omega_{p^{-1}P_\bullet/p^{-1}A}$ (\cite[Tag~08TQ]{tag}), and the fourth is induced by $p^{-1}$ and the isomorphism $p^{-1}(\Omega_{P_\bullet/A} \otimes B)\simeq \Omega_{p^{-1}P_\bullet/p^{-1}A}\otimes p^{-1}B$ (\cite[Tag~03EL]{tag}), and the squares commute by functoriality of the same isomorphisms.\note{the tag 03EL is for $f^*$, not $f^{-1}$, but $f^{-1}$ is just the special case when the domain is ringed by the inverse image of the target ring.}
Second, we have\note{the middle line really maps to $p_*\Estackbase{N(p^{-1}(\Omega_{P/A} \otimes B)}{p^{-1}I}{p^{-1}B}$ (the bottom right corner also maps to this one) and then to $p_*\Estackbase{p^{-1}\LL_{B/A}}{p^{-1}I[1]}{p^{-1}B}$}
{\small
\[
\begin{tikzcd}[column sep=small]
\Extcat(\Omega_{P_\bullet/A}\otimes B, I) \arrow[r, "\beta_4"] \arrow[d, "p^{-1}"] & \Estackbase{\LL_{B/A}}{I[1]}{B} \arrow[d, "\text{\eqref{eq:sheafy-pullback}}"] \arrow[dr, "\text{\eqref{eq:sheafy-pullback}}"]\\
\Extcat(p^{-1}(\Omega_{P_\bullet/A} \otimes B), p^{-1}I) \arrow[d, equal] \arrow[r] &  \arrow[d, equal] \Estackbase{p^{-1}\LL_{B/A}}{p^{-1}I[1]}{p^{-1}B} \arrow[r, "\text{\eqref{eq:sheafy-pullback}}"] &\Estackbase{p^*\LL_{B/A}}{p^*I[1]}{p^*B}\\
\Extcat(\Omega_{P'_\bullet/A'}\otimes B', p^{-1}I)  \arrow[r, "\beta_4"]& p_*\Estackbase{\LL_{B'/A'}}{p^{-1}I[1]}{p^{-1}B} & 
\end{tikzcd}
\]
}
The composition of the left vertical arrows is equal to the right vertical arrow in \eqref{eq:part3}. The middle horizontal arrow is comprised of $\beta_4$ and an isomorphism (see Lemma \ref{lem:other-last-guy}), and the top left square commutes by Lemma \ref{lem:other-last-guy}. The bottom square commutes by functoriality of $\beta_4$, and the triangle commutes by functoriality of  \eqref{eq:sheafy-pullback}
in the functors.
\end{proof}

\begin{lemma}\label{lem:last-one}
Let $(\sS, \OO_{\sS})$ be a ringed site, let $\Omega_\bullet$ be a simplicial $\Omega_\sS$-module, and let $I$ be an $\Omega_{\sS}$-module. Let $\OO_{\sS}' \rightarrow\OO_{\sS}$ be a flat morphism of rings. There is a commuting diagram
\begin{equation}\label{eq:last-one2}
\begin{tikzcd}
\Extcat_{\OO_\sS}(\Omega_\bullet, I) \arrow[d, "\rho"] \arrow[rr, "\beta_4"] &[-15pt]&[-15pt] \Estackbase{N(\Omega_\bullet)}{I[1]}{\OO_{\sS}} \arrow[d, "\text{\eqref{eq:sheafy-push}}"]\\
\Extcat_{\OO_{\sS}'}((\Omega_\bullet)_{\OO_{\sS}'}, I_{\OO_{\sS}'}) \arrow[r, "\beta_4"] & \Estackbase{N((\Omega_\bullet)_{\OO_{\sS}'})}{(I)_{\OO_{\sS}'}[1]}{\OO'_{\sS}} \arrow[r, equal] & \Estackbase{(N(\Omega_\bullet))_{\OO_{\sS}'}}{(I)_{\OO_{\sS}'}[1]}{\OO'_{\sS}} 
\end{tikzcd}
\end{equation}
where $N(\Omega_\bullet)$ is the normalization of the Moore complex associated to $\Omega_\bullet$ (see \cite[Tag~0194]{tag}), $\beta_4$ is the isomorphism of \cite[Prop~A.3]{olsson-sheaves}, $\rho$ applies restriction of scalars to an exact sequence, and \eqref{eq:sheafy-push} is in the context of Example \ref{ex:change-rings}\note{to understand the equality, read the definition of $N$ in 0194 and 017V. it is induced by the commutativity of $p^{-1}$ with kernels.}.
\end{lemma}
\begin{proof}
By \cite[Tag~05NI, 05T7]{tag} there is a quasi-isomorphism $N \rightarrow N(\Omega_\bullet)$ from a complex $N\in D^{[-\infty, 0]}(\OO_{\sS})$ of flat $\OO_{\sS}$-modules. We enlarge \eqref{eq:last-one2} on its right side by composing with the square induced by $N \rightarrow N(\Omega_\bullet)$ and show that the perimeter of the new diagram commutes.\note{this is ok because it differs from the original diagram by an isomorphism, and because the new square is commutative using the abstract defintion of \eqref{eq:sheafy-push}.} From the definition of $\beta_4$ we may assume $I$ is injective. To simplify notation let $B = \OO_{\sS}$ and $B' = \OO_{\sS}'$.

Most of the work is to describe \eqref{eq:sheafy-push} explicitly. To this end we first recall the definition of \eqref{eq:simplicial1}: in the notation of Section \ref{sec:csmcats}, it is the image of the composition
\begin{equation}\label{eq:last-guy1}
f^*\Hom(X, f_*Y)\otimes f^*X \xrightarrow[\sim]{\text{\eqref{eq:strong-monoidal}}} f^*(\Hom(X, f_*Y)\otimes X) \xrightarrow{f^*(\epsilon^{\otimes}_{f_*Y})} f^*f_*Y \xrightarrow{\epsilon^{f^*}_Y} Y
\end{equation}
under the $(\otimes, \Hom)$ adjunction and then the $(f^*, f_*)$ adjunction. One sees using the description of $\epsilon^\otimes$ in Example \ref{ex:change-rings} that with $M \in \D^{[-\infty, 0]}(B')$ a complex of flat $B'$-modules and $I$ as in the previous paragraph, the morphism $(\Rhom_{B'}(M, (I)_{B'}[1])\otimes_{B'}B)\lotimes_B (M\otimes_{B'} B) \rightarrow I[1]$ of \eqref{eq:last-guy1} is given by a product over $p+r=0$ of the canonical sheaf maps 
\[
(\Shom_{B'}(M^{-p}, (I)_{B'})\otimes_{B'}B) \otimes_B(M^r\otimes_{B'} B) \rightarrow I.
\]
 We are using the fact that $(-)_{B'}$ preserves injectives (since it has an exact left adjoint) and hence $(I)_{B'}$ is injective. We see that \eqref{eq:simplicial1}: $\Rhom_{B'}(M, (I)_{B'}[1]) \to (\Rhom_B(M\otimes_{B'}B, I[1])_{B'})$ is given in degree $n$ by the canonical sheaf map 
\[
\Shom_{B'}(M^{-n-1}, (I)_{B'}) \to (\Shom_B(M^{-n-1}\otimes_{B'}B, I)_{B'}).
\]
To compute \eqref{eq:sheafy-push}, given $N$ as at the beginning of this proof, we note that $(N)_{B'} \in \D^{[-\infty, 0]}$ is a complex of flat $B'$-modules by \cite[Tag~00HC]{tag}, so our previous description of \eqref{eq:simplicial1} applies with $M = (N)_{B'}$. From the definition of 
\begin{equation}\label{eq:whatamess}
\text{\eqref{eq:sheafy-push}}: (\Rhom_B(N, I[1]))_{B'} \rightarrow \Rhom_{B'}((N)_{B'}, (I)_{B'}[1])
\end{equation}
we see that it is given in degree $n$ by the usual sheaf map 
\[
\Shom_B(N^{-n-1}, I) \to \Shom_{B'}((N^{n-1})_{B'}, (I)_{B'}).
\]
We are interested in the cases $n=0$ and $n=-1$. Given $U \in \sS$ and a section $f:N^{n-1}|_U \rightarrow I|_U$ of the left hand side, i.e. a morphism of $B$-modules, this map sends $f$ to the corresponding morphism of $B'$-modules. (One may verify this claim by unwinding the definitions and ultimately appealing to Example \ref{ex:modules}.)
Now applying $\R \Gamma$ to \eqref{eq:whatamess} is straightforward since both complexes are complexes of injectives by \cite[Tag~0A96]{tag}.\note{immediately this gives an explicit description of the right vertical arrow in \eqref{eq:last-one2} at the level of Picard prestacks on a site with one object and one morphism. But a stack on this site is the same as a prestack on this site: it is just a category. So i have completely described the right vertical arrow.}

This gives a completely explicit description of the morphism labeled \eqref{eq:sheafy-push} in \eqref{eq:last-one2}. With this in hand it is easy to check that \eqref{eq:last-one2} commutes.\note{essentially this is because the functor $\beta_4$ only cares about MAPS of modules, which don't really see the module structure. For example, the kernel of a map of $B$ modules is the same as the kernel of the map of $B'$-modules. I guess it is important that restriction of scalars is exact.}

\end{proof}

\begin{lemma}\label{lem:other-last-guy}
Let $(\sS, \OO_{\sS})$ be a ringed site, let $\Omega_\bullet$ be a simplicial $\Omega_\sS$-module, and let $I$ be an $\Omega_{\sS}$-module. Let $(\sS, \OO_{\sS}) \rightarrow (\sS', \OO_{\sS'})$ be a continuous morphism of ringed sites inducing a flat morphism of topoi $(p^{-1}, p_*)$\note{flatness is used to apply tag 0730 below} such that $p^{-1}\OO_{\sS} = \OO_{\sS'}$. Then there is a commuting diagram
{\small
\begin{equation}\label{eq:other-last-one2}
\begin{tikzcd}
\Extcat_{\OO_{\sS}}(\Omega_\bullet, I)\arrow[d] \arrow[rr, "\beta_4"] &[-13pt]&[-13pt] \Estackbase{N(\Omega_\bullet)}{I[1]}{\OO_{\sS}} \arrow[d, "\text{\eqref{eq:sheafy-pullback}}"] \\
\Extcat_{\OO_{\sS'}}(p^{-1}(\Omega_\bullet), p^{-1}(I)) \arrow[r, "\beta_4"] & \Estackbase{N(p^{-1}\Omega_\bullet)}{p^{-1}I[1]}{\OO_{\sS'}} \arrow[r, equal] &\Estackbase{p^{-1}N(\Omega_\bullet)}{p^{-1}I[1]}{\OO_{\sS'}}
\end{tikzcd}
\end{equation}
}
where $N(\Omega_\bullet)$ is the normalization of the simplical module, $\beta_4$ is the isomorphism of \cite[Prop~A.3]{olsson-sheaves}, \eqref{eq:sheafy-pullback} is in the context of Example \ref{ex:change-topoi}, and the left vertical arrow applies $p^{-1}$ to an exact sequence.\note{to understand the equality, read the definition of $N$ in 0194 and 017V. it is induced by the commutativity of kernels and restriction of scalars (restriction of scalars is exact so it commutes with kernels).}
\end{lemma}
\begin{proof}
By \cite[Tag~05NI, 05T7]{tag} there is a quasi-isomorphism $N \rightarrow N(\Omega_\bullet)$ from a complex $N\in D^{[-\infty, 0]}(\OO_{\sS})$ of flat $\OO_{\sS}$-modules. We enlarge \eqref{eq:other-last-one2} on its right side by composing with the square induced by $N \rightarrow N(\Omega_\bullet)$ and show that the perimeter of the new diagram commutes. From the definition of $\beta_4$ we may assume $I$ is injective. To simplify notation let $B = \OO_{\sS}$ and $B' = \OO_{\sS}'$.

Most of the work is to describe \eqref{eq:sheafy-pullback} explicitly. To this end, we first note that (in the notation of Section \ref{sec:csmcats}) \eqref{eq:sheafy-pullback} is equal to the image of the composition
\begin{equation}\label{eq:other-last-guy1}
f^*\Hom(X, Y) \otimes f^*X \xrightarrow[\sim]{\text{\eqref{eq:strong-monoidal}}} f^*(\Hom(X, Y)\otimes X) \xrightarrow{f^*(\epsilon^{\otimes}_{Y})} f^*Y
\end{equation}
under the $(\otimes, \Hom)$ adjunction and the $(f^*, f_*)$ adjunction. (To see this, use \eqref{eq:last-guy1} and the triangle identity $\epsilon^{f^*}_{f^*Y}\circ f^*\eta^{f^*}_{Y}=1_{f^*Y}$.)
One sees using the description of $\epsilon^{\otimes}$ in Example \ref{ex:change-rings} that with $N$ and $I$ as in the previous paragraph, the morphism \eqref{eq:other-last-guy1}: $p^{-1}\Rhom_B(N, I[1]) \lotimes_{B'} p^{-1}N \rightarrow p^{-1}I[1]$ is given by the product over $r \in \ZZ$ of the usual sheaf maps
\[
p^{-1}\Shom_B(N^r, I)\otimes_{B'} p^{-1}N^r \rightarrow p^{-1}I.
\]
To compute the $(\otimes, \Hom)$ adjunction we must take an injective resolution $p^{-1}I[1] \rightarrow J$ of $p^{-1}I[1]$. Given this, one checks that 
\begin{equation}\label{eq:phoebe}
\text{\eqref{eq:sheafy-pullback}}: \Rhom_B(N, I[1]) \rightarrow \R p_*\Rhom_{B'}(p^{-1}N, J)
\end{equation}
is given in degree $n$ by the composition of the usual sheaf maps\note{a priori the right hand side should be a product over $\Shom(p^{-1}N^{-p}, J^q)$ with $p+q=n$, but the maps all come from adjoints of the previous line (composed with $p^{-1}I[1] \to J$), and those maps are zero in all degrees except -1. So the map is zero unless $q=-1$ and $p=-n-1$.}
\[
\Shom_B(N^{-n-1}, I) \rightarrow p_*\Shom_B'(p^{-1}N^{-n-1}, p^{-1}I) \rightarrow p_*\Shom_{B'}(p^{-1}N^{-n-1}, J^{-1})
\]
We have used \cite[Tag~0A96]{tag} to conclude that $\Rhom_{B'}(p^{-1}N, J)$ is a complex of injectives so its pushforward can be computed termwise. We are interested in the cases $n=0, -1$. Given $U \in \sS$ and a section $f: N^{-n-1}|_U \rightarrow I|_U$ of the left hand side, this map sends $f$ to $p^{-1}f: p^{-1}N^{-n-1}|_U \rightarrow p^{-1}I|_U \to J^{-1}|_U$ (it is an exercise to check that the ``usual sheaf map'' does this). Now applying $\R \Gamma$ to \eqref{eq:phoebe} is straightforward since both complexes are complexes of injectives by \cite[Tag~0730, 0A96]{tag}.\note{for passing from $pch\tau_{\leq 0}$ to $ch$, see note in the previous lemma.}

This gives a completely explicit description of \eqref{eq:sheafy-pullback} in \eqref{eq:other-last-one2}. With this in hand one may check directly that \eqref{eq:other-last-one2} commutes.

\end{proof}

\subsection{Description of \eqref{eq:ft}}\label{sec:defiso}

To define \eqref{eq:ft}, we use another example of Situation \ref{basic-situation}. 

\begin{example}\label{ex:change-simplicial}
Let $\cX$ be an algebraic stack and let $X \rightarrow \cX$ be a smooth cover by an algebraic space $X$. By Proposition \ref{prop:descent-olsson}, the morphism $\varpi^*: \Dqc(\cX_{\liset}) \to \Dqc(U^+_{\bullet, \et})$ is an equivalence of categories. In fact, it follows from the construction of $\varpi^*$ that it is a strong monoidal equivalence of symmetric monoidal categories. A standard argument shows that the inverse equivalence $\R \varpi_{ *}$ is also strong monoidal.\note{there is an earlier note about this with the phrase "helpful point"}
\end{example}

Let $f:\cX \rightarrow \cY$ be a representable morphism of algebraic stacks. Let $Y \rightarrow \cY$ be a smooth cover by a scheme with $ Y^+_\bullet \rightarrow \cY$ the associated strictly simplicial algebraic space and $\varpi: X^+_\bullet \rightarrow \cX$ its pullback to $\cX$. We will use $\varpi^*$ and $\R \varpi_*$ to denote the functors in Example \ref{ex:change-simplicial}.\note{ I think $X^+_{\bullet, \et} = \cX_{\et}|_{X^+, \bullet}$---NO this is false!!! so i need to use the first one. There are NO lisse-etale sites in this section. The reason we need to use simplicial sites is to compute the cotangent complex: since the base is an algebraic stack that's the only way.} 
We recall that the cotangent complex $\LL_{\cX/\cY}$ is defined to be the object in $\Dqc(\cX_{\liset})$ corresponding, under the equivalence $\R \varpi_*$ of Example \ref{ex:change-simplicial}, to the cotangent complex of the morphism of topoi $X^+_{ \bullet, \et} \rightarrow Y^+_{\bullet, \et}$.

\begin{definition}[{\cite{olsson-deformation}}]Let $I\in \Qcoh(\cX_{\liset})$. The isomorphism \eqref{eq:ft} is defined to be the following composition of morphisms of Picard categories on $\cX_{\liset}$:
{\footnotesize
\[
\Exal_{\cY}(\cX, I) \xrightarrow{\alpha} \Exal_{f^{-1}\OO_{Y^+_\bullet}}(\OO_{X^+_\bullet}, \varpi^*I) \xrightarrow{\varpi_*\beta} \Estack{\LL_{X^+_\bullet/Y^+_\bullet}}{ \varpi^*I[1]}\xleftarrow{\gamma} \Estack{\LL_{\cX/\cY}}{I[1]}.
\]}
The objects and maps in this composition are defined as follows.
\begin{itemize}
\item The Picard category $\Exal_{f^{-1}\OO_{Y^+_\bullet}}(\OO_{X^+_\bullet}, \varpi^*I)$ is defined as in Section \ref{sec:core} on the site $X^+_{\bullet, \et}$ 
\item 
The map $\alpha$ is the composition of \cite[(2.8.1),(2.20.1)]{olsson-deformation}: it sends an extension $\cX \rightarrow \cX'$ by $I$ to the exact sequence of $f^{-1}\OO_{Y^+_\bullet}$-modules
\[
0 \rightarrow \varpi^*I \rightarrow \OO_{\cX'^+_\bullet} \rightarrow \OO_{\cX^+_\bullet} \rightarrow 0.
\]
It is an isomorphism by \cite[Prop~2.9, Lem~2.21]{olsson-deformation}.

\item The map $\beta$
was defined in \eqref{eq:illusie}.

\item The arrow $\gamma$ is induced by applying $ch \circ \tau_{\leq 0} \circ \R \Gamma$ to \eqref{eq:sheafy-pullback} in the context of Example \ref{ex:change-simplicial}. It is an isomorphism since $\varpi^*$ is fully faithful (in fact, an equivalence of categories).
\end{itemize}
\end{definition}


\begin{remark} The proof of Lemma \ref{lem:ft-functoriality1} shows that \eqref{eq:ft} is independent of the choice of cover $Y \rightarrow \cY$.
\end{remark}

\subsection{Proofs of Lemmas \ref{lem:ft-functoriality1} and \ref{lem:ft-functoriality2}}\label{sec:lastproofs}
We describe an amalgamation \eqref{eq:3-for-1} of the three diagrams in Lemma \ref{lem:core} that will be used to prove both functoriality lemmas. Let $(\cS, \OO_{\cS}) \rightarrow (\cS', \OO_{\cS'})$ be a continuous morphism of ringed sites inducing a flat morphism of topoi $(p^{-1}, p_*)$. Let $A$ and $B:=\OO_{\cS}$ be sheaves of rings on $\cS$, let $I$ be a sheaf of $B$-modules, and let $A'$ be a sheaf of rings on $\cS'$ such that there is a commuting diagram as follows (note that $p^*B = \OO_{\cS'}$):
\[
\begin{tikzcd}
p^{-1}A \arrow[r] \arrow[d] & A' \arrow[d] \\
p^{-1}B \arrow[r] & p^*B
\end{tikzcd}
\]
Then we obtain the following commuting diagram.
{\footnotesize
\begin{equation}\label{eq:3-for-1}
	\begin{tikzcd}[column sep=tiny]
	\Estackbase{\LL_{B/A}}{I[1]}{B} \arrow[rr]&[-20pt] &[-50pt] \Estackbase{p^*\LL_{B/A}}{p^*I[1]}{p^*B} &[-50pt]&[-20pt] \Estackbase{\LL_{p^*B/A'}}{p^*I[1]}{p^*B} \arrow[ll] \\
	& \Estackbase{p^{-1}\LL_{B/A}}{p^{-1}I[1]}{p^{-1}B} \arrow[ur, "\text{\eqref{eq:sheafy-pullback}}"] \arrow[rr] && \Estackbase{p^{-1}\LL_{B/A}}{(p^*I[1])_{p^{-1}B}}{p^{-1}B}\arrow[ul, "\text{\eqref{eq:simplicial1}}"', "\sim"] & \\
	\Exal_A(B, I) \arrow[uu, "\beta"] \arrow[r] & \Exal_{p^{-1}A}(p^{-1}B, p^{-1}I) \arrow[rr]\arrow[u, "\beta"]&&\Exal_{p^{-1}A}(p^{-1}B, (p^*I)_{p^{-1}B}) \arrow[u, "\beta"]& \arrow[l] \Exal_{A'}(p^*B, p^*I) \arrow[uu, "\beta"]
	\end{tikzcd}
	\end{equation}
	}
	Here, the left square is Lemma \ref{lem:core} (3) and the right square is Lemma \ref{lem:core} (2). The middle square is Lemma \ref{lem:core} (1), using the unit $p^{-1}I \rightarrow (p^*I)_{p^{-1}B}$ of the $(\otimes, \Hom)$ adjunction. The triangle commutes by definition of the maps involved.

\begin{proof}[Proof of Lemma \ref{lem:ft-functoriality1}]
Construct a diagram
\[
\begin{tikzcd}
U\arrow[d, "\rho"] \arrow[r, "r"] &V \arrow[d]\\
Z \arrow[r] \arrow[d, "\varpi"] & \cW \times_{\cY} Y \arrow[r] \arrow[d]&Y\arrow[d]\\
\cZ \arrow[r] & \cW \arrow[r] & \cY
\end{tikzcd}
\]
where $U, V, Z$, and $Y$ are algebraic spaces with $Y \rightarrow \cY$ and $V \rightarrow \cW \times_{\cY} Y$ smooth and surjective and all squares are fibered. Let $q$ denote the map $Z \rightarrow Y$ and let $\varpi' = \varpi\circ \rho$, and use the same letters to denote induced morphisms of (simplicial) topoi. Then commutativity of \eqref{eq:wrap2} is equivalent to commutativity of the following diagram.
{\footnotesize
\begin{equation}\label{eq:wrap3}
\begin{tikzcd}[column sep=small]
\Estack{\LL_{\cZ/\cW}}{I[1]} \arrow[d, "\gamma", "\sim"']\arrow[rr, "A"] && \Estack{\LL_{\cZ/\cY}}{I[1]}\arrow[d, "\gamma", "\sim"'] \arrow[dl]\\
\Estack{\LL_{U^+_\bullet/{V^+_\bullet}}}{\varpi'^*I[1]} \arrow[r] & \Estack{\rho^{*}\LL_{Z^+_\bullet/{Y^+_\bullet}}}{\varpi'^*I[1]}&\arrow[l, "\sim"'] \Estack{\LL_{Z^+_\bullet/{Y^+_\bullet}}}{\varpi^*I[1]} \\
\Exal_{r^{-1}\OO_{V^+_\bullet}}(\OO_{U^+_\bullet}, \varpi'^*I) \arrow[r] \arrow[u, "\beta"]& \Exal_{\rho^{-1}q^{-1}\OO_{Y^+_\bullet}}(\rho^{-1}\OO_{Z^+_\bullet}, \varpi'^*I)  \arrow[u]& \Exal_{q^{-1}\OO_{Y^+_\bullet}}(\OO_{Z^+_\bullet}, \varpi^*I) \arrow[l]\arrow[u, "\beta"] \\
\Exal_{\cW}(\cZ, I) \arrow[rr, "B"]\arrow[u, "\alpha"] && \Exal_{\cY}(\cZ, I)\arrow[u, "\alpha"]
\end{tikzcd}
\end{equation}
}
In the triangle, all of the maps are equal to \eqref{eq:sheafy-pullback}, and the triangle commutes by the functoriality of \eqref{eq:sheafy-pullback} in the adjoint pair. The arrow
\[
\Estack{\rho^{*}\LL_{Z^\bullet/{Y^\bullet}}}{\varpi'^*I[1]} \leftarrow \Estack{\LL_{Z^\bullet/{Y^\bullet}}}{\varpi^*I[1]} 
\]
is an equivalence (as claimed in the diagram) because $\rho^*: \Dqc(Z^+_{\bullet, \et}) \rightarrow \Dqc(U^+_{\bullet,\et})$
is fully faithful\note{since $\varpi_U^*$ and $\varpi_V^*$ are)}.
 The trapezoid commutes by definition of the canonical map $\LL_{\cZ/\cY} \rightarrow \LL_{\cZ/\cW}$ (one can produce an explicit description for $\gamma$ by the same argument as was used in the proof of Lemma \ref{lem:other-last-guy}).\note{to say that the interior commuting implies the perimeter is commuting, you need to use the existence of a secret upward arrow between the middle guys (OK now I added it so it is not secret); you can see it exists in \eqref{eq:3-for-1}.} The commutativity of the middle square is \eqref{eq:3-for-1}, reflected left-to-right, with $p=\rho$, $A = q^{-1}\OO_{Y^+_\bullet}$, $A' = r^{-1}\OO_{V^+_\bullet}$, and $B=\OO_{Z^+_\bullet}$. 

It remains to check that the bottom square commutes. We do this by direct computation. Let $i: \cZ \hookrightarrow \cZ'$ be an element of $\Exal_{\cW}(\cZ, I)$. We have the following commuting diagram, where all squares are fibered:
\begin{equation}\label{eq:mess2}
\begin{tikzcd}
U \arrow[r, "i"] \arrow[d, "\rho"] & U' \arrow[r] \arrow[d, "\rho'"] & V\arrow[d]\\
Z \arrow[r, "i"] \arrow[d, "\varpi"] & Z' \arrow[r] \arrow[d] & \cW \times_\cY Y \arrow[r] \arrow[d] & Y\arrow[d]\\
\cZ \arrow[r, "i"] & \cZ' \arrow[r] & \cW \arrow[r] & \cY
\end{tikzcd}
\end{equation}
The map $\alpha$ sends $i: \cZ \hookrightarrow \cZ'$ to the extension
\[
0 \rightarrow \varpi'{}^*I \xrightarrow{m} i^{-1}\OO_{U'{}^+_\bullet} \rightarrow \OO_{U^+_\bullet} \rightarrow 0
\]
of $r^{-1}\OO_{V^+_\bullet}$-modules, and the maps \eqref{eq:changeA} and \eqref{eq:changeB} send this to the extension
\begin{equation}\label{eq:mess3}
0 \rightarrow \varpi'{}^*I \xrightarrow{(m, 0)} i^{-1}\OO_{U'{}^+_\bullet} \times_{\OO_{U^+_\bullet}}\rho^{-1}\OO_{Z^+_\bullet} \rightarrow \rho^{-1} \OO_{Z^+_\bullet} \rightarrow 0
\end{equation}
of $\rho^{-1}q^{-1}\OO_{Y^+_\bullet}$-modules.

On the other hand, the map $B$ sends $i:\cZ \hookrightarrow \cZ'$ to the same extension, now as an element of $\Exal_{\cY}(\cZ, I)$. The image of this under $\rho^{-1}\circ \alpha$ is
\[
0 \rightarrow \rho^{-1}\varpi^*I \rightarrow\rho^{-1}i^{-1}\OO_{Z'{}^+_\bullet} \xrightarrow{n} \rho^{-1}\OO_{Z^+_\bullet} \rightarrow 0
\]
an extension of $\rho^{-1}q^{-1}\OO_{Y^\bullet}$-modules. Here $n$ is part of the data of the morphism of ringed topoi associated to $i: Z \rightarrow Z'$.
Finally the map \eqref{eq:changeI} sends this extension to 
\begin{equation}\label{eq:mess4}
0 \rightarrow \varpi'{}^*I \rightarrow \varpi'{}^*I \oplus_{\rho^{-1}\varpi^*I}\rho^{-1}i^{-1}\OO_{Z'{}^+_\bullet} \xrightarrow{(0, n)} \rho^{-1}\OO_{Z^+_\bullet} \rightarrow 0
\end{equation}
also an extension of $\rho^{-1}q^{-1}\OO_{Y^+_\bullet}$-modules.

A morphism from \eqref{eq:mess4} to \eqref{eq:mess3} in the groupoid $\Exal_{\rho^{-1}q^{-1}\OO_{Y^+_\bullet}}(\rho^{-1}\OO_{Z^+_\bullet}, \varpi'{}^*I)$ is given by a collection of dotted arrows making the following diagram commute.
\begin{equation}
\begin{tikzcd}
\rho^{-1}\varpi^*I \arrow[r] \arrow[dr] & \varpi'{}^*I \arrow[r, dashrightarrow] \arrow[dr, dashrightarrow] & i^{-1}\OO_{U'{}^+_\bullet} \arrow[dr]\\
&\rho^{-1}i^{-1}\OO_{Z'{}^+_\bullet} \arrow[ur, dashrightarrow] \arrow[r, dashrightarrow] & \rho^{-1}\OO_{Z^+_\bullet} \arrow[r] & \OO_{U^+_\bullet}
\end{tikzcd}
\end{equation}
We choose arrows as follows (note they are compatible with restriction).
\begin{align*}
&m:\varpi'{}^*I \rightarrow i^{-1}\OO_{U'^+_{\bullet}} &n: \rho^{-1}i^{-1}\OO_{Z'{}^+_\bullet} \rightarrow \rho^{-1}\OO_{Z^+_\bullet}\\
&0: \varpi'{}^*I \rightarrow \rho^{-1}\OO_{Z^+_\bullet} &k:\rho^{-1}i^{-1}\OO_{Z'{}^+_\bullet}\rightarrow i^{-1}\OO_{U'^+_{\bullet}}.
\end{align*}
where $k$ is equal to $i^{-1}$ applied to the canonical morphism $\rho^{-1}\OO_{Z'^+_\bullet} \rightarrow \OO_{U'^+_\bullet}$.
Commutativity of the resulting diagram follows from commutativity of \eqref{eq:mess2}.\note{well . . . there are two exact sequences. Commutativity of the two parallelograms just uses that fact. Commutativity of the righthand triangle uses commutativity of the top left square of \eqref{eq:mess2}. Commutativity of the lefthand triangle needs you to say why the unit $\rho^{-1}\varpi^*I \rightarrow \varpi'^*I$ is the same as the induces map of kernels for the two exact sequences. This is because the other two maps connecting the exact triangles are also the same kind of unit: $i^{-1}(\rho^{-1}\OO_{Z'^+_\bullet} \rightarrow \OO_{U'^+_\bullet})$ and $\rho^{-1}\OO_{Z^+_\bullet} \rightarrow \OO_{U^+_\bullet}$}

We claim that this morphism is natural for arrows coming from $\Exal_{\cW}(\cZ, I)$. If we are given an arrow $f$ from $i_1: \cZ \rightarrow \cZ_1$ to $i_2:\cZ \rightarrow \cZ_2$ inducing maps $f_U:i_1^{-1}\OO_{U^+_{1,\bullet}} \rightarrow i_2^{-1}\OO_{U^+_{2,\bullet}}$ and $f_Z:\rho^{-1}i_1^{-1}\OO_{Z^+_{1,\bullet}} \rightarrow\rho^{-1}i_1^{-1}\OO_{Z^+_{2,\bullet}}$, then this naturality is equivalent to the fact that the maps in the following two criss-cross diagrams coincide.
\[
\begin{tikzcd}[column sep=10em]
 \varpi'{}^*I \arrow[r, dashrightarrow, "m_2"] \arrow[dr, dashrightarrow, "0" description, pos=.2] & i^{-1}_2\OO_{U{}^+_{2,\bullet}}\\
\rho^{-1}i^{-1}_1\OO_{Z{}^+_{1,\bullet}} \arrow[ur, dashrightarrow, "k_2 \circ f_Z" description, pos=.2] \arrow[r, dashrightarrow, "n_2 \circ f_Z"'] & \rho^{-1}\OO_{Z^+_\bullet}
\end{tikzcd}
\begin{tikzcd}[column sep=10em]
 \varpi'{}^*I \arrow[r, dashrightarrow, "f_U\circ m_1"] \arrow[dr, dashrightarrow, "0" description, pos=.2] & i^{-1}_2\OO_{U{}^+_{2,\bullet}}\\
\rho^{-1}i^{-1}_1\OO_{Z{}^+_{1,\bullet}} \arrow[ur, dashrightarrow, "f_U\circ k_1" description, pos=.2] \arrow[r, dashrightarrow, "n_1"'] & \rho^{-1}\OO_{Z^+_\bullet}
\end{tikzcd}
\]
\end{proof}

\begin{proof}[Proof of Lemma \ref{lem:ft-functoriality2}]
Construct a fiber diagram
\[
\begin{tikzcd}
Z \arrow[r] \arrow[d, "r"] & X\times_{\cX} \cZ \arrow[r] \arrow[d] & X \arrow[d, "q"]\\
W \arrow[r] & Y\times_{\cY}\cW \arrow[r] & Y
\end{tikzcd}
\]
where  $Y \rightarrow \cY$ is a smooth cover by a scheme, $X = Y \times_{\cY} \cX$, and $W \rightarrow Y\times_{\cY}\cW$ is a smooth cover by a scheme. Use $p$ to denote the map $Z \rightarrow X$. Then commutativity of \eqref{eq:idunno} is equivalent to commutativity of the diagram below.
{\footnotesize
 \begin{equation}\label{eq:mess1}
 \begin{tikzcd}
	\Estack{\LL_{\cZ/\cW}}{I[1]} \arrow[r, "\text{\eqref{eq:sheafy-pullback}}"', "C"]\arrow[d, "\gamma"', "\sim"]&[-12pt] \Estack{p^*\LL_{\cX/\cY}}{p^*I[1]} \arrow[d, "\text{\eqref{eq:sheafy-pullback}}"', "\sim"] &[-12pt] \arrow[l, "\sim", "D"'] \Estack{\LL_{\cZ/\cW}}{p^*I[1]} \arrow[d, "\gamma"', "\sim"] \\
	\Estackbase{\LL_{X^+_\bullet/Y^+_\bullet}}{\varpi^*I[1]}{\OO_{X^+_\bullet}}\arrow[r, "\text{\eqref{eq:sheafy-pullback}}"]& \Estackbase{p^*\LL_{X^+_\bullet/Y^+_\bullet}}{p^*\varpi^*I[1]}{p^*\OO_{X^+_\bullet}} & \arrow[l, "\sim"'] \Estackbase{\LL_{Z^+_\bullet/W^+_\bullet}}{\varpi^*p^*I[1]}{\OO_{Z^+_\bullet}}\\
\Exal_{q^{-1}\OO_{Y^+_\bullet}}(\OO_{X^+_\bullet}, \varpi^*I) \arrow[u, "\beta"] \arrow[rr] & & \Exal_{r^{-1}\OO_{W^+_\bullet}}(\OO_{Z^+_\bullet}, \varpi^*p^*I) \arrow[u, "\beta"]\\
	\Exal_\cY(\cX, I) \arrow[rr, "E"] \arrow[u, "\alpha"]&& \Exal_\cW(\cZ, p^*I)\arrow[u, "\alpha"]
	\end{tikzcd}\end{equation}
}
The vertical instance of \eqref{eq:sheafy-pullback} is an isomorphism since $\varpi^*$ is fully faithful. This implies that the unnamed arrow in the top right square of \eqref{eq:mess1} is an isomorphism (it is already labeled as such), since the other three maps in the square are.
The commutativity of the top left square uses functoriality of \eqref{eq:sheafy-pullback} in the adjoint functors. The top right square commutes by definition of the canonical map of cotangent complexes. The middle rectangle is \eqref{eq:3-for-1} with $p$ the map $Z\rightarrow X$, $A = q^{-1}\OO_{Y^+_\bullet}$, $B = \OO_{X^+_\bullet}$, and $A' = r^{-1}\OO_{W^+_\bullet}$. We have suppressed various squares commuting the maps $p$ and $\varpi$. 

It remains to check that the bottom square of \eqref{eq:mess1} commutes. This we do by direct computation, using \eqref{eq:3-for-1} to factor the map
\[
\Exal_{q^{-1}\OO_{Y^+_\bullet}}(\OO_{X^+_\bullet}, \varpi^*I)  \rightarrow \Exal_{r^{-1}\OO_{W^+_\bullet}}(\OO_{Z^+_\bullet}, \varpi^*p^*I).
\]
Let $i:\cX \hookrightarrow \cX'$ be an element of $\Exal_{\cY}(\cX, I)$. Then we have a commuting diagram
\[
\begin{tikzcd}[column sep=tiny, row sep=tiny]
Z \arrow[rr] \arrow[dr] \arrow[dd] && Z' \arrow[rr] \arrow[dr] \arrow[dd] && W \arrow[dd] \arrow[dr] \\
&X \arrow[rr, crossing over]  && X' \arrow[rr, crossing over]  && Y' \arrow[dd]\\
\cZ \arrow[rr] \arrow[dr]&& \cZ' \arrow[rr] \arrow[dr] && \cW \arrow[dr]\\
& \cX \arrow[from=uu, crossing over]\arrow[rr] && \cX'\arrow[rr]\arrow[from=uu, crossing over] && \cY
\end{tikzcd}
\]
where the front, bottom, and back squares are fibered (six squares in all). The map $p^{-1}\circ\alpha$ sends $i:\cX \hookrightarrow \cX'$ to the extension
\[
0 \rightarrow p^{-1}\varpi^*I \rightarrow p^{-1}i^{-1}\OO_{X'{}^+_\bullet} \rightarrow p^{-1}\OO_{X^+_\bullet} \rightarrow 0
\]
of $p^{-1}q^{-1}\OO_{Y^+_\bullet}$-algebras, and the map \eqref{eq:changeI} sends this extension to the extension
\begin{equation}\label{eq:mess5}
0 \rightarrow p^*\varpi^*I \rightarrow p^*\varpi^*I\oplus_{z^{-1}\varpi^*I}p^{-1}i^{-1}\OO_{X^+_\bullet} \rightarrow p^{-1}\OO_{X^+_\bullet} \rightarrow 0.
\end{equation}
On the other hand, the map $E$ sends $i:\cX \hookrightarrow \cX'$ to $\cZ \hookrightarrow \cZ'$, which under $\alpha$ corresponds to the extension
\[
0 \rightarrow \varpi^*p^*I \rightarrow i^{-1}\OO_{Z'{}^+_\bullet} \rightarrow \OO_{Z^+_\bullet} \rightarrow 0
\]
of $r^{-1}\OO_{W^+_\bullet}$-algebras. After applying $\varpi^*p^*=p^*\varpi^*$ and the morphisms \eqref{eq:changeA} and \eqref{eq:changeB}, this becomes the extension
\begin{equation}\label{eq:mess6}
0 \rightarrow p^*\varpi^*\rightarrow i^{-1}\OO_{Z'{}^+_\bullet} \times_{\OO_{Z^+_\bullet}}p^{-1}\OO_{X^+_\bullet} \rightarrow p^{-1}\OO_{X^+_\bullet} \rightarrow 0
\end{equation}
of $p^{-1}q^{-1}\OO_{Y^+_\bullet}$-algebras. As in the proof of Lemma \ref{lem:ft-functoriality1}, one can write down a functorial (necessarily iso)morphism between \eqref{eq:mess5} and \eqref{eq:mess6}, and check that it is compatible with restrictions.

\end{proof}

\printbibliography

\end{document}